%% file: master.tex
\documentclass{amsart}

\usepackage{amsmath,amssymb,amsthm}
\usepackage{enumitem}
\usepackage{fixltx2e}
\usepackage{mathtools}
\usepackage{graphicx}
\usepackage{tikz-cd}
\usepackage[textsize=small]{todonotes}
\usepackage{xparse}
\usepackage{xspace}
\usepackage{mathtools}
\usepackage{dsfont}
\usepackage[all]{xy}
\usepackage{color}
\usepackage{verbatim}
\usepackage{mathrsfs}
\usepackage{microtype}
\usepackage{a4wide}

\usepackage{longtable} % for 'longtable' environment
\usepackage{pdflscape} % for 'landscape' environment
\usepackage{array}

\makeatletter
\def\Cline#1#2{\@Cline#1#2\@nil}
\def\@Cline#1-#2#3\@nil{%
	\omit
	\@multicnt#1%
	\advance\@multispan\m@ne
	\ifnum\@multicnt=\@ne\@firstofone{&\omit}\fi
	\@multicnt#2%
	\advance\@multicnt-#1%
	\advance\@multispan\@ne
	\leaders\hrule\@height#3\hfill
	\cr}
\makeatother

\numberwithin{equation}{section}
\mathtoolsset{showonlyrefs=true}

\let\cal\mathcal

\def\Cscr{{\cal C}}

\def\Kscr{{\cal K}}

\def\Pscr{{\cal P}}

\def\Tscr{{\cal T}}

\let\blb\mathbb

\def \ZZ{{\blb Z}}

\def\rad{\operatorname {rad}}

\def\Ext{\operatorname {Ext}}

\def\End{\operatorname {End}}

\def\im{\operatorname {im}}

\def\Ker{\operatorname {ker}}

\def\End{\operatorname {End}}

\DeclareMathOperator{\Aut}{Aut}

\DeclareMathOperator{\soc}{soc}

\theoremstyle{definition}

\newtheorem{lemma}{Lemma}[section]
\newtheorem{proposition}[lemma]{Proposition}
\newtheorem{theorem}[lemma]{Theorem}
\newtheorem{corollary}[lemma]{Corollary}

\newtheorem{example}[lemma]{Example}
\newtheorem{definition}[lemma]{Definition}

\newtheorem{remark}[lemma]{Remark}

\DeclareMathOperator\Hom{Hom}

\newcommand{\modC}{\mathbf{mod}}

\newcommand{\projC}{\mathbf{proj}}

\newcommand{\Z}{\mathbb Z}

\usepackage{xcolor}

\mathchardef\mhyphen="2D

\newcounter{todocounter}
\DeclareDocumentCommand\addreference{g}{\stepcounter{todocounter}\todo[color = blue!30, fancyline]{\thetodocounter. Add reference\IfNoValueF{#1}{: #1}}\xspace}
\DeclareDocumentCommand\checkthis{g}{\stepcounter{todocounter}\todo[color = red!50, fancyline]{\thetodocounter. Check this\IfNoValueF{#1}{: #1}}\xspace}
\DeclareDocumentCommand\fixthis{g}{\stepcounter{todocounter}\todo[color = orange!50, fancyline]{\thetodocounter. Fix this\IfNoValueF{#1}{: #1}}\xspace}
\DeclareDocumentCommand\expand{g}{\stepcounter{todocounter}\todo[color = green!50, fancyline]{\thetodocounter. Expand\IfNoValueF{#1}{: #1}}\xspace}

\title[]{A reduction theorem for $\tau$-rigid modules}
\author{Florian Eisele}
\author{Geoffrey Janssens}
\author{Theo Raedschelders}

\address{(Florian Eisele) \newline Department of Mathematics, City University London,
Northampton Square, London EC1V 0HB, United Kingdom \newline E-mail address: {\tt Florian.Eisele@city.ac.uk}}
\address{(Geoffrey Janssens) \newline Departement Wiskunde, Vrije Universiteit Brussel,
Pleinlaan $2$, B-1050 Elsene, Belgium \newline E-mail address: {\tt geofjans@vub.ac.be}}
\address{(Theo Raedschelders) \newline Departement Wiskunde, Vrije Universiteit Brussel, 
Pleinlaan $2$, B-1050 Elsene, Belgium \newline E-mail address: {\tt traedsch@vub.ac.be}}

\begin{document}

\begin{abstract}
We prove a theorem which gives a bijection between the support $\tau$-tilting modules over a given finite-dimensional algebra $A$ and the support $\tau$-tilting modules over $A/I$, where $I$ is the ideal generated by the intersection of the center of $A$ and the radical of $A$. This bijection is both explicit and well-behaved.
We give various corollaries of this, with a particular focus on blocks of group rings of finite groups. 
In particular we show that there are $\tau$-tilting finite wild blocks with more than one simple module.  
We then go on to classify all support $\tau$-tilting modules for all algebras of dihedral, semidihedral and quaternion type, as defined by Erdmann, which include all tame blocks of group rings. Note that since these algebras are symmetric, this is the same as classifying all basic two-term tilting complexes, and it turns out that a tame block has at most $32$ different basic two-term tilting complexes. We do this by using the aforementioned reduction theorem, which reduces the problem to ten different algebras only depending on the ground field $k$, all of which happen to be string algebras. To deal with these ten algebras we give a combinatorial classification of all $\tau$-rigid modules over (not necessarily symmetric) string algebras.
\end{abstract}

\maketitle

\setcounter{tocdepth}{1}
{%\hypersetup{linkcolor=black}
\tableofcontents
}

\newcommand\blfootnote[1]{%
  \begingroup
  \renewcommand\thefootnote{}\footnote{#1}%
  \addtocounter{footnote}{-1}%
  \endgroup
}

\blfootnote{The first author is supported by the EPSRC, grant EP/M02525X/1, and was, at the beginning of the project presented in this article, supported by the FWO,  project G.0157.12N.} 
\blfootnote{The second and third authors are aspirants at the FWO.}

\section{Introduction}

The theory of support $\tau$-tilting modules, as introduced by Adachi, Iyama and Reiten in~\cite{MR3187626}, is related to, and to some extent generalizes, several classical concepts in the representation theory of finite dimensional algebras.  

On the one hand, it is related to silting theory for triangulated categories, which was introduced by Keller and Vossieck in~\cite{MR976638} and provides a generalization of tilting  theory. Just like tilting objects, silting objects generate the triangulated category they live in, but in contrast to tilting objects they are  allowed to have negative self-extensions. Using Keller's version~\cite{MR1649844} of Rickard's derived Morita theorem, a silting object $S$ in an algebraic triangulated category $\Tscr$ gives rise to an equivalence between $\Tscr$ and the perfect complexes over the derived endomorphism ring $\mathbf{R}\End_{\Tscr}(S)$. This ring is a non-negatively graded DGA, which can however be very hard to present in a reasonable way (see for example~\cite{oppermann2015quivers}). 

On the other hand, $\tau$-tilting theory is related to mutation theory, which has its origins in the Bernstein-Gelfand-Ponomarev reflection functors. The basic idea is to replace an indecomposable summand of a tilting object by a new summand to obtain a new tilting object. This mutation procedure has played an important role in several results concerning Brou\'e's abelian defect group conjecture, see~\cite{MR1880662,okuyama1997some,MR1027750}. However, it is not always possible to replace a summand of a tilting object and get a new tilting object in return, which may be seen as sign that one needs to consider a larger class of objects.  
%a common example being the module $DA$, which usually cannot be obtained by mutation from $A$ ($A$ a finite `dimensional algebra). 
This is why Aihara and Iyama introduced the concept of silting mutation~\cite{MR2927802}, where one observes quite the opposite behavior: any summand of a silting object can be replaced to get (infinitely) many new silting objects, and among all of those possibilities one is distinguished as the ``right mutation'' and another one as the ``left mutation''. So in this setting it is natural to ask whether the action of iterated silting mutation on the set of basic silting objects in $\Kscr^b(\projC_A)$ is transitive (for an explicit reference, see Question 1.1 in~\cite{MR2927802}).
In general this question is hard, but to make it more manageable, one can start by studying not all of the basic silting complexes, but just the two-term ones. These have the benefit of being amenable to  the theory of support $\tau$-tilting modules mentioned above.

 A support $\tau$-tilting module $M$ is a module which satisfies $\Hom_{A}(M,\tau M)=0$ and which has as many non-isomorphic indecomposable summands as it has non-isomorphic simple composition factors. These modules correspond bijectively to two-term silting complexes, and possess a compatible mutation theory as well. Using $\tau$-tilting theory, the computation of all the two-term silting complexes and their mutations is a lot more manageable, and in nice cases, one can deduce from the finiteness of the number of two-term silting complexes, the transitivity of iterated silting mutation.

In this article, we will be concerned with determining all basic two-term silting complexes (or equivalently support $\tau$-tilting modules) for various finite dimensional algebras $A$ defined over an algebraically closed field. To this end, we prove the following very general reduction theorem:
	\begin{theorem}[see Theorem~\ref{maintheorem}]
	\label{intromain}
	For an ideal $I$ which is generated by central elements and contained in the Jacobson radical of $A$, the $g$-vectors of indecomposable $\tau$-rigid (respectively support $\tau$-tilting) modules over $A$ coincide with the  ones for  $A/I$, as do the mutation quivers.  
	\end{theorem}
	For the purpose of this introduction we will call an algebra obtained from $A$ by taking successive central quotients a \emph{good} quotient of $A$. The proof of the theorem is a simple application of a four-term exact sequence 
		$$
		\begin{tikzpicture}[baseline= (a).base]
		\node[scale=.8] (a) at (0,0){
		\begin{tikzcd}
		0 \to \Hom_{\Cscr^b(\projC_A)}(C(\alpha),C(\beta)) \to \Hom_A(P,R) \times \Hom_A(Q,S) \xrightarrow{f_{\alpha,\beta}} \Hom_A(P,S) \to \Hom_{\Kscr^b(\projC_A)}(C(\alpha),C(\beta)[1]) \to 0
		\end{tikzcd}};
		\end{tikzpicture}
		$$
which is constructed in Proposition~\ref{exact}, where $P, Q, R$ and $S$ are projective modules, and $C(\alpha)$ and $C(\beta)$ are two-term complexes with terms $P$ and $Q$, respectively $R$ and $S$. The power of this theorem lies in its generality. For example, as an immediate corollary, we recover a result of \cite{Adachi:2015aa} saying that the mutation quiver and the $g$-vectors of a Brauer graph algebra do not depend on the multiplicities of the exceptional vertices, without having to classify all $\tau$-tilting modules beforehand.

One can often use Theorem~\ref{intromain} to effectively compute all two-term silting complexes over a given algebra. In fact, it turns out that many algebras of interest (for example all special biserial algebras and all algebras of dihedral, semidihedral and quaternion type) have a string algebra as a good quotient. Thus, in Section~\ref{separated}, we give a combinatorial algorithm to determine the indecomposable $\tau$-rigid modules, the support $\tau$-tilting modules and the mutation quiver of a string algebra, provided it is $\tau$-tilting finite (otherwise one still gets a description, but no algorithm for obvious reasons). 

As an application, in Section~\ref{Tame blocks}, we consider blocks of group algebras. Note that because these algebras are symmetric, silting and tilting complexes coincide. We show that all tame blocks are $\tau$-tilting finite, i.e. there are only a finite number of isomorphism classes of $\tau$-tilting modules, and we give non-trivial (i.e. non-local) examples of wild blocks of (in some sense) arbitrary large defect which are $\tau$-tilting finite. 

For tame blocks, there is a list of algebras containing all possible basic algebras of these blocks, which is due to Erdmann~\cite{MR1064107}.
It turns out that all algebras of dihedral, semidihedral and quaternion type (which are the algebras that Erdmann classifies) have a string algebra as good quotient, and we exploit this to determine the $g$-vectors and Hasse quivers of all of them. In particular, we  obtain the following theorem.
	\begin{theorem}[see Theorem~\ref{independence}]
	The $g$-vectors and Hasse quivers for tame blocks of group algebras depend only on the Ext-quiver of their basic algebras.
	\end{theorem}
The actual computation of the $g$-vectors and Hasse quivers, which we present in the form of several tables, has been relegated to Appendix~\ref{tameblocks}.
	
Using a result of Aihara and Mizuno~\cite{MR3362259}, we deduce the following theorem:
\begin{theorem}
	All tilting complexes over an algebra of dihedral, semidihedral or quaternion type can be obtained from $A$ (as a module over itself) by iterated tilting mutation.
\end{theorem}
This implies in particular that if $B$ is another algebra and $X\in \mathcal D^b(A^{\rm op} \otimes B)$ is a two-sided tilting complex, then there exists a sequence of algebras $A=A_0,\ A_1, \ldots, A_n=B$ and two-sided two-term tilting complexes $X_i\in \mathcal D^b(A_{i-1}^{\rm op}\otimes A_{i})$ such that $X\cong X_1\otimes_{A_1}^{L} \cdots \otimes_{A_{n-1}}^L X_{n}$.

\section{Preliminaries}
\label{prelims}
Throughout this paper, $k$ denotes an algebraically closed field of arbitrary characteristic, and $A$ is a basic finite-dimensional $k$-algebra with Jacobson radical $\rad(A)$. The category of finitely generated right $A$-modules is denoted by $\modC_A$ and the subcategory of finitely generated  projective $A$-modules is denoted by $\projC_A$. Let $P_1, \ldots, P_l$ denote the non-isomorphic projective indecomposable $A$-modules. By $\tau$ we denote the Auslander-Reiten translate for $A$. The category of bounded complexes of projective modules is denoted by $\Cscr^b(\projC_A)$. Moreover, $\Kscr^b(\projC_A)$ denotes the corresponding homotopy category and $K_0(\projC_A)$ denotes its Grothendieck group. For any $M \in \modC_A$, $| M |$ is defined as the number of indecomposable direct summands of $M$. We will use the same notation for complexes.

We will now give a short summary of the theory of silting complexes and the theory of support $\tau$-tilting modules introduced in \cite{MR3187626}. 
\subsection{Two-term silting complexes}

\begin{definition}
A complex $C=C^{\bullet} \in \Kscr^b(\projC_A)$ is called \emph{two-term} if $C^i =0$ for all $i \neq 0,-1$.
\end{definition}

\begin{definition}
A complex $C \in \Kscr^b(\projC_A)$ is called
	\begin{enumerate}
	\item \emph{presilting} if $\Hom_{ \Kscr^b(\projC_A)}(C,C[i])=0$ for $i>0$,
	\item \emph{silting} if it is presilting and generates $\Kscr^b(\projC_A)$.
	\end{enumerate}
\end{definition}

It can be shown that a silting complex has exactly $|A|$ summands. 

\begin{remark}
A two-term presilting complex is also known as a \emph{rigid} two-term complex. These terms will be used interchangeably.
\end{remark}

On the set of basic silting complexes, one can define a partial order as follows:

\begin{theorem}\cite[Theorem 2.11]{MR2927802}
\label{order}
For basic silting complexes $C$ and $D$, we write $D \leq C$ if 
	$$
	\Hom_{\Kscr^b(\projC_A)}(C,D[1])=0.
	$$
This defines a partial order on the set of silting complexes. 
\end{theorem}

Let us denote the Hasse quiver of this poset by $H(A)$. Now let $C=D \oplus E$ be a basic silting complex with $D$ indecomposable. Then there is a triangle (in $\mathcal K^b(\projC_A)$)
	$$
	D \xrightarrow{f} E' \to D' \to D[1],
	$$
such that $f$ is a minimal left ${\tt add}\ E$-approximation of $D$. 

\begin{definition}
The left mutation of $C$ with respect to $D$ is defined to be
	$$
	\mu_{D}^-(C)=D' \oplus E.
	$$
The right mutation $\mu_{D}^+(C)$ is defined dually.
\end{definition}

We denote by $Q(A)$ the left mutation quiver of $A$ with vertices corresponding to basic silting complexes,  there being an arrow $C \to C'$ whenever $C'=\mu_{D}^-(C)$ for some indecomposable direct summand $D$ of $C$.

\begin{remark}
For symmetric algebras, silting complexes are in fact tilting complexes, so $Q(A)$ is the mutation quiver of tilting complexes.
\end{remark}

\begin{theorem}\cite[Theorem 2.35]{MR2927802}
\label{compare}
The quivers $H(A)$ and $Q(A)$ are the same.
\end{theorem}

In general, this quiver can be disconnected and has no regularity properties. However, if we restrict our attention to basic two-term silting complexes, then more structure appears. In fact, using the theory of support $\tau$-tilting modules, one can prove the following theorem.

\begin{theorem}\cite[Corollary 3.8]{MR3187626}\label{theorem_bongartz}
Any basic two-term rigid complex $C$ with $| C |=|A|-1$ is a direct summand of exactly two basic two-term silting complexes. Moreover, if two basic two-term silting complexes $C$ and $D$ have $|A|-1$ summands in common, then $C$ is a left or right mutation of $D$. 
\end{theorem}

This means that if we denote by $Q_2(A)$ the full subquiver of $Q(A)$ containing the vertices corresponding to basic two-term silting complexes, then we get an $|A|$-regular graph. With an eye towards explicit calculations, the following properties are very useful.

\begin{proposition}\cite[Corollary 3.10]{MR3187626}\label{prop_connected}
If $Q_2(A)$ has a finite connected component $C$, then $Q_2(A)=C$.
\end{proposition}

In some cases, finiteness of $Q_2(A)$ implies that for every $n$, there are only finitely many $n$-term silting complexes.

\begin{proposition}\cite[Theorem 2.4]{MR3362259}\cite[Proposition 6.9]{Adachi:2015aa}
\label{nterm}
Let $A$ be a symmetric algebra. If for any tilting complex $C$ in the connected component of $Q(A)$ containing $A$, the set of basic two-term $\End_{\Kscr^b(\projC_A)}(C)$-tilting complexes is finite, then for every $n$, the set of basic $n$-term $A$-tilting complexes is finite. 
\end{proposition}

\begin{theorem}\cite[Theorem 3.5]{MR3049676}
\label{connected}
If for every $n$, there are only finitely many isomorphism classes of basic $n$-term silting complexes, then $Q(A)$ is connected, i.e. mutation acts transitively on basic silting complexes.
\end{theorem}

Rigid two-term complexes have a complete numerical invariant.

\begin{theorem}\cite[Theorem 5.5]{MR3187626}
\label{g-vector}
A two-term rigid complex $C$ is uniquely determined by its class $[C] \in K_0(\projC_A)$.
\end{theorem}

Expanding out $[C]$ in terms of the basis $[P_0], \ldots, [P_l]$, we get
	$$
	[C]=\sum_{i=1}^{l} g_i^{C} [P_i].
	$$
The tuple $g^{C}=(g_1^{C},\ldots,g_l^{C})$ is known as the $g$-vector of $C$. So Theorem~\ref{g-vector} says that two-term rigid complexes are uniquely determined by their $g$-vectors.

\subsection{Support $\tau$-tilting modules}

\begin{definition}
A module~$M \in \modC_A$ is  called
	\begin{enumerate}
	\item \emph{$\tau$-rigid} if~$\Hom_A(M,\tau M)=0$,
	\item \emph{$\tau$-tilting} if it is $\tau$-rigid and $|M|=|A|$,
	\item \emph{support $\tau$-tilting} if there is an idempotent $e \in A$ such that $M$ is a $\tau$-tilting $A/(e)$-module.
	\end{enumerate}
\end{definition}

We will often think of a support $\tau$-tilting module as a pair $(M,e\cdot A)$, and say that it is basic if both $M$ and $e\cdot A$ are basic. Also, direct sums are defined componentwise.

For basic support $\tau$-tilting modules, there is again a notion of mutation, and one can then similarly define a left mutation quiver $Q_{\tau}(A)$. Also, there is a partial order on this set giving rise to a Hasse quiver $H_{\tau}(A)$. For details, see~\cite[Section 2.4]{MR3187626}. These quivers are again the same (see \cite[Corollary 2.34]{MR3187626}) and isomorphic to $Q_2(A)$, as the following shows:

\begin{theorem}\cite[Theorem 3.2, Corollary 3.9]{MR3187626}
\label{bijection}
There are mutually inverse functions
$$
\{\text{ basic two-term silting complexes }\}\ \overset{f}{\underset{g}{\larger{\larger{\rightleftarrows}}}}\ \{\text{ basic support } \tau\text{-tilting modules }\}
$$
%	$$
%	\begin{tikzcd}
%	\{\text{basic two-term silting complexes}\} \arrow[yshift=2pt]{r}{f} &\{\text{basic support } \%tau-\text{tilting modules}\} \arrow[yshift=-2pt]{l}{g} 
%	\end{tikzcd}
%	$$
which are defined in the following way: 
	\begin{align}
	f(C) &= {\tt H}^0(C) \\
	g((M,R)) &= (P \oplus R \xrightarrow{(p \ 0)} Q),
	\end{align}
where $P \xrightarrow{p} Q \to M$ is a minimal projective presentation of $M$, and $R$ is the (uniquely determined, up to isomorphism) basic projective module such that $(M, R)$ is a support $\tau$-tilting pair. Moreover, this bijection gives an isomorphism of posets between the left mutation quivers $Q_2(A)$ and $Q_{\tau}(A)$.
\end{theorem}

%\begin{theorem}\cite[Corollary 4.5]{MR3024264}
%\label{theorem:gradable}
%If $A$ is also $\mathbb{Z}$-graded, then all rigid modules (in particular the~$\tau$-rigid modules) are gradable.
%\end{theorem}

\section{Geometry of two-term complexes of projective modules}
\label{geometry}
In this section we construct an exact sequence which will be useful for proving our first main theorem. It also serves to provide an elementary proof of Theorem~\ref{g-vector}. 

For two fixed projective $A$-modules $P$ and $Q$, $\Hom_A(P,Q)$ can be considered as algebraic variety, isomorphic to affine space. The connected algebraic group $G=\Aut_A P \times \Aut_A Q$ acts on $\Hom_A(P,Q)$, in such a way that there is a bijection between the set of isomorphism classes of two-term complexes in $\Cscr^b(\projC_A)$ with terms $P$ and $Q$, and the set of orbits of $G$ in $\Hom_A(P,Q)$. 

When we consider $\alpha \in \Hom_A(P,Q)$ as a complex in $\Cscr^b(\projC_A)$, we will denote it as $C(\alpha)$. The orbit of $\alpha \in \Hom_A(P,Q)$ will be denoted by $G \cdot \alpha$ and its stabilizer by $G_{\alpha}$. The following proposition should be well known, but we do not know of a reference. 
	
	\begin{proposition}
	\label{exact}
	For $\alpha \in \Hom_A(P,Q)$ and $\beta \in \Hom_A(R,S)$ ($P,Q,R,S$ projective $A$-modules), there is an exact sequence 
		$$
		\begin{tikzpicture}[baseline= (a).base]
		\node[scale=.8] (a) at (0,0){
		\begin{tikzcd}
		0 \to \Hom_{\Cscr^b(\projC_A)}(C(\alpha),C(\beta)) \to \Hom_A(P,R) \times \Hom_A(Q,S) \xrightarrow{f_{\alpha,\beta}} \Hom_A(P,S) \xrightarrow{g} \Hom_{\Kscr^b(\projC_A)}(C(\alpha),C(\beta)[1]) \to 0
		\end{tikzcd}};
		\end{tikzpicture}
		$$
	where 
	\begin{equation}
		f_{\alpha,\beta} (X,Y) = Y \circ\alpha - \beta\circ X 
	\end{equation}
	and $g$ is just the natural map (we have $C(\alpha)^0 = P$ and $C(\beta)[1]^0=C(\beta)^{-1}=S$, so we may view an element of $\Hom_A(P,S)$ as a map of chain complexes).
	
	Now, in the situation $\alpha=\beta$, let again $G=\Aut_A(P)\times \Aut_A(Q)$ and 
	\begin{equation}
		\phi_\alpha: G \longrightarrow \Hom_A(P, Q):\ (g_1, g_2) \mapsto g_2\circ \alpha\circ g_1^{-1} 
	\end{equation}
	Then $d(\phi_\alpha)_{e}=f_{\alpha,\alpha}$ (where $e$ denotes the unit element of $G$),  
	and we get 
	\begin{equation}
		\im d (\phi_{\alpha})_e=T_{\alpha}(G \cdot \alpha)
	\end{equation}
	\end{proposition}
	\begin{proof}
	The only place where exactness is not immediately clear is at the $\Hom_A(P,S)$ term. The kernel of $g$ consists of all $\gamma \in \Hom_A(P,S)$ such that $g(\gamma)$ is homotopic to zero. This is exactly the image of $f_{\alpha,\beta}$ (the negative sign in the definition of $f_{\alpha,\beta}$ does not affect the image).
	
	That $f_{\alpha,\alpha}$ can be identified with the differential at the identity of the orbit map is clear. 
	
	Let us consider the last statement. We should first remark that the equality $G_\alpha=\Aut_{\mathcal C^b(\projC_A)}(C(\alpha))$ follows immediately from the definition, and hence $\dim G_\alpha = \dim \End_{\mathcal C^b(\projC_A)}(C(\alpha))$. Since orbits are smooth, and by using~\cite[Theorem 4.3]{MR0396773}, we find
		\begin{align}
		\dim T_{\alpha}(G \cdot \alpha) &=\dim G \cdot \alpha \\
		&=\dim G - \dim G_{\alpha} \\
		&= \dim \End_A(P) \times \End_A(Q) - \dim \End_{\Cscr^b(\projC_A)}(C(\alpha)) \\
		&= \dim \im f_{\alpha},
		\end{align}
	where we also used exactness of~\eqref{exact} in the last equality. By the identification of $f_{\alpha,\alpha}$ and $d(\phi_{\alpha})_e$, we are done.
	\end{proof}
	
%	\begin{lemma}
%	Let $X$ be an algebraic variety, and $Y \subset X$ a locally closed, smooth subvariety. Then $Y$ is open in $X$ if and only if there is some point $y \in Y$ such that the normal space $N_{Y,y}(X)=0$.
%	\end{lemma}
%	\begin{proof}
%	For any $y\in Y$, $\dim T_y Y=\dim Y$, since $Y$ is smooth. Also, $\dim T_y X \geq \dim X$, so if $N_{Y,y}(X)=0$, then 
%		\begin{align}
%		\dim Y &= \dim T_yY \\
%		&= \dim T_yX \\
%		& \geq \dim X,
%		\end{align}
%	so $\dim X=\dim Y$ and hence $Y$ is open in $X$. The other implication is clear.
%	\end{proof}
	
Using this proposition we can easily reprove the following well-known results by Jensen-Su-Zimmermann in the special case of two-term complexes. 
	
	\begin{lemma}\cite[Lemma 4.5]{MR2133687}
	\label{denseorbits}
	A two-term complex $C(\alpha) \in \Kscr^b(\projC_A)$ with terms $P$ and $Q$ is rigid if and only if the orbit $G \cdot \alpha$ is open (and thus dense) in $\Hom_A(P,Q)$.
	\end{lemma}
	\begin{proof}
	Orbits are always locally closed and smooth, so they are open exactly when there is a point  $x \in G \cdot \alpha$ such that $\dim T_x(G \cdot \alpha)=\dim T_x \Hom_A(P,Q)$. By Proposition~\ref{exact}, there is an isomorphism
		$$
		T_{\alpha} \Hom_A(P,Q)/T_{\alpha} (G \cdot \alpha) \xrightarrow{\cong} \Hom_{\Kscr^b(\projC_A)}(C(\alpha),C(\alpha)[1]),
		$$
	and hence the lemma follows. Note that the denseness of $G \cdot \alpha$ follows from the fact that $\Hom_A(P,Q)$ is an affine space, and hence irreducible.
	\end{proof}

We can use this lemma to obtain an alternative proof of Theorem~\ref{g-vector}. Note that a slightly weaker form of this theorem was also obtained in~\cite[Corollary 4.8]{MR2133687}.

	\begin{theorem}
	\label{unique}
	Two-term rigid complexes in $\Kscr^b(\projC_A)$ are determined up to isomorphism by their $g$-vectors.
	\end{theorem}
	\begin{proof}
	We first show that any two-term rigid complex $C$ is isomorphic in $\Kscr^b(\projC_A)$ to a complex with terms having no direct summands in common. So assume $C$ can be represented by a complex
	$$
	0 \to P \xrightarrow{d} Q \to 0,
	$$
which is minimal with respect to the number of direct summands of $P$ and $Q$. This minimality ensures that $\im d \subseteq \rad(Q)$. To now prove that $P$ and $Q$ have no summands in common, it suffices to show that the image of any morphism $f:P \to Q$ is contained in $\rad(Q)$. By rigidity, there exist $h_P \in \End_A(P)$ and $h_Q \in \End_A(Q)$ such that 
	$$
	f=h_Q \circ d + d \circ h_P
	$$
But since $\im d \subseteq \rad(Q)$, also $\im f \subseteq \rad(Q)$.

	Now let $C(\alpha)$ and $C(\beta)$ denote two-term rigid complexes, both with terms $P$ and $Q$. Then by Lemma~\ref{denseorbits}, the orbits $G \cdot \alpha$ and $G \cdot \beta$ are dense in $\Hom_A(P,Q)$, so they intersect and we get an isomorphism $C(\alpha) \cong C(\beta)$ in $\Cscr^b(\projC_A)$. In particular, we find that two-term rigid complexes are uniquely determined by their terms.  
	
	Since we know by the first part of the proof that $P$ and $Q$ do not have any summands in common,  the class $[C(\alpha)] \in K_0(\projC_A)$ already suffices to determine $C(\alpha)$ up to isomorphism, which is exactly what we needed to prove. 
	\end{proof}

\section{Quotients by a centrally generated ideal}
\label{centraltheorem}
Suppose $z \in Z(A)\cap \rad(A)$ is an element such that $z^2=0$ and consider the ideal $I=(z)$ of $A$. By $\bar{P_1}, \ldots, \bar{P_l}$ we denote the projective indecomposable $\bar A=A/I$ modules. Note that since $z \in \rad(A)$, the number of projectives is the same. 
%Choosing a vector space decomposition $A \cong V \oplus I$, 
%there is an isomorphism $$\Hom_{A/I}(\bar{P_i},\bar{P_j}) \cong \frac{e_j A e_i}{e_j I e_i}$$

%	$$
%	\begin{tikzcd}
%	e_jVe_i \oplus e_jIe_i \cong \Hom_A(P_i,P_j)   \arrow[twoheadrightarrow]{r}{\pi} &
%	\Hom_{A/I}(\bar{P_i},\bar{P_j})
%	\end{tikzcd}
%	$$
%such that $e_jIe_i = \ker(\pi)$. 
% 

We know that $\Hom_A(P_i, P_j)$ is isomorphic to $e_jAe_i$, and under this isomorphism the kernel of the natural epimorphism 
\begin{equation}
\Hom_A(P_i, P_j)\twoheadrightarrow \Hom_{A/I}(\bar P_i, \bar P_j):\ \alpha\mapsto \bar\alpha
\end{equation} 
corresponds to $e_jIe_i$.
Since $I=(z)$ we therefore have for any $\alpha \in \Hom_A(P_i, P_j)$
\begin{equation}\label{decomp}
	\bar\alpha=0 \iff \alpha=z\cdot \alpha_0 
		\textrm{ for some $\alpha_0\in\Hom_A(P_i, P_j)$}
\end{equation}

%Since $I=(z)$, any $\bar{\alpha} \in \Hom_{A/I}(\bar{P_i},\bar{P_j})$ thus has a (non-unique) lift 
%	\begin{align}
%	\label{decomp}
%	\alpha=\alpha_0 + z \alpha_1.
%	\end{align}
%with $\alpha_0=0$.

The following is our main reduction theorem, which, despite its simple proof, will turn out to be very powerful in the remainder of this article.

	\begin{theorem}
	\label{maintheorem}
	For an ideal $I \subseteq (Z(A)\cap \rad(A))\cdot A$ of $A$, the $g$-vectors of two-term rigid (respectively silting) complexes for $A$ coincide with the ones for $A/I$, as do the mutation quivers. 
	\end{theorem}
	\begin{proof}
	It suffices to consider $I=(z)$ a principal ideal, with $z \in Z(A)$ such that $z^2=0$. From Proposition~\ref{exact}, we know that for all $\alpha \in \Hom_A(P,Q)$, and $\beta \in \Hom_A(R,S)$ ($P,Q,R,S$ projective $A$-modules) there is a commutative diagram with exact rows:
		$$
		\begin{tikzpicture}[baseline= (a).base]
		\node[scale=.7] (a) at (0,0){
		\begin{tikzcd}
		0 \rar &\Hom_{\Cscr^b(\projC_A)}(C(\alpha),C(\beta)) \rar \dar[rightarrow] & \Hom_A(P,R) \times \Hom_A(Q,S) \rar{f_{\alpha,\beta}} \dar[twoheadrightarrow]{\phi} & \Hom_A(P,S) \rar{g} \dar[twoheadrightarrow]{\psi} & \Hom_{\Kscr^b(\projC_A)}(C(\alpha),C(\beta)[1]) \dar \rar & 0 \\
		0 \rar & \Hom_{\Cscr^b(\projC_{\bar{A}})}(C(\bar{\alpha}),C(\bar{\beta})) \rar & \Hom_{\bar{A}}(\bar{P},\bar{R}) \times \Hom_{\bar{A}}(\bar{Q},\bar{S}) \rar{f_{\bar{\alpha},\bar{\beta}}} & \Hom_{\bar{A}}(\bar{P},\bar{S}) \rar{\bar{g}} & \Hom_{\Kscr^b(\projC_{\bar{A}})}(C(\bar{\alpha}),C(\bar{\beta})[1]) \rar & 0 
		\end{tikzcd}};
		\end{tikzpicture}
		$$ 
	Since $\psi$ is surjective, by commutativity of the rightmost square, so is the rightmost vertical arrow. This ensures that if $\Hom_{\Kscr^b(\projC_A)}(C(\alpha),C(\beta)[1])=0$, also $\Hom_{\Kscr^b(\projC_{\bar{A}})}(C(\bar{\alpha}),C(\bar{\beta})[1])=0$.
	
	The other way round, if $\Hom_{\Kscr^b(\projC_{\bar{A}})}(C(\bar{\alpha}),C(\bar{\beta})[1])=0$, we claim that also $\Hom_{\Kscr^b(\projC_{A})}(C(\alpha),C(\beta)[1])=0$. From the exact sequence, we see that $f_{\bar{\alpha},\bar{\beta}}$ is surjective, and it suffices to prove that $f_{\alpha,\beta}$ is also surjective. Let $\gamma \in \Hom_A(P,S)$ be arbitrary, then there exists an element $(X,Y) \in \Hom_A(P,R) \times \Hom_A(Q,S)$ such that 
		\begin{align}
		\psi(\gamma) &= (f_{\bar{\alpha},\bar{\beta}} \circ \phi)(X,Y) \\
		&= (\psi \circ f_{\alpha,\beta})(X,Y),
		\end{align}
	so $\gamma-f_{\alpha,\beta}(X,Y) \in \Ker \psi$, and therefore, by~\eqref{decomp},
		$$
		\gamma=f_{\alpha,\beta}(X,Y) +z \gamma',
		$$
	for some $\gamma' \in \Hom_A(P,S)$. Using surjectivity of $f_{\bar{\alpha},\bar{\beta}} \circ \phi$ again, there exists $(X',Y') \in \Hom_A(P,R) \times \Hom_A(Q,S)$ such that 
		$$
		\gamma'=f_{\alpha,\beta}(X',Y') +z \gamma''.
		$$
	Thus we find that 
		\begin{align}
		\gamma &= f_{\alpha,\beta}(X,Y) +z \gamma' \\
		&= f_{\alpha,\beta}(X,Y) + z f_{\alpha,\beta}(X',Y') \\
		&= f_{\alpha,\beta}(X+zX',Y+zY'),
		\end{align}
	where we used that $z \in Z(A)$. Thus $f_{\alpha,\beta}$ is surjective.
	
	We conclude that for all $\alpha$ and $\beta$:
	\begin{equation}
	\label{drop}
	\Hom_{\Kscr^b(\projC_A)}(C(\alpha),C(\beta)[1])=0 \iff \Hom_{\Kscr^b(\projC_{\bar{A}})}(C(\bar{\alpha}),C(\bar{\beta})[1])=0.
	\end{equation}
	Since the assignment $C(\alpha)\in \mathcal K^b(\projC_A) \mapsto C(\bar \alpha)\in \mathcal K^b(\projC_{A/I})$ does not change the $g$-vectors, and these uniquely determine a rigid complex by Theorem~\ref{g-vector}, it is bijective. For the same reason it preserves and reflects direct sums, which means that the aforementioned assignment induces a bijection between the two-term rigid complexes for $A$ and the two-term rigid complexes for $A/I$ which restricts to a bijection between the indecomposable complexes, and therefore also between the silting complexes. The mutation quivers will coincide as well since by Theorem~\ref{compare} they coincide with the Hasse quivers of the posets formed by the  two-term silting complexes, and the order is preserved due to \eqref{drop}. 
	\end{proof}

As an immediate corollary of Theorem~\ref{maintheorem}, we obtain~\cite[Theorem B]{Adachi:2013aa} in the symmetric case.

	\begin{corollary}
	For a symmetric algebra $A$ with $\soc(A)\subseteq \rad(A)$, the $g$-vectors of two-term rigid (respectively silting) complexes for $A$ coincide with the ones for $A/\soc(A)$, as do the mutation quivers. 
	\end{corollary}
	\begin{proof}
	We may assume without loss that $A$ is basic. In this case $A/\rad(A)$ is a commutative ring. That implies that
	$a\cdot m = m\cdot a$ for all $m\in A/\rad(A)$ (which we now see as an $A$-$A$-bimodule), and all $a\in A$.
	Since $A$ is symmetric we have $\soc(A) \cong \Hom_k(A/\rad(A), k)$ as an $A$-$A$-bimdule, which implies that
	$a\cdot m = m\cdot a$ for all $m\in \soc(A)$ and all $a\in A$. That is, $\soc(A)\subseteq Z(A)$.
	\end{proof}

Here is another immediate application of the foregoing theorem, which recovers the result of \cite{Adachi:2015aa} saying that the mutation quiver and the $g$-vectors of a Brauer graph algebra do not depend on the multiplicities of the exceptional vertices, without having to classify all $\tau$-tilting modules beforehand. 
\begin{example}\label{example_biserial}
Recall that an algebra $A=kQ/I$ is called \emph{special biserial} if
	\begin{enumerate}
		\item There are at most two arrows emanating from each vertex of $Q$.
		\item There are at most two arrows ending in each vertex of $Q$.
		\item For any path $\alpha_1\cdots \alpha_n\not \in I$ ($n\geq 1$) there is at most one arrow
		$\alpha_0$ in $Q$ such that $\alpha_0\cdot \alpha_1\cdots \alpha_n\not \in I$
		and there is at most one arrow $\alpha_{n+1}$ in $Q$ such that
		$\alpha_1\cdots \alpha_n\cdot \alpha_{n+1} \not \in I$.
	\end{enumerate}
Now suppose $A$ is symmetric special biserial. By the main result of~\cite{MR3406174}, these correspond exactly to the Brauer graph algebras, and using the description of the center of such algebras in~\cite[Proposition 2.1.1]{MR2469411}, Theorem~\ref{maintheorem} allows one to recover the fact that the poset of two-term tilting complexes of a Brauer graph algebra is independent of the multiplicities involved (cf.~\cite[Proposition 6.16]{Adachi:2015aa}), simply because all
Brauer graph algebras with the same Brauer graph but different exceptional multiplicities have the same quotient $A/\rad(Z(A))A$ (the sum of all walks around a vertex in the Brauer graph, with each adjacent edge occuring as a starting point precisely once, is a central element; one then obtains the isomorphism of the quotients fairly easily by checking that all relations involving the exceptional multiplicities become zero modulo the ideal generated by these central elements).
%Now using the result (stated without proof!) in~\cite[p. 25]{Adachi:2015aa}, that the endomorphism algebra of a two-term tilting complex for a Brauer graph algebra is again a Brauer graph algebra, the proof of Theorem 6.18 in loc.cit. goes through word for word, thus obtaining a positive answer to the following conjecture.
%	\begin{conjecture}\cite{Adachi:2015aa}
%	Let $A(\GG)$ be a Brauer graph algebra. Then there is an isomorphism of posets between $H(A(\GG))$ and $H(A(\GG_0))$, where $\GG_0$ is the multiplicity-free Brauer graph corresponding to $\GG$.
%	\end{conjecture} 
\end{example}

\section{String algebras}
\label{separated}

As a consequence of Theorem \ref{maintheorem} the classification of indecomposable $\tau$-rigid modules over an algebra $A$ often reduces to the same problem over a quotient $A/I$, which will typically have a simpler structure than $A$ itself. But of course this quotient still needs to be dealt with. One class of algebras for which one might hope to determine all indecomposable $\tau$-rigid modules are the algebras of radical square zero (see \cite{MR3362257}), but this class is not large enough for our purposes.  
In this section we study the $\tau$-rigid modules of string algebras, which are special biserial algebras (as defined in Example \ref{example_biserial})
with monomial relations. There is a well-known classification of indecomposable modules over these algebras, in terms of combinatorial objects known as ``strings'', which are certain walks around the $\Ext$-quiver of the algebra. All Auslander-Reiten sequences are known as well. Hence it is clear that it should be possible to give a combinatorial description of the $\tau$-rigid modules and support $\tau$-tilting modules in terms of these ``strings''. Note that for symmetric special biserial algebras such a classification exists already (see \cite{Adachi:2015aa}). But, as $A$ being symmetric does not imply that $A/I$ is symmetric as well, it is useful to consider non-symmetric special biserial algebras even if one is merely interested in symmetric algebras $A$. By Remark \ref{remark string is everything} below we may then restrict our attention to string algebras, even if we are interested in arbitrary special biserial algebras. 

\begin{definition}
	Let $Q$ be a finite quiver and and let $I$ be an ideal contained in the $k$-span of all paths of length $\geq 2$. We say that $A = kQ/I$ is a \emph{string algebra} if
	the following conditions are met:
	\begin{enumerate}
		\item There are at most two arrows emanating from each vertex of $Q$.
		\item There are at most two arrows ending in each vertex of $Q$.
		\item $I$ is generated by monomials.
		\item For any path $\alpha_1\cdots \alpha_n\not \in I$ ($n\geq 1$) there is at most one arrow
		$\alpha_0$ in $Q$ such that $\alpha_0\cdot \alpha_1\cdots \alpha_n\not \in I$
		and there is at most one arrow $\alpha_{n+1}$ in $Q$ such that
		$\alpha_1\cdots \alpha_n\cdot \alpha_{n+1} \not \in I$.
	\end{enumerate}
\end{definition}

We will now introduce the combinatorial notions which are needed to classify $\tau$-rigid modules over string algebras. For the most part we use the same terminology as used by Butler and Ringel in \cite{MR876976}, where they classify all (finite dimensional) indecomposable modules over string algebras, as well as all Auslander-Reiten sequences. We will nonetheless make some definitions which are particular to our situation, since we only have the very specific goal of classifying $\tau$-rigid modules in mind. One noteworthy detail on which we deviate from \cite{MR876976} is that, since the convention we use for multiplication in path algebras is the opposite of the one used in \cite{MR876976}, the string module $M(\alpha_1\cdots\alpha_m)$ we define below is going to correspond to the string module $M(\alpha_1^{-1}\cdots \alpha_m^{-1})$ as defined in \cite{MR876976}.    

\begin{definition}
	Let $A = kQ/I$ be a string algebra, and let $Q_1 =\{\alpha_1,\ldots, \alpha_h\}$ denote the set of arrows in $Q$. By $\alpha_i^{-1}$ for $i\in\{1,\ldots,h\}$ we denote formal inverses of the arrows $\alpha_i$.
	\begin{enumerate}
	\item A \emph{string} $C$ is a word $c_1\cdots c_m$, where $c_i \in \{\alpha_1, \alpha_1^{-1},\ldots,\alpha_h,\alpha_h^{-1}\}$ such that $c_i\neq c_{i+1}^{-1}$ for all $i\in\{1,\ldots,m-1\}$ and for every subword $W$ of $C$, $W \notin I$ and $W^{-1} \notin I$. We also ask that if $C$ contains a subword of the form $\alpha_i\cdot \alpha_j^{-1}$, then the target of $\alpha_i$ is equal to the target of $\alpha_j$, and 
	if $C$ contains a subword of the form $\alpha_i^{-1}\cdot \alpha_j$, then the source of $\alpha_i$ is equal to the source of $\alpha_j$.
	Moreover, for each vertex $e$ of $Q$, we define two paths of length zero, one of which is called ``direct'' and one of which is called ``inverse'' (this will make sense in the context of the next point below).
	\item We call a string of length greater than zero \emph{direct} if all $c_i$'s are arrows, and \emph{inverse} if all  $c_i$'s are inverses of arrows. We call a string \emph{directed} if it is either direct or inverse.  
	\item For an arrow $\alpha_i$ we denote by $s(\alpha_i)$ its \emph{source} and by $t(\alpha_i)$ its \emph{target}.
	We define $s(\alpha_i^{-1})=t(\alpha_i)$ and $t(\alpha_i^{-1})=s(\alpha_i)$. We extend this notion to strings by defining $s(c_1\cdots c_m)=s(c_1)$ and $t(c_1\cdots c_m)=t(c_m)$. For a string $C$ of length zero, given by a vertex $e$, we define $s(C)=t(C)=e$.
	\item If $C$ is directed, then we define the corresponding direct string $\bar C$ as follows: if $C$ is direct, then $\bar C := C$, and if $C$ is inverse, then $\bar C := C^{-1}$.
	\item Given a string $C$ of length greater than zero let $C_1\cdots C_l$ be the unique factorization of $C$ such that each $C_i$ is directed of length greater than zero, and for each $i\in\{1,\ldots,l-1\}$ exactly one of the strings $C_i$ and $C_{i+1}$ is direct.
	If $C$ is of length zero we set $l=1$ and define $C_1=C_l:=C$.
	\item We call $C_1$ a \emph{loose end} if $C_1$ is inverse and  $C_1^{-1}\cdot \alpha_i \in I$ for all arrows $\alpha_i$. Similarly, we call $C_l$ a \emph{loose end} if  $C_l$ is direct and  $C_l\cdot \alpha_i\in I$ for all arrows $\alpha_i$. The other constituent factors $C_{2},\cdots, C_{l-1}$ are never considered loose ends.
	\item Assume that $C$ is not of length zero. Then we define a string $_PC$ as follows: if $C_1$ is a loose end, we define $_PC=C_2\cdots C_l$ (or one of the strings of length zero corresponding to $t(C_1)$ if $l=1$). If $C_1$ is not a loose end 
	then there is at most one arrow $\alpha_i$ such that $\alpha_{i}^{-1}\cdot C$ is a string, and we define $_PC:= \alpha_{i}^{-1}\cdot C$ if such an $\alpha_i$ exists, and $_PC := C$ otherwise. In the same vein, if $C_l$ is a loose end we define $C_P := C_1\cdots C_{l-1}$ (or one of the strings of length zero corresponding to $s(C_l)$ if $l=1$). If $C_l$ is not a loose end, then there is at most one arrow $\alpha_i$ such that $C\cdot \alpha_i$ is a string and we define $C_{P} := C\cdot \alpha_i$ if such an $\alpha_i$ exists, and $C_P := C$ otherwise.
	
	Now, if $C$ is of length zero, then $C$ is given by a vertex $e$, and we define 
	$C_P = {_PC}=\alpha_i^{-1}$ for some $\alpha_i$ emanating from $e$, provided such an arrow exists, and $C_P ={ _PC}:=C$ if no such arrow exists (note that we make a choice here, so in order to make $_PC$ and $C_P$ well-defined, we technically have to designate one of the arrows emanating from each vertex as the one to be used).
	
	Unless $l=1$ and $C_1=C_l$ is a loose end, we define ${_PC_P} := (_PC)_P={_P(C_P)}$. If $l=1$ and $C_1=C_l$ is an inverse loose end, then we define $_PC_P := {_P(C_P)}$, and if $C_1=C_l$ is a direct loose end we define ${_PC_P} := (_PC)_P$. 
	
	Note that we always have $_P(C^{-1})_P = ({_PC_P})^{-1}$.
	\item If $C$ has length greater than zero, then we call 
	\begin{equation}
		I_C(0)=s(C_1),\ I_C(1)=t(C_1),\ I_C(2)=t(C_2),\ \ldots,\ I_C(l)=t(C_l)
	\end{equation}
	the \emph{intermediate points} of $C$, and for $1\leq i < l$ we call $C_{i}$ and $C_{i+1}$ the \emph{adjacent 
	directed strings} of the intermediate point $I_C(i)$. We say that $C_1$ is the adjacent directed string of 
	$I_C(0)$ and $C_l$ is the adjacent directed string of $I_C(l)$. We say that $I_C(i)$ is an \emph{upper intermediate point} if 
	$C_{i}$ (if it exists, i.e. if $i>0$) is inverse and $C_{i+1}$ (if it exists) is direct. We call $I_C(i)$ a \emph{lower intermediate point} if $C_i$ (if it exists) is direct and $C_{i+1}$ (if it exists) is inverse. In particular, $I_C(0)$ is an upper (respectively lower) intermediate point if $C_1$ is direct (respectively inverse) and $I_C(l)$ is an upper (respectively lower) intermediate point if $C_l$ is inverse (respectively direct).
	
	If $C$ is of length zero, then it corresponds to a vertex $e$, which we consider an upper intermediate point of 
	$C$. That is, $I_C(0)=e$ is an upper intermediate point (and, by definition, the only intermediate point of $C$), and we say that there are no adjacent directed strings.
	\item Let $C$ and $D$ be two strings. Write $C'={_PC_P}=C'_1\cdots C'_m$ and 
	$D'={_PD_P}=D'_1\cdots D'_n$. We say that \emph{$D$ is $C$-rigid} if the following two conditions are met:
	\begin{enumerate}
	\item For any $i\in\{0,\ldots,m\}$ such that $I_{C'}(i)$ is a lower intermediate point of $C'$ and any $j\in\{0,\ldots,n\}$ such that $I_{D'}(j)$ is an upper intermediate point of $D'$ we have that any direct string $W$ with $s(W)=I_{D'}(j)$  and $t(W)=I_{C'}(i)$ factors as either
	$W= \bar X \cdot W'$, where $X$ is an adjacent string of $I_{D'}(j)$, or as $W=W'\cdot \bar Y$, where $Y$ is an adjacent string of $I_{C'}(i)$.
	\item Assume that there are $i\in\{1,\ldots,m\}$ and $j\in\{1,\ldots,n\}$ such that $I_{C'}(i)=I_{D'}(j)$ and
	$I_{C'}(i)$ and $I_{D'}(j)$ are either both upper intermediate points or they are both lower intermediate points. 
	By replacing, if necessary, $D'$ by $D'^{-1}$ (and, as a consequence, $D'_x$ by $D'^{-1}_{n-x+1}$ for each $x$) and $j$ by $n-j$ we can assume without loss that if $i+1 \leq m$ and $j+1\leq n$, then the strings $C_{i+1}'$ and $D_{j+1}'$ both start with the same 
	arrow or inverse of an arrow, and if $i>0$ and $j>0$ then the strings $C_i'$ and $D_j'$ both end on the same arrow or inverse of an arrow.
	 
	 Define $t(1)=1$ and $t(-1)=0$. For $\sigma \in\{1,-1\}$ we define $e(\sigma)\in \Z _{\geq 0}$  to be maximal with respect to the property that 
	\begin{equation}
	C'_{i+\sigma  x+t(\sigma)}=D'_{j+\sigma x+t(\sigma)}\textrm{ for all $0\leq x < e(\sigma)$}
	\end{equation} 
	whilst at the same time satisfying  $0 \leq i+\sigma e(\sigma)\leq m$ and $0 \leq j+\sigma e(\sigma)\leq n$.
	Now we ask that one of the following holds for at least one of the two choices for $\sigma$:	
	\begin{enumerate}
	\item\label{rim_cond} $E_{C'}(i+\sigma e(\sigma))$ is an upper intermediate point and $i+\sigma (e(\sigma)+1) \in\{-1, m+1\}$ 
	\item\label{fact_cond} $E_{C'}(i+\sigma e(\sigma))$ is an upper intermediate point, the previous condition is not met, $j+\sigma (e(\sigma)+1)\not \in \{-1, n+1\}$ and 
	\begin{equation}
		\bar C'_{i+\sigma e(\sigma)+t(\sigma)}=\bar D'_{j+\sigma e(\sigma)+t(\sigma)} \cdot W
	\end{equation}
	for some direct string $W$ (which, by the maximality of $e(\sigma)$, must have positive length).
	\item\label{rim_cond2} $E_{C'}(i+\sigma e(\sigma))$ is a lower intermediate point and $j+\sigma (e(\sigma)+1) \in\{-1, n+1\}$ 
	\item\label{fact_cond2} $E_{C'}(i+\sigma e(\sigma))$ is a lower intermediate point, the previous condition is not met, $i+\sigma (e(\sigma)+1)\not \in \{-1, m+1\}$ and 
	\begin{equation}
		\bar D'_{j+\sigma e(\sigma)+t(\sigma)}=W\cdot \bar C'_{i+\sigma e(\sigma)+t(\sigma)}
	\end{equation} 
	for some direct string $W$ (which, by the maximality of $e(\sigma)$, must have positive length).
	\end{enumerate}
	\end{enumerate} 
	\item We say that a string $C$ is \emph{rigid} if $C$ is $C$-rigid.
	\item We say that a vertex $e$ lies in the support of a string $C$ if one of the following holds (again $C'={_PC_P}=C_1'\cdots C_m'$):
	\begin{enumerate}
	\item $I_{C'}(i)=e$ for some lower intermediate point $I_{C'}(i)$ of $C'$ with $i\neq 0, m$.
	\item There is a direct string $W$ whose source is an upper intermediate point $I_{C'}(i)$ of $C'$ and whose target is $e$, such that $W$
	does not factor as $W=\bar C_j'\cdot W'$ for any adjacent directed string $C_j'$ of $I_{C'}(i)$.
	\end{enumerate}
%	\item We call a collection $\mathcal X$ of strings rigid if for any two strings $C,D\in\mathcal X$ (possibly equal % to one another) we have that $C$ is $D$-rigid.
	\end{enumerate}
\end{definition}

For a vertex $e$ of $Q$ we denote by $P_e=e\cdot A$ the corresponding projective indecomposable module. Given two vertices $e$ and $f$ we will identify direct strings $C$ with $s(C)=e$ and $t(C)=f$ with the homomorphism from
$P_f$ to $P_e$ induced by left multiplication with $C$ (considered as an element of $A$).

\begin{definition}\label{definition string modules}
	Let $A=kQ/I$ be a string algebra and let $C$ be a string. Decompose $C'={_PC_P}= C'_1\cdots  C'_m$. If $C'$ has length greater than zero we define the \emph{string module} $M(C)$ as follows:
	\begin{enumerate}
	\item If $C'_1$ is direct and $C'_m$ is inverse (note that in this case $m$ is even):
	define $Q(i) := P_{t(C'_{2i-1})}$ and $P(i) := P_{s(C'_{2i-1})}$ for $i\in \{1,\ldots,\frac{m}{2}\}$. Define $P(\frac{m}{2}+1):= P_{t(C'_{m})}$. Define
	\begin{equation}
		Q := \bigoplus_{i=1}^{\frac{m}{2}} Q(i)\quad\textrm{ and }\quad P := \bigoplus_{i=1}^{\frac{m}{2}+1} P(i)
	\end{equation}
	Furthermore, for each $i\in\{1,\ldots, \frac{m}{2}\}$, we denote by $\pi_{Q(i)}$ the projection from $Q$ onto $Q(i)$, and for each $i\in\{1,\ldots, \frac{m}{2}+1\}$ we denote by $\iota_{P(i)}$ the embedding of $P(i)$ into $P$. We define 
	a homomorphism $\psi_C:\ Q \longrightarrow P$ as follows:
	\begin{equation}
		\psi_C = \sum_{i=1}^{\frac{m}{2}} \iota_{P(i)}\circ C'_{2i-1}\circ \pi_{Q(i)} + \iota_{P(i+1)} \circ C'^{-1}_{2i} \circ \pi_{Q(i)}
	\end{equation} 
	\item If $C'_1$ is inverse and $C'_m$ is direct (in this case $m$ is even):
	define $Q(i) := P_{s(C'_{2i-1})}$ and $P(i) := P_{t(C'_{2i-1})}$ for $i\in \{1,\ldots,\frac{m}{2}\}$. Define $Q(\frac{m}{2}+1):= P_{t(C'_{m})}$. Define
	\begin{equation}
		Q := \bigoplus_{i=1}^{\frac{m}{2}+1} Q(i)\quad\textrm{ and }\quad P := \bigoplus_{i=1}^{\frac{m}{2}} P(i)
	\end{equation}
	Furthermore, for each $i\in\{1,\ldots, \frac{m}{2}+1\}$, we denote by $\pi_{Q(i)}$ the projection from $Q$ onto $Q(i)$, and for each $i\in\{1,\ldots, \frac{m}{2}\}$ we denote by $\iota_{P(i)}$ the embedding of $P(i)$ into $P$. We define 
	a homomorphism $\psi_C:\ Q \longrightarrow P$ as follows:
	\begin{equation}
		\psi_C = \sum_{i=1}^{\frac{m}{2}} \iota_{P(i)}\circ C'^{-1}_{2i-1}\circ \pi_{Q(i)} + \iota_{P(i)} \circ C'_{2i} \circ \pi_{Q(i+1)}
	\end{equation} 
	\item If $C'_1$ is direct and $C'_m$ is direct (in this case $m$ is odd):
	define $Q(i) := P_{t(C'_{2i-1})}$ and $P(i) := P_{s(C'_{2i-1})}$ for $i\in \{1,\ldots,\frac{m+1}{2}\}$. Define
	\begin{equation}
		Q := \bigoplus_{i=1}^{\frac{m+1}{2}} Q(i)\quad\textrm{ and }\quad P := \bigoplus_{i=1}^{\frac{m+1}{2}} P(i)
	\end{equation}
	Furthermore, for each $i\in\{1,\ldots, \frac{m+1}{2}\}$, we denote by $\pi_{Q(i)}$ the projection from $Q$ onto $Q(i)$, and by $\iota_{P(i)}$ the embedding of $P(i)$ into $P$. We define 
	a homomorphism $\psi_C:\ Q \longrightarrow P$ as follows:
	\begin{equation}
		\psi_C = \iota_{P(\frac{m+1}{2})}\circ C'_m \circ \pi_{Q(\frac{m+1}{2})}+\sum_{i=1}^{\frac{m-1}{2}} \iota_{P(i)}\circ C'_{2i-1}\circ \pi_{Q(i)} + \iota_{P(i+1)} \circ C'^{-1}_{2i} \circ \pi_{Q(i)}
	\end{equation} 
	\end{enumerate}
	In each of the above three cases we define $M(C)$ as the cokernel of $\psi_C$, and we note that $\psi_C$ is a minimal projective presentation of $M(C)$. In case 
	 both $C'_1$ and $C'_m$ are inverse we can define $M(C)$ as $M(C^{-1})$ (this will fall into the ``$C'_1$ and $C'_m$ both direct'' case). Note that $M(C)\cong M(C^{-1})$ holds in the other cases as well.
	 
	 Now, if $C'$ is of length zero, then it is given by a vertex $e$, and we define $M(C) = P_e$.
\end{definition}

Note that the previous definition is much less technical than it looks: given a string $C$, we quite simply form $C'={_PC_P}$, and then define a presentation $Q\longrightarrow P$ such that the indecomposable direct summands of $P$ are in bijection with the upper intermediate points of $C'$ and the indecomposable  direct summands of $Q$ are in bijection with the lower intermediate points of $C'$. The map between $Q$ and $P$ is then simply the sum of the direct versions $\bar C'_1,\ldots, \bar C'_m$ of the factors   $C'_1,\ldots, C'_m$,  each being considered as a map from the summand of $Q$ corresponding to its target to the summand of $P$ corresponding to its source.

\begin{proposition}
	Let $A=kQ/I$ be a string algebra and let $M$ be an indecomposable $\tau$-rigid $A$-module. Then $M$ is a string module.
\end{proposition}
\begin{proof}
	By \cite[Theorem on page 161]{MR876976} each indecomposable $A$-module is either a string module or a so-called band module. By  \cite[Bottom of page 165]{MR876976} each band module occurs in an Auslander-Reiten sequence as both the leftmost and the rightmost term, which means that each band module is isomorphic to its Auslander-Reiten translate. But by definition such a module cannot be $\tau$-rigid.
\end{proof}

\begin{proposition}
	Let $A=kQ/I$ be a string algebra and let $C$ and $D$ be two strings. Denote by 
	$T(C)^\bullet\in \mathcal K^b(\projC_A)$ and $T(D)^\bullet \in \mathcal K^b(\projC_A)$ minimal projective presentations of $M(C)$ respectively $M(D)$. Then
	$\Hom_{\mathcal K^b(\projC_A)}(T(C)^\bullet, T(D)^\bullet[1])=0$ if and only if $D$ is $C$-rigid.
\end{proposition}
\begin{proof}
	We know that a minimal projective presentation of the string module $M(C)$ is given by the following two-term complex:
	\begin{equation}
		T(C)^\bullet = \bigoplus_i Q^{(C)}(i) \stackrel{\psi_C}{\longrightarrow} \bigoplus_j P^{(C)}(j)  
	\end{equation}
	where $i$ ranges over all lower intermediate points of $C'={_PC_P}$ and $j$ ranges over all upper intermediate points of $C'$ (just as in Definition \ref{definition string modules}, we merely added the superscript $(C)$, and are intentionally less explicit about the range of the direct sum in order to avoid having to deal with three different cases again). In the same vein we have the minimal projective presentation
	\begin{equation}
		T(D)^\bullet = \bigoplus_i Q^{(D)}(i) \stackrel{\psi_D}{\longrightarrow} \bigoplus_j P^{(D)}(j)  
	\end{equation}
	of $M(D)$, where $i$ and $j$ range over the lower respectively upper intermediate points of 
	$D'={_PD_P}$.
	We adopt the following notation for homomorphisms: given a direct string $W$ whose source is the upper intermediate point of $D'$ associated with $P^{(D)}(j)$ and whose target is the lower intermediate point of $C'$ associated with $Q^{(C)}(i)$, we denote by $W_{j,i}$ the element of $\Hom_A(Q^{(C)}(i), P^{(D)}(j))$ induced by left multiplication with $W$. Whenever we write $W_{j,i}$ below, we will mean this to tacitly imply that $W$ starts and ends in the right vertices.
	Moreover, we identify 
	\begin{equation}
		\bigoplus_{i,j}\Hom_A(Q^{(C)}(i), P^{(D)}(j)) = \Hom_{\mathcal C^b(A)} (T(C)^\bullet, T(D)^\bullet[1])
	\end{equation}
	Note that the $W_{j,i}$ form a basis of the above vector space, and we will refer to them as \emph{basis elements} in what follows. We say that $W_{j,i}$ is \emph{involved} in an element $\varphi$ of the above space if the coefficient of $W_{j,i}$ is non-zero when we write $\varphi$ as a linear combination of the basis elements. 
	 
	Now the condition $\Hom_{\mathcal K^b(\projC_A)}(T(C)^\bullet, T(D)^\bullet[1])=0$ is equivalent to asking that for each summand $Q^{(C)}(i)$ of $Q^{(C)}$ (that is, for each lower intermediate $i$ point of $C'$) and each summand $P^{(D)}(j)$ of $P^{(D)}$ (that is, for each upper intermediate point $j$ of $D'$) each basis element  $W_{j,i}$ is zero-homotopic. One deduces from the definition of $\psi_C$ and $\psi_D$ that the space of zero-homotopic maps from $T(C)^\bullet$ to $T(D)^\bullet[1]$ is spanned by the following two families of maps:
	\begin{enumerate}
	\item Let $u$ be an upper intermediate point of $D'$ and let $u'$ be an upper intermediate point of $C'$. We define
		\begin{equation}
		\begin{array}{rcll}
		h_C(W,u,u') &=& (W\cdot \bar C'_{u'})_{u,u'-1} + (W\cdot \bar C'_{u'+1})_{u,u'+1} & \textrm{if $0 < u' < m$} \\ \\
		h_C(W,u,u') &=&  (W\cdot \bar C'_{u'+1})_{u,u'+1} & \textrm{if $0 = u' < m$} \\ \\
		h_C(W,u,u') &=&  (W\cdot \bar C'_{u'})_{u,u'-1} & \textrm{if $0 < u' = m$} \\ \\
		\end{array}
		\end{equation}
		if $C'$ has length greater than zero, and $h_C(W,u,u') = 0$ otherwise.
	\item Let $l$ be a lower intermediate point of $D'$ and let $l'$ be a lower intermediate point of $C'$
	\begin{equation} 
	\begin{array}{rcll}
		h_D(W,l,l') &=& (\bar D'_{l}\cdot W)_{l-1,l'}+(\bar D'_{l+1} \cdot W)_{l+1,l'} & \textrm{if $0<l<n$} \\ \\
		h_D(W,l,l') &=& (\bar D'_{l+1} \cdot W)_{l+1,l'} & \textrm{if $0=l<n$} \\ \\
		h_D(W,l,l') &=& (\bar D'_{l}\cdot W)_{l-1,l'}& \textrm{if $0<l=n$} \\ \\
	\end{array}
	\end{equation}
	By definition, $D'$ having a lower intermediate point implies that $D'$ is of length greater than zero, so the length zero case does not need to be considered. 
	\end{enumerate}
	The source and target of the direct string $W$ are $I_{D'}(u)$ and $I_{C'}(u')$ in the first case and 
	$I_{D'}(l)$ and $I_{C'}(l')$ in the second.
	
	The first condition in the definition of $C$-rigidness is fulfilled if and only if each basis element $W_{j,i}$ is involved in some  $h_C(W',u,u')$ or some $h_D(W', l, l')$ for some $W'$. So clearly the first condition is necessary. 
	
	Now notice that if $W$ has positive length, then the unique continuation condition in the definition of string algebras ensures that $h_C(W,u,u')$ respectively $h_D(W,l,l')$ actually involves at most one basis element. Hence every $W_{j,i}$ is zero-homotopic if and only if all basis vectors involved in maps of the form $h_C(I_{C'}(u'),u,u')$ with $I_{C'}(u')=I_{D'}(u)$ and $h_D(I_{C'}(l'),l,l')$ with $I_{C'}(l')=I_{D'}(l)$ are zero-homotopic. So assume that we have such a pair $l',l$ respectively $u',u$. These correspond precisely to the pairs $i,j$ which are considered in the second part of the definition of $C$-rigidness. We may assume that $D'$ is oriented as in the definition of $C$-rigidness, and we get non-negative integers $e(\sigma)$ for $\sigma \in\{1,-1\}$ just as in said definition. 
	For ease of notation we will write $h$ instead of $h_C$ and $h_D$ (the parameters do in fact determine which of the two we are dealing with). So we want to know when the basis elements involved in $h(I_{C'}(i), j, i)$ are zero-homotopic, that is, can be written as a linear combination of other $h$'s. Without loss of generality we can assume that all $h$'s occurring in such a linear combination lie in the equivalence class of $h(I_{C'}(i), j, i)$
	with respect to the transitive closure of the relation $h(W, a, b)\sim h(W',c,d)$ if there is a basis element which is involved in both $h(W, a, b)$ and $h(W',c,d)$. We call $h(W, a, b)$ and $h(W',c,d)$ neighbors of each other. Note that either $c=a+1$ and $d=b+1$, in which case we call $h(W',c,d)$ a right neighbor of $h(W, a, b)$, or $c=a-1$ and $d=b-1$, in which case $h(W',c,d)$ is called a left neighbor of $h(W, a, b)$. Left and right neighbors are unique if they exist.
	For any $-e(-1)< x < e(1)$
	we have 
	\begin{equation}
		h(I_{C'}(i+x), j+x,i+x) = \left\{ \begin{array}{cc}(\bar C'_{i+x})_{j+x,i+x-1} + (\bar C'_{i+x+1})_{j+x,i+x+1}  \\ \textrm{ or } \\ (\bar C'_{i+x})_{j+x-1,i+x} + (\bar C'_{i+x+1})_{j+x+1,i+x} \end{array} \right. 
	\end{equation} 
	depending on whether $I_{C'}(i+x)$ is an upper or a lower intermediate point.
	Hence $h(I_{C'}(i+x), j+x,i+x)$ has exactly two neighbors, namely the right neighbor  $h(I_{C'}(i+x+1), j+x+1,i+x+1)$ and the left neighbor
	$h(I_{C'}(i+x-1), j+x-1,i+x-1)$. It hence suffices to check what the right neighbor of $h(I_{C'}(i+e(1)), j+e(1),i+e(1))$ and the left neighbor of $h(I_{C'}(i-e(-1)), j-e(-1),i-e(-1))$ are.
	
	If $I_{C'}(i+e(1))$ is an upper intermediate point, and $i+e(1)=m$ (i. e. \ref{rim_cond} for $\sigma=+1$ is met), then $h(I_{C'}(i+e(1)), j+e(1),i+e(1))$ has no right neighbors, but in this case $h(I_{C'}(i+e(1)), j+e(1),i+e(1))=(\bar C'_{i+e(1)})_{j+e(1),i+e(1)-1}$ or $h(I_{C'}(i+e(1)), j+e(1),i+e(1)) =0$ (if $C'$ has length zero). If \ref{rim_cond} is met neither for $\sigma=1$ nor for $\sigma=-1$ then $h(I_{C'}(i+e(1)), j+e(1),i+e(1))$ involves  two different basis elements.
	If $i+e(1)<m$, then a right neighbor of $h(I_{C'}(i+e(1)), j+e(1),i+e(1))$ must have the form $h(W, i+e(1)+1, j+e(1)+1)$ where $\bar D'_{j+e(1)+1}\cdot W = \bar C'_{i+e(1)+1}$. That is, a right neighbor exists if and only if the factorization condition \ref{fact_cond} for $\sigma=+1$ is met, and this right neighbor involves just a single basis element. 
	
	Similarly one verifies that if  $I_{C'}(i+e(1))$ is a lower intermediate point, and $j+e(1)=n$ (i. e. \ref{rim_cond2} for $\sigma=+1$ is met), then $h(I_{C'}(i+e(1)), j+e(1),i+e(1))$ has no right neighbors, but in this case $h(I_{C'}(i+e(1)), j+e(1),i+e(1))=(\bar D'_{j+e(1)})_{j+e(1)-1,i+e(1)}$. If \ref{rim_cond2} is met neither for $\sigma=1$ nor for $\sigma=-1$ then $h(I_{C'}(i+e(1)), j+e(1),i+e(1))$ involves  two different basis elements. If \ref{rim_cond2} is not met for $\sigma=+1$, then
	$h(I_{C'}(i+e(1)), j+e(1),i+e(1))$ has a right neighbor (which necessarily involves but a single basis element) if and only \ref{fact_cond2} is met for $\sigma=+1$.
	
	We can of course apply the same line of reasoning to the left neighbors of $h(I_{C'}(i-e(-1)), j-e(-1),i-e(-1)$. What we obtain then is the statement that the equivalence class of $h(I_{C'}(i), j, i)$ with respect to the neighborhood relation contains an element involving only a single basis element if and only if 
	one of the conditions \ref{rim_cond} - \ref{fact_cond2} is met for either $\sigma=+1$ or $\sigma=+1$.
	
	Now one just has to realize that if some element in the equivalence class of $h(I_{C'}(i), j, i)$ involves just a single basis element, then any basis element involved in any element of the equivalence class can be written as a linear combination of the elements of the equivalence class. Conversely, if every element of the equivalence class of $h(I_{C'}(i), j, i)$ involves two basis elements, then no basis element involved in any of the elements of the equivalence class can be written as a linear combination of elements of the equivalence class (note that this is just linear algebra, since such an equivalence class written as row vectors with respect to the basis consisting of all involved basis elements in the right order, looks like $(1,1,0,\ldots,0)$, $(0,1,1,0,\ldots,0)$, \ldots,  $(0,\ldots, 0,1,1)$, and possibly $(1,0,\ldots,0)$ and/or $(0,\ldots,0,1)$). 
\end{proof}

\begin{proposition}
	Let $A=kQ/I$ be a string algebra and let $C$ be a string. Denote by 
	$T(C)^\bullet\in \mathcal K^b(\projC_A)$ a minimal projective presentation of $M(C)$. Let $e$ be a vertex of $Q$, and denote by $P_e^\bullet$ the stalk complex belonging to the projective indecomposable $P_e$. Then we always have
	$\Hom_{\mathcal K^b(\projC_A)}(T(C)^\bullet, P_e^\bullet[2])=0$, and $\Hom_{\mathcal K^b(\projC_A)}(P_e^\bullet[1], T(C)^\bullet[1])=0$  if and only if $e$ is not in the support of $C$.
\end{proposition}
\begin{proof}
	Analogous to the previous proposition.
\end{proof}

\begin{corollary}
	Let $A=kQ/I$ be a string algebra.
	\begin{enumerate}
	\item There are bijections 
	\begin{equation}
		\begin{array}{c}
		\{ \textrm{ rigid strings for $A$ } \} \\ \updownarrow \\ \{\textrm{ indecomposable $\tau$-rigid $A$-modules } \} \\ \updownarrow \\  \{ \textrm{ indecomposable rigid two-term complexes
			$T^\bullet \in \mathcal K^b(\projC_A)$ with ${\tt H}^0(T^\bullet) \neq 0$ } \} 
		\end{array} 
	\end{equation}
	where the first bijection is given by the correspondence between strings and indecomposable $A$-modules, and the second bijection is given by taking a minimal projective presentation of an indecomposable $\tau$-rigid module and, in the other direction, taking homology in degree zero.
	\item \label{point_twoo} If $\{C(1),\ldots, C(l)\}$ is a collection of rigid strings, and $\{e(1), \ldots, e(m)\}$ is a collection of vertices of $Q$, then 
	\begin{equation}
		\left( \bigoplus_{i=1}^l M(C(i)),\ \bigoplus_{j=1}^{m} P_{e(j)} \right)
	\end{equation}
	is a support $\tau$-tilting pair if and only if $l+m =|A|$, each $C(i)$ is $C(j)$-rigid for all $i,j\in \{1,\ldots, l\}$, and none of the $e(j)$'s is in the support of any of the $C(i)$'s.
	\item \label{point_threee} If $\{C(1),\ldots, C(l)\}$, $\{e(1), \ldots, e(m)\}$ and $\{D(1),\ldots, D(l')\}$,$\{f(1),\ldots,f(m') \}$ both give rise to a support $\tau$-tilting module in the sense of the previous point, say $M$ and $N$, then $M \geq N$ if and only if $D(i)$ is $C(j)$-rigid for all $i,j$.
	\end{enumerate}
\end{corollary}

In fact, the preceding corollary shows that there is a combinatorial algorithm to determine the indecomposable $\tau$-rigid modules, the support $\tau$-tilting modules and the mutation quiver of a string algebra, provided $A$ has only finitely many indecomposable $\tau$-rigid modules. We simply run through a list of all strings up to a given length (which we may have to increase if we do not obtain an $|A|$-regular graph as the Hasse quiver below), check which of these strings are rigid and which vertices lie in their support. Then we can determine the support $\tau$-tilting modules involving only those rigid strings using point (\ref{point_twoo}) of the preceding corollary.
We can immediately see which of these support $\tau$-tilting modules are mutations of one another, which gives us 
a subquiver of the Hasse quiver (the direction of the arrows follows from point \ref{point_threee} above). If each vertex in the quiver has exactly $|A|$ neighbors, then we are done by Theorem \ref{theorem_bongartz} and Proposition 
\ref{prop_connected}. Otherwise we need to use a bigger maximal length above and start over. Of course, in practice, this can be done somewhat more efficiently.

\begin{example}[cf. \cite{iyama2016classifying}]
	Consider the quiver
	\begin{equation}
		Q = \xygraph{
			!{<0cm,0cm>;<1.5cm,0cm>:<0cm,1cm>::}
			!{(0,0) }*+{\bullet_{1}}="a"
			!{(1.5,0) }*+{\bullet_{2}}="b"
			!{(3,0) }*+{\cdots}="c"
			!{(4.5,0) }*+{\bullet_{n-1}}="d"
			!{(6,0) }*+{\bullet_n}="e"
			"a" :@/^/^{\alpha_1} "b"
			"b" :@/^/^{\beta_2} "a"
			"b" :@/^/^{\alpha_2} "c"
			"c" :@/^/^{\beta_3} "b"
			"c" :@/^/^{\alpha_{n-2}} "d"
			"d" :@/^/^{\beta_{n-1}} "c"
			"d" :@/^/^{\alpha_{n-1}} "e"
			"e" :@/^/^{\beta_{n}} "d"
		}
	\end{equation}
	and define $A=kQ/I$, where 
	\begin{equation}
		I = \langle \alpha_1\cdot \beta_2, \alpha_i\cdot \beta_{i+1}-\beta_i\cdot \alpha_{i-1}\ \mid\ i=2,\ldots,n-1 \rangle 
	\end{equation}
	Then 
	\begin{equation}
		z = \sum_{i=2}^{n} \beta_{i}\cdot \alpha_{i-1}
	\end{equation}
	is central in $A$.
	We have
	\begin{equation}
		J := \langle I, z\rangle = \langle \alpha_i\cdot \beta_{i+1},\ \beta_{i+1}\cdot \alpha_i \mid i=1,\ldots,n-1\rangle 
	\end{equation}
	So by Theorem \ref{maintheorem} the poset of $2$-term silting complexes over $A$  is isomorphic to  the poset of $2$-term silting complexes over $kQ/J$, which is a string algebra.
	
	In the same vein, the algebra
	\begin{equation}
		B := kQ/\langle I, \beta_n\cdot\alpha_{n-1}\rangle
	\end{equation}
	also has $kQ/J$ as a central quotient, because $z$ obviously remains central modulo $\beta_n\cdot\alpha_{n-1}$.
	
	This shows that $A$, which is the Auslander algebra of $k[x]/(x^n)$, and $B$, which is the preprojective algebra of type $A_n$, have isomorphic posets of $2$-term silting complexes. In fact, Theorem \ref{maintheorem} immediately recovers all of \cite[Theorem 5.3]{iyama2016classifying}. Now by a result of Mizuno~\cite{MR3229959}, the poset of 
	$2$-term silting complexes over $B$ is isomorphic to the group $S_{n+1}$ with the generation order as its poset structure. 
	
	One could in principle try to reprove that last assertion using string combinatorics for the algebra $kQ/J$, but it is not clear whether this would make matters easier. 
	However, what we can easily see is that each string for $kQ/J$ is rigid (and strings can easily be counted in this case), and hence both $A$ and $B$ have exactly $2\cdot (2^n-1) - n$ indecomposable $\tau$-rigid modules.
	Note that for $n=3$, the algebra $kQ/J$ is equal to the algebra $R(3C)$ given in the appendix. $R(3C)$ has, as expected, $24$ support $\tau$-tilting modules and $11$ indecomposable $\tau$-rigid modules, and all (rigid) strings are listed in Figure~\ref{R3Cstring}.
\end{example}

\begin{remark}\label{remark string is everything}
Suppose $A=kQ/I$ is a special biserial algebra and denote by $\Pscr$ a full set of non-isomorphic indecomposable projective-injective non-uniserial $A$-modules. Then it is well known (see for example~\cite{MR717892}) that the quotient algebra
	$$
	B=A/\bigoplus_{P \in \Pscr} \soc(P)
	$$
is a string algebra. By~\cite[Theorem B]{Adachi:2013aa}, the support $\tau$-tilting modules of $A$ can be explicitly computed from those of $B$, so the techniques in this section can be used for arbitrary special biserial algebras. 
\end{remark}

\section{Blocks of group algebras}
\label{Tame blocks}
Now we will apply our Theorem \ref{maintheorem} and the results of the preceding section to blocks of group algebras. Throughout this section, $k$ is an algebraically closed field of characteristic $p$, and $G$ is a finite group. Remember that since (blocks of) $kG$ are symmetric, two-term silting complexes are in fact tilting.

\subsection{$\tau$-tilting finite blocks} Recall that an algebra is called $\tau$-tilting finite if there are only finitely many isomorphism classes of support $\tau$-tilting modules. The following theorem determines the representation type of (blocks of) group algebras.

\begin{theorem}\cite{MR0472984,MR0255704,MR0067896}
Let $B$ be a block of $kG$ and let $P$ be a defect group of $B$. The block algebra $B$ and the group algebra $kP$ have the same representation type. Moreover:
\begin{enumerate}
\item $kP$ is of finite type if $P$ is cyclic.
\item $kP$ is of tame type if $p=2$ and $P$ is the Klein four-group, or a generalized quaternion, dihedral or semi-dihedral group.
\item In all other cases $kP$ is of wild type.
\end{enumerate}
\end{theorem}

\begin{corollary}
There exist $\tau$-tilting finite blocks of group algebras of every representation type and of arbitrary large defect.
\end{corollary}
\begin{proof}
Let $G$ denote a finite $p$-group, so $kG$ is local with defect group $G$. It is easy to see that local algebras are $\tau$-tilting finite, and by the previous theorem they can be of arbitrary representation type.
\end{proof}

A more interesting question is whether there exist non-local blocks of group algebras which are $\tau$-tilting finite but not representation finite. Using Theorem \ref{maintheorem} we can show the following:
\begin{theorem}
	There exist non-local $\tau$-tilting finite wild blocks of group algebras with arbitrary large defect groups, in the sense that every $p$-group occurs as a subgroup of the defect group of a $\tau$-tilting finite non-local block.
\end{theorem}
\begin{proof}
	Assume $B$ is a block of $kG$, with defect group $P$. For $Q$ an arbitrary $p$-group, the algebra $kQ \otimes_k B$ is a block of $k(Q \times G)$ with defect group $Q \times P$ (see for example~\cite[Ch. IV, \S 15, Lemma 6]{MR860771}). Since $Q$ is a $p$-group, there is a non-trivial element $z \in Z(Q)$, and we can form the quotient
	$$
	kQ \otimes_k B/((1-z) \otimes 1) \cong k\bar{Q} \otimes_k B,
	$$
with $\bar{Q}=Q/\langle z \rangle$. Since $\bar{Q}$ is again a $p$-group and $Q$ is finite we can keep repeating this until we get $B$ as a quotient. Now Theorem~\ref{maintheorem} provides a bijection between the support $\tau$-tilting modules for $B$ and the support $\tau$-tilting modules for $kQ \otimes_k B$, so it suffices to take for $B$ a block of cyclic defect or (as we will see below) a tame block, to obtain examples as in the statement of the theorem.
\end{proof}

\subsection{Tame blocks} 
In~\cite{MR1064107}, Erdmann determined the basic algebras of all algebras satisfying the following definition, which is  satisfied in particular by all tame blocks of group algebras.
\begin{definition}
	\label{erdmann2}
	A finite-dimensional algebra $A$ defined over an algebraically closed field $k$ of arbitrary characteristic is of dihedral, semidihedral or quaternion type if it satisfies the following conditions:
	\begin{enumerate}
		\item $A$ is tame, symmetric and indecomposable.
		\item The Cartan matrix of $A$ is non-singular.
		\item The stable Auslander-Reiten quiver of $A$ has the following properties:
		\begin{center}
			\begin{tabular}{ l  l l l }
				& Dihedral type & Semidihedral type & Quaternion type \\ \hline
				Tubes: & rank $1$ and $3$ & rank $\leq 3$ & rank $\leq 2$ \\
				& at most two $3$-tubes & at most one $3$-tube & \\
				Others: & $\ZZ A_{\infty}^{\infty}/\Pi$ & $\ZZ A_{\infty}^{\infty}$ and $\ZZ D_{\infty}$ & \\
				
			\end{tabular}
		\end{center}
	\end{enumerate}
\end{definition}

The following is clear a priori, even without looking at Erdmann's classification~\cite{MR1064107} in greater detail:

\begin{proposition}\label{finite}
A block of a group algebra which is of dihedral, semidihedral or quaternion type is $\tau$-tilting finite.
\end{proposition}
\begin{proof}
The class of algebras defined in Definition \ref{erdmann2} is clearly closed under derived equivalences (cf. \cite[Proposition 2.1]{MR1656577}), and 
it follows from~\cite{MR1064107} that the entries of the Cartan matrices of algebras in the derived equivalence class of an algebra satisfying Definition~\ref{erdmann2} are bounded (to see this one has to use the fact that the dimension of the center is a derived invariant). Now for a two-term tilting complex $T=T_1 \oplus \cdots \oplus T_l$ (remember that $l=|A|$), write
	$$
	T_i: 0 \to \bigoplus_{j=1}^{l} P_j^{\oplus t_{ij}^-} \to \bigoplus_{j=1}^l P_j^{\oplus t_{ij}^+} \to 0,
	$$
and consider $B=\End_{\Kscr^b(\projC_A)}(T)$. Denote by $C_A$ (respectively $C_B$) the Cartan matrix of $A$ (respectively $B$), and denote by $\chi$ the Euler form on $K_0(\projC_A)$. Then:
	\begin{align}
	(C_B)_{m,n} &= \dim_k \Hom_{\Kscr^b(\projC_A)}(T_m,T_n) \\
	&= \sum_{i} (-1)^i \dim_k \Hom_{\Kscr^b(\projC_A)}(T_m,T_n[i]) \\
	&= \chi(T_m,T_n) \\
	&= \chi(\sum_j (t_{mj}^+ - t_{mj}^-)[P_j],\sum_j(t_{nj}^+ - t_{nj}^-)[P_j]) \\
	&= \sum_{i,j} (t_{mi}^+ - t_{mi}^-) \chi([P_i],[P_j]) (t_{nj}^+ - t_{nj}^-) \\
	&= \sum_{i,j} (t_{mi}^+ - t_{mi}^-) (C_A)_{i,j} (t_{nj}^+ - t_{nj}^-), 
	\end{align}
where in the second equality we used that the $T_i$ are tilting. Defining $M \in M_l(\ZZ)$ by $M_{ij}=t_{ij}^+ - t_{ij}^-$, we obtain
	\begin{align}
	\label{cartan}
	C_B=MC_AM^t.
	\end{align}
In fact, there are only finitely many such $M$. To see this suppose that $M' \in M_l(\ZZ)$ also satisfies $C_B=M'C_AM'^{t}$. Then
	\begin{align}
	C_A=M^{-1} M' C_A M'^t (M^{-1})^t,
	\end{align}
which shows that $M^{-1}M' \in O(\ZZ^n,C_A)$, the group of orthogonal, integral matrices preserving $C_A$.
Since the Cartan matrix $C_A$ is positive definite by virtue of $A$ being a block, this is a finite group. 

We have shown that for each of the finitely many Cartan matrices corresponding to algebras derived equivalent to $A$, there are only finitely many matrices $M$ satisfying~\eqref{cartan}. Since the matrix $M$ above is just the matrix of $g$-vectors of the $T_i$, this matrix already determines the tilting complex $T$ by Theorem~\ref{g-vector}, so we are done. 
\end{proof}

\begin{remark}
The proof above works more generally for any symmetric algebra with positive definite Cartan matrix and a derived equivalence class in which the entries of the Cartan matrices are bounded.
\end{remark}

It is known that tame blocks of group algebras satisfy Definition \ref{erdmann2}, and the appendix Erdmann~\cite{MR1064107} furnishes a complete list of basic algebras satisfying said definition. Later, Holm~\cite{MR1461486} showed that non-local tame blocks must actually be of one of the following types:
	$$
	\begin{tabular}{l  l }
	\label{listtame}
	Dihedral: & $D(2A), D(2B), D(3A), D(3B)_1, D(3K)$ \\
	Semidihedral: & $SD(2A)_{1,2}, SD(2B)_{1,2}, SD(3A)_1, SD(3B)_{1,2}, SD(3C), SD(3D), SD(3H)$ \\
	Quaternion: & $Q(2A), Q(2B)_1, Q(3A)_2, Q(3B), Q(3K)$ \\ 
	\end{tabular}
	$$
By Proposition~\ref{finite}, tame blocks are always $\tau$-tilting finite, so it is possible to completely classify the two-term tilting complexes and their associated Hasse quiver. Using Theorem \ref{maintheorem} and the results on string algebras, we are able to achieve this (in fact, for all algebras in Erdmann's list, not just blocks). In Appendix~\ref{tameblocks}, we provide the presentations of the algebras from the appendix of \cite{MR1064107}, along with central elements and the quotients one obtains. A direct application of Theorem~\ref{maintheorem} then reduces the computation of $g$-vectors and Hasse quivers for all tame blocks to the same computation for five explicitly given finite dimensional algebras (with trivial center): $R(2AB), R(3ABD), R(3C), R(3H)$ and $R(3K)$ (see Table \ref{description quotient algebras} below), which do not depend on any extra data. The computation of $g$-vectors and Hasse quivers for all algebras of dihedral, semidihedral or quaternion type reduces to the same computation for the aforementioned five algebras, and in addition
the five algebras $W(2B)$, $W(3ABC)$, $W(Q(3A)_1)$, $W(3F)$ and $W(3QLR)$ (also given in Table \ref{description quotient algebras}).

 In particular we obtain the following theorems.

\begin{theorem}
\label{independence}
The $g$-vectors and Hasse quivers for tame blocks of group algebras depend only on the Ext-quiver of their basic algebras. 
\end{theorem} 

We also find the following more general (but slightly weaker) version of Theorem~\ref{independence}.
\begin{theorem}
	All algebras of dihedral, semidihedral or quaternion type are $\tau$-tilting finite and their $g$-vectors and Hasse quivers are independent of the characteristic of $k$ and the parameters involved in the presentations of their basic algebras.
\end{theorem}

\begin{corollary}
	All tilting complexes for algebras of dihedral, semidihedral, or quaternion type can be obtained from $A$ (as module over itself) by iterated tilting mutation.
\end{corollary}
\begin{proof}
	By~\cite[Proposition 2.1]{MR1656577}, the class of algebras satisfying Definition~\ref{erdmann2} is closed under derived equivalence, so the result follows immediately from Proposition~\ref{nterm} and Theorem~\ref{connected}.
\end{proof}

Since $R(2AB), R(3ABD), R(3C), R(3H)$, $R(3K)$, $W(2B)$, $W(3ABC)$, $W(Q(3A)_1)$, $W(3F)$ and $W(3QLR)$ are all string algebras one can go further and actually compute (using the results in Section~\ref{separated}) the $g$-vectors and Hasse quivers of all algebras of dihedral, semidihedral and quaternion type. For details we refer to Appendix~\ref{tameblocks} below.

\appendix

%\section{Relation to other work}
%\label{otherwork}
%Theorems~\ref{maintheorem,graded} allow us to recover many results already available in the literature, but now with much simpler proofs. In this appendix, we will briefly indicate exactly exactly what statements can be recovered.\checkthis{I don't think we have to provide every detail here, just work it out for ourselves and state the exact results that we can recover.}
%
%\subsection{Brauer graph algebras}
%
%In~\cite{Adachi:2015aa}, a classification of two-term tilting complexes over Brauer graph algebras is provided. 
%
%\subsection{Quotient by socle}
%
%If $A$ is a symmetric algebra, then $\soc(A) \subset \rad(Z(A))$, so by Theorem~\ref{maintheorem}, we immediately obtain~\cite[Theorem B]{Adachi:2013aa} for symmetric algebras.
%	\begin{theorem}
%	The $g$-vectors and Hasse quivers for $A$ and $A/\soc(A)$ coincide.
%	\end{theorem}
%\fixthis{We could probably recover the complete theorem?}

\section{Results for algebras of dihedral, semidihedral and quaternion type}
\label{tameblocks}

In this appendix, we use Theorem~\ref{maintheorem} and the results in Section~\ref{separated} to give a complete description of the $g$-vectors and Hasse quivers of all algebras in Erdmann's list (see Definition~\ref{erdmann2}). These include all the basic algebras of tame blocks of group algebras.

In the first three columns of Table \ref{table Erdmann list}, we give the presentations of these basic algebras, along with their names and parameters. The fourth column describes central elements: note that a horizontal bar above an element means that this element only becomes central in the quotient of the algebra by the central elements above the bar. Taking these successive quotients we obtain the finite dimensional algebras listed in the fifth column, whose presentation can be found in Table \ref{description quotient algebras}. Since these algebras were obtained from the original algebras by taking successive central quotients, Theorem~\ref{maintheorem} ensures that the $g$-vectors and Hasse quiver of the original algebras coincide with the ones of the algebra specified in the fifth column. Using the presentations in Table \ref{description quotient algebras} one checks easily that these are all string algebras. Hence we can use the results of  Section~\ref{separated} to determine all of their support $\tau$-tilting modules using string combinatorics, see Figures 1 - 10. We did the string combinatorics using GAP~\cite{GAP4}, and checked the results against computations done over $k=\mathbb F_2$ using the GAP-package QPA~\cite{QPA}, but it is possible to check the correctness of the results by hand.%\fixthis{Florian, ik kan niet lezen wat je bij deze zin hebt geschreven?}

As for the notation used in Figures 1 - 10: we first give a list of all rigid strings $C$. The notation used for strings should be more or less self-explanatory: for instance, 
\begin{equation}1 \stackrel{\beta}{\longleftarrow} 0 \stackrel{\alpha\beta}{\longrightarrow}1
\end{equation} denotes the string $\beta^{-1}\alpha\beta$. We give a name to each such string (listed in the first column), which is essentially arbitrary (apart from the fact that we use the name $X^\vee$ for a string whose $g$-vector is equal to the $g$-vector of $X$ multiplied by $-1$). For each rigid string $C$ we give $_PC_P$ and its $g$-vector, as well as the names of all other rigid strings which are $C$-rigid. Among these, we highlight those $D$ for which $C$ is $D$-rigid. When $P_i^\vee$ figures among the $C$-rigid strings, it simply means that the corresponding vertex $e_i$ is not in the support of $C$. Below each such table we give the corresponding Hasse quiver, with vertices corresponding to the support $\tau$-tilting modules. We did not draw arrowheads, but the quiver was drawn in such a way that arrows always point downwards.

\begin{landscape}
\renewcommand\arraystretch{1.5}
{\footnotesize {\setlength{\tabcolsep}{2pt} \begin{longtable}{ccccc}
	\Cline{1-5}{1.7pt}
	\textbf{Quiver} & \textbf{Name \& parameters} & \textbf{Relations} & \textbf{Successive central elements} & \textbf{Quotient} \\\Cline{1-5}{1.3pt} \endhead 
	\Cline{1-5}{1.7pt} \caption{Algebras from Erdmann's list together with central quotients} \endfoot
  \label{table Erdmann list}
$
	\xygraph{
		!{<0cm,0cm>;<1.5cm,0cm>:<0cm,1cm>::}
		!{(0,0) }*+{\bullet_{0}}="a"
		!{(1,0) }*+{\bullet_{1}}="b"
		"a" :@(lu,ld)_{\alpha} "a"
		"a" :@/^/^{\beta} "b"
		"b" :@/^/^{\gamma} "a"
	}
	$
	& $\begin{array}{c}D(2A) \\k \geq 1, c= 0 \mbox{ or }1 \end{array}$ 
	& $
	\begin{array}{c}
	\gamma \beta = 0 \\
	\alpha^2 = c(\alpha \beta \gamma)^k\\
	(\alpha \beta \gamma)^k = (\beta \gamma \alpha)^k
	\end{array}
	$
	& $\begin{array}{c}
		(\beta \gamma \alpha)^{k-1}\beta \gamma \\\hline
		\alpha \beta \gamma + \gamma \alpha \beta + \beta \gamma \alpha \\\hline
		\beta \gamma \\
		\end{array}
	$ 
	& $\begin{array}{c}R(2AB)\end{array}$\\\cline{2-5} 	
	& $\begin{array}{c}Q(2A)\\ k\geq 2, c \in K \end{array}$ & 
	$
	\begin{array}{c}
		\beta\gamma\beta=(\alpha\beta\gamma)^{k-1}\alpha \beta \\
		\gamma\beta\gamma=(\gamma \alpha \beta)^{k-1}\gamma \alpha \\
	 	\alpha^2 = (\beta \gamma \alpha)^{k-1} \beta \gamma + c (\beta \gamma \alpha)^k\\
		\alpha^2 \beta = 0 \\
	\end{array}
	$
	&
	$
	\begin{array}{c}
	\alpha^2 \\
	\alpha \beta \gamma + \beta \gamma \alpha + \gamma \alpha \beta \\
	\gamma \beta + (\alpha \beta \gamma)^{k-1} \alpha \\ \hline
	\beta \gamma
	\end{array}
	$ & $R(2AB)$
	\\\cline{2-5}
	& $\begin{array}{c}SD(2A)_1\\ k\geq 2, c \in K \end{array}$
	& $
	\begin{array}{c}
	\alpha^2 = c (\alpha \beta \gamma)^k \\
	\beta \gamma \beta = (\alpha \beta \gamma)^{k-1} \alpha \beta \\
	\gamma \beta \gamma = (\gamma \alpha \beta)^{k-1} \gamma \alpha\\
	(\alpha \beta \gamma)^k\alpha =0
	\end{array}
	$ & $
		\begin{array}{c}
			\alpha^2 \\
	\alpha \beta \gamma + \beta \gamma \alpha + \gamma \alpha \beta \\
	\gamma \beta + (\alpha \beta \gamma)^{k-1} \alpha \\ \hline
	\beta \gamma
		\end{array}
	$ & $R(2AB)$\\\cline{2-5}
	& $\begin{array}{c}SD(2A)_2\\ k\geq 2, c \in K \end{array}$
	& $
	\begin{array}{c}
	\alpha^2 = (\beta \gamma \alpha)^{k-1} \beta \gamma + c (\alpha \beta \gamma)^k \\
	\gamma \beta = 0 \\
	(\alpha \beta \gamma)^k = (\beta \gamma \alpha)^k\\
	\end{array}
	$ & $
		\begin{array}{c}
			(\beta \gamma \alpha)^{k-1}\beta \gamma \\\hline
		\alpha \beta \gamma + \gamma \alpha \beta + \beta \gamma \alpha \\\hline
		\beta \gamma \\
		\end{array}
	$ & $R(2AB)$\\\hline
		$
	\xygraph{
		!{<0cm,0cm>;<1.5cm,0cm>:<0cm,1cm>::}
		!{(0,0) }*+{\bullet_{0}}="a"
		!{(1,0) }*+{\bullet_{1}}="b"
		"a" :@(lu,ld)_{\alpha} "a"
		"a" :@/^/^{\beta} "b"
		"b" :@/^/^{\gamma} "a"
		"b" :@(ru,rd)^{\eta} "b"
	}
	$
	& $\begin{array}{c}D(2B) \\k \geq 1, s \geq 1, c= 0 \mbox{ or }1 \end{array}$ 
	& $
	\begin{array}{c}
	\beta \eta = 0 = \eta \gamma = \gamma \beta \\
	\alpha^2 = c (\alpha \beta \gamma)^k \\
	(\alpha \beta \gamma)^k = (\beta \gamma \alpha)^k \\
	(\gamma \alpha \beta)^k = \eta^s
	\end{array}
	$
	& $\begin{array}{c}
		\eta \\
		(\beta \gamma \alpha)^{k-1}\beta \gamma \\\hline
		\alpha \beta \gamma + \gamma \alpha \beta + \beta \gamma \alpha \\\hline
		\beta \gamma \\
		\end{array}
	$ 
	& $\begin{array}{c}R(2AB)\end{array}$\\\cline{2-5} 	
	& $\begin{array}{c}Q(2B)_1\\ k\geq 2, s \geq 3, c \in K \end{array}$ & 
	$
	\begin{array}{c}
		\gamma \beta = \eta^{s-1} \\
		\beta \eta = (\alpha \beta \gamma)^{k-1} \alpha \beta \\
	\eta \gamma = (\gamma \alpha \beta)^{k-1} \gamma \alpha \\
	\beta \eta = (\alpha \beta \gamma)^{k-1} \alpha \beta \\
	\alpha^2 = (\beta \gamma \alpha)^{k-1} \beta \gamma + c (\beta \gamma \alpha)^k
	\alpha^2 \beta =0
	\end{array}
	$
	&
	$
	\begin{array}{c}
	(\beta \gamma \alpha)^{k-1}\beta \gamma \\\hline
	(\alpha \beta \gamma)^{k-1} \alpha + \eta \\\hline
	\alpha \beta \gamma + \gamma \alpha \beta + \beta \gamma \alpha \\\hline
	\beta \gamma \\
	\end{array}
	$ & $R(2AB)$
	\\
	& $\begin{array}{c}SD(2B)_1\\ k\geq 1, t \geq 2, c \in K \end{array}$
	& $
	\begin{array}{c}
	\gamma \beta = 0 = \eta \gamma = \beta \eta \\
	\alpha^2 = (\beta \gamma \alpha)^{k-1} \beta \gamma + c (\beta \gamma \alpha)^k \\
	\eta^t = (\gamma \alpha \beta)^k \\
	(\alpha \beta \gamma)^k = (\beta \gamma \alpha)^k\\
	\end{array}
	$ & $
		\begin{array}{c}
			\eta \\
		(\beta \gamma \alpha)^{k-1}\beta \gamma \\\hline
		\alpha \beta \gamma + \gamma \alpha \beta + \beta \gamma \alpha \\\hline
		\beta \gamma \\
		\end{array}
	$ & $R(2AB)$\\\cline{2-5}
	& $\begin{array}{c}SD(2B)_2\\ k\geq 1, t \geq 3, c \in K \end{array}$
	& $
	\begin{array}{c}
	\beta \eta = (\alpha \beta \gamma)^{k-1} \alpha \beta \\
	\eta \gamma = (\gamma \alpha \beta)^{k-1} \gamma \alpha \\
	\gamma \beta = \eta^{t-1} \\
	\alpha^2 = c (\alpha \beta \gamma)^k \\
	\beta \eta^2 = 0 = \eta^2 \gamma
	\end{array}
	$ & $
		\begin{array}{c}
			(\beta \gamma \alpha)^{k-1}\beta \gamma \\\hline
	(\alpha \beta \gamma)^{k-1} \alpha + \eta \\\hline
	\alpha \beta \gamma + \gamma \alpha \beta + \beta \gamma \alpha \\\hline
	\beta \gamma \\
		\end{array}
	$ & $R(2AB)$\\\cline{2-5}
	
	& $\begin{array}{c}SD(2B)_4^s(c)\\ s \geq 2, c \in K \end{array}$
	& $
	\begin{array}{c}
	\beta \gamma = \alpha^2, \alpha \beta = \beta \eta, \eta \gamma = \gamma \alpha \\
	\gamma \beta = \eta^2 ( 1 +c \eta^{s+1}), \alpha^s \beta =0 \\
	\gamma \alpha^s = 0= \eta^s \gamma,  \alpha^{s+2} = 0 = \eta^{s+2}\\
	\beta \eta^s = 0 \\
	\end{array}
	$ & $
		\begin{array}{c}
			\beta \gamma + \gamma \beta \\ \hline
		\alpha + \eta
		\end{array}
	$ & $W(2B)$\\\cline{2-5}
	
	& $\begin{array}{c} Q(2B)_2\\ s \geq 4, a \neq 0, p(t) \in K[t], p(0)=1 \end{array}$
	& $
	\begin{array}{c}
	\alpha \beta = \beta \eta, \eta \gamma = \gamma \eta \\
	\beta \gamma = \alpha^2 p(\alpha), \gamma \beta = \eta^2 p(\eta) + a \eta^{s-1} + c \eta^s \\
	\alpha^{s+1} = 0 = \eta^{s+1} \\
	\gamma \alpha^{s-1} = 0 = \alpha^{s-1} \beta \\ 
	\end{array}
	$ & $
		\begin{array}{c}
			\beta \gamma + \gamma \beta \\ \hline
		\alpha + \eta
		\end{array}
	$ & $W(2B)$\\\cline{2-5}
	
	& $\begin{array}{c} Q(2B)_3\\ t \geq 3, a,c,d \in K, a \neq 0 \end{array}$
	& $
	\begin{array}{c}
	\alpha \beta = \beta \eta, \eta \gamma = \gamma \eta \\
	\beta \gamma = \alpha^2 +c \alpha^3, \gamma \beta = a \eta^{t-1} + d \eta^t \\
	\alpha^4 = 0 = \eta^{t+1} = \gamma \alpha^2 = \alpha^2 \beta\\ 
	\end{array}
	$ & $
		\begin{array}{c}
			\beta \gamma + \gamma \beta \\ \hline
		\alpha + \eta
		\end{array}
	$ & $W(2B)$\\\cline{1-5}	
	
	$
	\xygraph{
		!{<0cm,0cm>;<1.5cm,0cm>:<0cm,1cm>::}
		!{(0,0) }*+{\bullet_{1}}="a"
		!{(1,0) }*+{\bullet_{0}}="b"
		!{(2,0) }*+{\bullet_{2}}="c"
		"b" :@/^/^{\gamma} "a"
		"a" :@/^/^{\beta} "b"
		"b" :@/^/^{\delta} "c"
		"c" :@/^/^{\eta} "b"
	}
	$
	& $\begin{array}{c}D(3A)_1 \\k \geq 1\end{array}$ & $
	\begin{array}{c}\beta\gamma=0=\eta\delta \\ 
	(\gamma\beta\delta\eta)^k=(\delta\eta\gamma\beta)^k
	\end{array}
	$
	& $\begin{array}{c}
		\gamma \beta \delta \eta + \beta \delta \eta \gamma + \delta \eta \gamma \beta + \eta \gamma \beta \delta \\\hline
		\gamma \beta \delta \eta
		\end{array}
	$ & $\begin{array}{c}R(3ABD)\end{array}$\\\cline{2-5} 	
	& $\begin{array}{c}Q(3A)_2\\k\geq 2\end{array}$ & 
	$
	\begin{array}{c}
		\beta\gamma\beta=(\beta\delta\eta\gamma)^{k-1}\beta\delta\eta \\
		\gamma\beta\gamma=(\delta\eta\gamma\beta)^{k-1}\delta\eta\gamma \\
		\eta\delta\eta=(\eta\gamma\beta\delta)^{k-1}\eta\gamma\beta\\
		\delta\eta\delta=(\gamma\beta\delta\eta)^{k-1}\gamma\beta\delta \\
		\beta\gamma\beta\delta = 0 = \eta\delta\eta\gamma 
	\end{array}
	$
	&
	$
	\begin{array}{c}
	\gamma \beta \delta \eta + \beta \delta \eta \gamma + \delta \eta \gamma \beta + \eta \gamma \beta \delta \\\hline
	\gamma \beta \delta \eta \\
	\beta\gamma \\ \eta\delta
	\end{array}
	$ & $R(3ABD)$
	\\\cline{2-5}
	& $\begin{array}{c}SD(3A)_1\\ k\geq 1\end{array}$
	& $
	\begin{array}{c}
	\beta\gamma=0\\
	\delta\eta\delta=(\gamma\beta\delta\eta)^{k-1}\gamma\beta\delta \\
	\eta\delta\eta=(\eta\gamma\beta\delta)^{k-1}\eta\gamma\beta
	\end{array}
	$ & $
		\begin{array}{c}
			\gamma \beta \delta \eta + \beta \delta \eta \gamma + \delta \eta \gamma \beta + \eta \gamma \beta \delta \\\hline
			\gamma \beta \delta \eta
			\eta\delta
		\end{array}
	$ & $R(3ABD)$\\\cline{2-5}
	
	& $\begin{array}{c}D(3A)_2\\ l,k \geq 2\end{array}$
	& $
	\begin{array}{c}
	\beta \delta = 0 = \eta \gamma \\
	(\gamma \beta)^k = (\delta \eta)^l\\
	\end{array}
	$ & $
		\begin{array}{c}
			\gamma \beta + \beta \gamma \\
		\delta \eta + \eta \delta \\
		\end{array}
	$ & $W(3ABCD)$\\\cline{2-5}
	
	& $\begin{array}{c}Q(3A)_1\\ a \geq b \geq 2, 0 \neq d \in K \end{array}$
	& $
	\begin{array}{c}
	\beta \delta \eta = (\beta \gamma)^{a-1} \beta \\
	\delta \eta \gamma = (\gamma \beta)^{a-1} \gamma \\
	\eta \gamma \beta = d (\eta \delta)^{b-1} \eta \\
	\gamma \beta \delta = d (\delta \eta)^{b-1} \delta \\
	\beta \delta \eta \delta = 0 = \eta \gamma \beta \gamma\\
	\end{array}
	$ & $
		\begin{array}{c}
			\delta \eta + \eta \delta + (\beta \gamma)^{a-1} \\
			\gamma \beta + \beta \gamma + (\delta \eta)^{b-1} \\
		\end{array}
	$ & $W(Q(3A)_1)$\\\cline{2-5}
	
	& $\begin{array}{c}SD(3A)_2\\ k \geq 2 \end{array}$
	& $
	\begin{array}{c}
	\gamma \beta = \delta \eta \\
	(\beta \gamma)^{k-1} \beta \delta = 0 = (\eta \delta)^{k-1} \eta \gamma
	\end{array}
	$ & $
		\begin{array}{c}
			\beta \gamma + \gamma \beta + \eta \delta
		\end{array}
	$ & $R(3C)$\\\cline{1-5}
	
		$
		\xygraph{
			!{<0cm,0cm>;<1.5cm,0cm>:<0cm,1cm>::}
			!{(0,0) }*+{\bullet_{1}}="a"
			!{(1,0) }*+{\bullet_{0}}="b"
			!{(2,0) }*+{\bullet_{2}}="c"
			"b" :@/^/^{\gamma} "a"
			"a" :@/^/^{\beta} "b"
			"b" :@/^/^{\delta} "c"
			"c" :@/^/^{\eta} "b"
			"a" :@(ld,lu)^{\alpha} "a"
		}
		$ & $\begin{array}{c}D(3B)_1\\k\geq 1,s\geq 2 \end{array}$ &
		$
		\begin{array}{c}
			\alpha\beta=0=\gamma\alpha \\
			\beta\gamma =0=\eta\delta\\
			(\gamma\beta\delta\eta)^k = (\delta\eta\gamma\beta)^k\\
			\alpha^s=(\beta\delta\eta\gamma)^k
		\end{array}
		$ & $\begin{array}{c}\alpha\\
			\gamma \beta \delta \eta + \beta \delta \eta \gamma + \delta \eta \gamma \beta + \eta \gamma \beta \delta \\\hline
			\gamma \beta \delta \eta
		\end{array}$ & $\begin{array}{c}R(3ABD) \end{array}$ 
		\\\cline{2-5}
		& $\begin{array}{c}Q(3B)\\ k\geq 1, s \geq 3\end{array}$ &
		$
		\begin{array}{c}
		\beta\gamma=\alpha^{s-1} \\
		\alpha\beta=(\beta\delta\eta\gamma)^{k-1}\beta\delta\eta \\
		\gamma\alpha=(\delta\eta\gamma\beta)^{k-1}\delta\eta\gamma \\
		\eta\delta\eta=(\eta\gamma\beta\delta)^{k-1}\eta\gamma\beta \\
		\delta\eta\delta=(\gamma\beta\delta\eta)^{k-1}\gamma\beta\delta \\
		\alpha^2\beta=0=\beta\delta\eta\delta
		\end{array}
		$ &
		$
			\begin{array}{c}
				\alpha + (\delta\eta\gamma\beta)^{k-1}\delta\eta \\
				\gamma \beta \delta \eta + \beta \delta \eta \gamma + \delta \eta \gamma \beta + \eta \gamma \beta \delta\\\hline 
				\gamma \beta \delta \eta \\
				\eta\delta
			\end{array}
		$ & $R(3ABD)$ 
		\\\cline{2-5}
		& $\begin{array}{c}SD(3B)_1\\ k\geq 1, s\geq 2\end{array}$ &
		$
		\begin{array}{c}
		\alpha\beta =0 =\gamma\alpha=\beta\gamma\\
		\alpha^s = (\beta\delta\eta\gamma)^k\\
		\eta\delta\eta = (\eta\gamma\beta\delta)^{k-1}\eta\gamma\beta \\
		\delta\eta\delta=(\gamma\beta\delta\eta)^{k-1}\gamma\beta\delta
		\end{array}
		$ & $
		\begin{array}{c}
			\alpha \\	
						\gamma \beta \delta \eta + \beta \delta \eta \gamma + \delta \eta \gamma \beta + \eta \gamma \beta \delta \\\hline
						\gamma \beta \delta \eta \\
						\eta\delta
		\end{array}
		$ & $R(3ABD)$
		\\\cline{2-5}
		& $\begin{array}{c}SD(3B)_2\\ k\geq 1, s\geq 3\end{array}$ &
		$
		\begin{array}{c}
		\eta\delta=0\\
		\beta\gamma=\alpha^{s-1}\\
		\gamma\alpha=(\delta\eta\gamma\beta)^{k-1}\delta\eta\gamma\\
		\alpha\beta=(\beta\delta\eta\gamma)^{k-1}\beta\delta\eta
		\end{array}
		$ & $
			\begin{array}{c}
				\alpha+(\delta\eta\gamma\beta)^{k-1}\delta\eta\\ 
				\gamma \beta \delta \eta + \beta \delta \eta \gamma + \delta \eta \gamma \beta + \eta \gamma \beta \delta \\\hline
				\gamma \beta \delta \eta
			\end{array}
		$ & $R(3ABD)$\\\cline{2-5}
		
      & $\begin{array}{c}D(3B)_2\\ k\geq 1, s\geq 3\end{array}$ &
		$
		\begin{array}{c}
		\gamma \alpha = 0 = \alpha \beta \\
		\beta \delta = 0 = \eta \gamma \\
		(\gamma \beta)^k = (\delta \eta)^l \\
		(\beta \gamma)^k = \alpha^s
		\end{array}
		$ & $
			\begin{array}{c}
				\alpha \\
				\beta \gamma + \gamma \beta \\
				\delta \eta + \eta \delta \\
			\end{array}
		$ & $W(3ABCD)$\\\cline{1-5}

		$
		\xygraph{
			!{<0cm,0cm>;<1.5cm,0cm>:<0cm,1cm>::}
			!{(0,0) }*+{\bullet_{1}}="a"
			!{(1,0) }*+{\bullet_{0}}="b"
			!{(2,0) }*+{\bullet_{2}}="c"
			"b" :@/^/^{\gamma} "a"
			"a" :@/^/^{\beta} "b"
			"b" :@/^/^{\delta} "c"
			"c" :@/^/^{\eta} "b"
			"a" :@(ld,lu)^{\alpha} "a"
			"c" :@(rd,ru)_{\zeta} "c"
		}
		$ & $\begin{array}{c}SD(3D)\\ k\geq 1, s\geq 3, t \geq 2 \end{array}$ 
		& $
		\begin{array}{c}
		\eta\delta=0 = \delta \zeta = \zeta \eta \\
		\beta\gamma=\alpha^{s-1}\\
		\gamma\alpha=(\delta\eta\gamma\beta)^{k-1}\delta\eta\gamma\\
		\alpha\beta=(\beta\delta\eta\gamma)^{k-1}\beta\delta\eta \\
		\zeta^t = (\eta \gamma \beta \delta)^k
		\end{array}
		$
		& $ \begin{array}{c}
				\alpha+(\delta\eta\gamma\beta)^{k-1}\delta\eta\\ 
				\gamma \beta \delta \eta + \beta \delta \eta \gamma + \delta \eta \gamma \beta + \eta \gamma \beta \delta \\
				\zeta \\\hline
				\gamma \beta \delta \eta
			\end{array}$
		& $R(3ABD)$\\\cline{2-5}

		& $\begin{array}{c}D(3D)_1\\ s,t\geq 2, k \geq 1 \end{array}$ 
		& $
		\begin{array}{c}
		\beta \gamma = 0 = \eta \delta, \alpha \beta = 0 = \gamma \alpha \\
		\delta \zeta = 0 = \zeta \eta \\
		\alpha^s = (\beta \delta \eta \gamma)^k \\
		\zeta^t = (\eta \gamma \beta \delta)^k \\
		(\gamma \beta \delta \eta)^k = (\delta \eta \gamma \beta)^k
		\end{array}
		$
		& $ \begin{array}{c}
				\alpha \\
				\zeta \\
				\beta \delta \gamma \eta + \delta \eta \gamma \beta + \eta \gamma \beta \delta + \gamma \beta \delta \eta \\\hline
				\delta \eta \gamma \beta
			\end{array}$
		& $R(3ABD)$\\\cline{2-5}
		
		& $\begin{array}{c}D(3D)_2\\ s,t\geq 2, k ,l\geq 1 \end{array}$ 
		& $
		\begin{array}{c}
		\gamma \alpha = 0 = \alpha \beta, \beta \delta = 0 = \eta \gamma \\
		\delta \zeta = 0= \zeta \eta \\
		(\beta \gamma)^k  = \alpha^s, (\eta \delta)^l = \zeta^t \\
		(\gamma \beta)^k = (\delta \eta)^l
		\end{array}
		$
		& $ \begin{array}{c}
				\alpha \\
				\zeta \\
				\beta \gamma + \gamma \beta \\
				\eta \delta + \delta \eta
			\end{array}$
		& $W(3ABCD)$\\\cline{2-5}
		
		& $\begin{array}{c}Q(3D)\\ s,t\geq 3, k \geq 1 \end{array}$ 
		& $
		\begin{array}{c}
		\beta \gamma = \alpha^{s-1} \\
		\gamma \alpha = (\delta \eta \gamma \beta)^{k-1} \delta \eta \gamma \\
		\alpha \beta = (\beta \delta \eta \gamma)^{k-1} \beta \delta \eta \\
		\eta \delta = \zeta^{t-1} \\
		\delta \zeta = (\gamma \beta \delta \eta)^{k-1} \gamma \beta \delta \\
		\zeta \eta = (\eta \gamma \beta \delta)^{k-1} \eta \gamma \beta \\
		\alpha^2 \beta = 0 = \delta \eta \delta
		\end{array}
		$
		& $ \begin{array}{c}
				\beta \gamma \\
				\eta \delta \\
				\beta \delta \eta \gamma + \delta \eta \gamma \beta + \eta \gamma \beta \delta + \gamma \beta \delta \eta \\\hline
				\alpha \\
				\zeta \\
				\delta \eta \gamma \beta
			\end{array}$
		& $R(3ABD)$\\\cline{1-5}
		
			$
		\xygraph{
			!{<0cm,0cm>;<1.5cm,0cm>:<0cm,1cm>::}
			!{(0,0) }*+{\bullet_{1}}="a"
			!{(1,0) }*+{\bullet_{0}}="b"
			!{(2,0) }*+{\bullet_{2}}="c"
			"b" :@/^/^{\gamma} "a"
			"a" :@/^/^{\beta} "b"
			"b" :@/^/^{\delta} "c"
			"c" :@/^/^{\eta} "b"
			"b" :@(dl,dr)_{\rho} "b"
		}
		$ & $\begin{array}{c}SD(3C)_{2, I/II}\\ k\geq 2, s\geq 2 \end{array}$ 
		& $
		\begin{array}{c}
		\beta \rho = 0 = \rho \delta \\
		\eta \rho = 0 = \rho \gamma \\
		\gamma \beta = \delta \eta\\
		(\beta \gamma)^k = \rho^s \\
		(\beta \gamma)^{k-1} \beta \delta = 0 = (\eta \delta)^{k-1}\eta \gamma 
		\end{array}
		$
		& $ \begin{array}{c}
				\rho\\ 
				\gamma \beta + \beta \gamma + \eta \delta \\
			\end{array}$
		& $R(3C)$\\\cline{2-5}

		& $\begin{array}{c}SD(3C)_{1}\\ s\geq 3 \end{array}$ 
		& $
		\begin{array}{c}
		\beta \delta = 0 = \beta \rho = \rho \gamma \\
		\eta \gamma = 0 = \eta \rho = \rho \delta \\
		\rho^s = \gamma \beta = \delta \eta \\
		\beta \gamma \beta = 0 = \eta \delta \eta
		\end{array}
		$
		& $ \begin{array}{c}
				\rho\\ 
				\gamma \beta + \beta \gamma \\\hline
				\eta \delta\\
			\end{array}$
		& $W(3ABCD)$\\\cline{2-5}
		
		& $\begin{array}{c}Q(3C)\\ k \geq 2, s\geq 3 \end{array}$ 
		& $
		\begin{array}{c}
		\beta \rho = 0 = \rho \gamma \\
		\eta \rho^2 = 0 = \rho^2 \delta \\
		\delta \eta - \gamma \beta = \rho^{s-1} \\
		\eta \rho = (\eta \delta)^{k-1} \eta \\
		\rho \delta = (\delta \eta)^{k-1} \delta \\
		(\beta \gamma)^{k-1} \beta \delta =  0= (\eta \delta)^{k-1}\eta \gamma\\
		\end{array}
		$
		& $ \begin{array}{c}
				\gamma \beta + \delta \eta \\
				\beta \gamma + \delta \eta + \eta \delta \\\hline
				\rho
			\end{array}$
		& $R(3C)$\\\cline{1-5}	
		
		$
		\xygraph{
			!{<0cm,0cm>;<1.5cm,0cm>:<0cm,1cm>::}
			!{(0,0) }*+{\bullet_{0}}="a"
			!{(2,0) }*+{\bullet_{1}}="b"
			!{(1,-2) }*+{\bullet_{2}}="c"
			"b" :@/^/^{\gamma} "a"
			"a" :@/^/^{\beta} "b"
			"b" :@/^/^{\delta} "c"
			"c" :@/^/^{\eta} "b"
			"c" :@/^/^{\lambda} "a"
		}
		$ & $\begin{array}{c}SD(3H)\\ k \geq 2, s \geq 2 \end{array}$ 
		& $
		\begin{array}{c}
		\delta \lambda = (\gamma \beta)^{k-1} \gamma \\
		\lambda \beta = (\eta \delta)^{s-1} \eta \\
		\beta \delta \eta = \gamma \beta \delta = 0 \\
		\eta \gamma =0
		\end{array}
		$
		& $ \begin{array}{c}
				\beta \gamma + \gamma \beta \\
				\eta \delta + \delta \eta
			\end{array}$
		& $R(3H)$\\\cline{1-5}
			$
		\xygraph{
			!{<0cm,0cm>;<1.5cm,0cm>:<0cm,1cm>::}
			!{(0,0) }*+{\bullet_{0}}="a"
			!{(2,0) }*+{\bullet_{1}}="b"
			!{(1,-2) }*+{\bullet_{2}}="c"
			"b" :@/^/^{\gamma} "a"
			"a" :@/^/^{\beta} "b"
			"b" :@/^/^{\delta} "c"
			"c" :@/^/^{\eta} "b"
			"c" :@/^/^{\lambda} "a"
			"a" :@/^/^{\kappa} "c"
		}
		$ & $\begin{array}{c}D(3K)\\ a\geq b \geq c \geq 1 \end{array}$ 
		& $
		\begin{array}{c}
		\beta \delta = \delta \lambda = \lambda \beta = 0 \\
		\gamma \kappa = \kappa \eta = \eta \gamma = 0 \\
		(\beta \gamma)^{a} = (\kappa \lambda)^b \\
		(\lambda \kappa)^b = (\eta \delta)^c\\
		(\delta \eta)^c = (\gamma \beta)^{a}
		\end{array}
		$
		& $ \begin{array}{c}
				\beta \gamma + \gamma \beta \\
		\lambda \kappa + \kappa \lambda \\
		\delta \eta + \eta \delta\\
			\end{array}$
		& $R(3K)$\\\cline{2-5}
		& $\begin{array}{c}Q(3K)\\ a\geq b \geq c \geq 2 \end{array}$ 
		& $
		\begin{array}{c}
		\beta \delta = (\kappa \lambda)^{a-1} \kappa \\
		\eta \gamma = (\lambda \kappa)^{a-1}  \lambda \\
		\delta \lambda = (\gamma \beta)^{b-1} \gamma \\
		\kappa \eta = (\beta \gamma)^{b-1} \beta \\
		\lambda \beta = (\eta \delta)^{c-1} \eta \\
		\gamma \kappa = (\delta \eta)^{c-1} \delta \\
		\gamma \beta \delta = 0 = \delta \eta \gamma = \lambda \kappa \eta 
		\end{array}
		$
		& $ \begin{array}{c}
				\beta \gamma + \gamma \beta \\
		\lambda \kappa + \kappa \lambda \\
		\delta \eta + \eta \delta\\
			\end{array}$
		& $R(3K)$\\\cline{2-5}
		
		& $\begin{array}{c}SD(3K)\\ a\geq b \geq c \geq 2 \end{array}$ 
		& $
		\begin{array}{c}
		\kappa \eta = \eta \gamma = \gamma \kappa = 0 \\
		\delta \lambda = (\gamma \beta)^{k-1} \gamma \\
		\beta \delta = (\kappa \lambda)^{b-1} \kappa \\
		\lambda \beta = (\eta \delta)^{c-1} \eta
		\end{array}
		$
		& $ \begin{array}{c}
				\beta \gamma + \gamma \beta \\
				\delta \eta + \eta \delta \\
				\kappa \lambda + \lambda \kappa \\
				\beta \delta \lambda + \delta \beta \lambda + \lambda \beta \delta
		
			\end{array}$
		& $R(3K)$\\\cline{1-5}
		
			$
		\xygraph{
			!{<0cm,0cm>;<1.5cm,0cm>:<0cm,1cm>::}
			!{(0,0) }*+{\bullet_{0}}="a"
			!{(2,0) }*+{\bullet_{1}}="b"
			!{(1,-2) }*+{\bullet_{2}}="c"
			"a" :@/^/^{\beta} "b"
			"b" :@/^/^{\delta} "c"
			"c" :@/^/^{\eta} "b"
			"c" :@/^/^{\lambda} "a"
		}
		$ & $\begin{array}{c}SD(3F)\\ k \geq 2 \end{array}$ 
		& $
		\begin{array}{c}
		\lambda \beta = (\eta \delta)^{k-1} \eta \\
		\beta \delta \eta = 0 = \eta \delta \lambda\\
		\delta \lambda \beta \delta = 0
		\end{array}
		$
		& $ \begin{array}{c}
		\delta \eta + \eta \delta\\
		\beta \delta \lambda +  \delta \lambda \beta + \lambda \beta \delta
			\end{array}$
		& $W(3F)$\\\cline{1-5}
		
		$
		\xygraph{
			!{<0cm,0cm>;<1.5cm,0cm>:<0cm,1cm>::}
			!{(0,0) }*+{\bullet_{0}}="a"
			!{(2,0) }*+{\bullet_{1}}="b"
			!{(1,-2) }*+{\bullet_{2}}="c"
			"a" :@(ul,dl)_{\alpha} "a"
			"a" :@/^/^{\beta} "b"
			"b" :@/^/^{\delta} "c"
			"c" :@/^/^{\lambda} "a"
			"b" :@(ur,dr)^{\rho} "b"
		}
		$ & $\begin{array}{c}D(3Q)\\ s,t \geq 2, k \geq 1 \end{array}$ 
		& $
		\begin{array}{c}
		\lambda \alpha = 0 = \alpha \beta \\
		\beta \rho = 0 = \rho \delta \\
		(\beta \delta \lambda)^k = \alpha^s \\
		(\delta \lambda \beta)^k = \rho^t
		\end{array}
		$
		& $ \begin{array}{c}
		\alpha\\
		\rho \\
		\beta \delta \lambda +  \delta \lambda \beta + \lambda \beta \delta
			\end{array}$
		& $W(3QLR)$\\\cline{1-5}
		
		$
		\xygraph{
			!{<0cm,0cm>;<1.5cm,0cm>:<0cm,1cm>::}
			!{(0,0) }*+{\bullet_{0}}="a"
			!{(2,0) }*+{\bullet_{1}}="b"
			!{(1,-2) }*+{\bullet_{2}}="c"
			"a" :@(ul,dl)_{\alpha} "a"
			"a" :@/^/^{\beta} "b"
			"b" :@/^/^{\delta} "c"
			"c" :@/^/^{\lambda} "a"
		}
		$ & $\begin{array}{c}D(3L)\\ k,s \geq 2 \end{array}$ 
		& $
		\begin{array}{c}
	\alpha \beta = 0 = \lambda \alpha \\
	(\beta \delta \lambda)^k = \alpha^s \\
	(\delta \lambda \beta)^k \delta = 0
		\end{array}
		$
		& $ \begin{array}{c}
		\alpha\\
		\beta \delta \lambda +  \delta \lambda \beta + \lambda \beta \delta
			\end{array}$
		& $W(3QLR)$\\\cline{1-5}
		
		$
		\xygraph{
			!{<0cm,0cm>;<1.5cm,0cm>:<0cm,1cm>::}
			!{(0,0) }*+{\bullet_{0}}="a"
			!{(2,0) }*+{\bullet_{1}}="b"
			!{(1,-2) }*+{\bullet_{2}}="c"
			"a" :@(ul,dl)_{\alpha} "a"
			"a" :@/^/^{\beta} "b"
			"b" :@/^/^{\delta} "c"
			"c" :@/^/^{\lambda} "a"
			"b" :@(ur,dr)^{\rho} "b"
			"c" :@(dl,dr)_{\zeta} "c"
		}
		$ & $\begin{array}{c}D(3R)\\ s,t,u \geq 2 \end{array}$ 
		& $
		\begin{array}{c}
	\alpha \beta = 0 = \beta \rho = \rho \delta \\
	\delta \zeta = 0= \zeta \lambda = \lambda \alpha \\
	\alpha^s = (\beta \delta \lambda)^k \\
	\rho^t = (\delta \lambda \beta)^k \\
	\zeta^{u} = (\lambda \beta \delta)^k
		\end{array}
		$
		& $ \begin{array}{c}
		\alpha\\
		\rho \\
		\zeta \\
		\beta \delta \lambda +  \delta \lambda \beta + \lambda \beta \delta
			\end{array}$
		& $W(3QLR)$\\\cline{1-5}
\end{longtable}}}

\begin{longtable}{ccccc}
	\Cline{1-5}{1.7pt}
	\textbf{Quiver} & \textbf{Name} & \textbf{Relations} & \textbf{\# of support $\tau$-tilting modules} & \textbf{Figure} \\\Cline{1-5}{1.3pt} \endhead 
	\Cline{1-5}{1.7pt} \caption{ Description of the quotient algebras } \endfoot \label{description quotient algebras}

$
	\xygraph{
		!{<0cm,0cm>;<1.5cm,0cm>:<0cm,1cm>::}
		!{(0,0) }*+{\bullet_{0}}="a"
		!{(1,0) }*+{\bullet_{1}}="b"
		"a" :@(lu,ld)_{\alpha} "a"
		"a" :@/^/^{\beta} "b"
		"b" :@/^/^{\gamma} "a"
	}
	$ & $R(2AB)$ & 
		$\begin{array}{c}
		\beta \gamma = 0 = \gamma \beta \\
		\alpha^2 = 0 \\
		\gamma \alpha \beta=0
		\end{array}$
		& 8
		&  Fig. \ref{R2ABstring} \\\cline{1-5}
	$
	\xygraph{
		!{<0cm,0cm>;<1.5cm,0cm>:<0cm,1cm>::}
		!{(0,0) }*+{\bullet_{0}}="a"
		!{(1,0) }*+{\bullet_{1}}="b"
		"a" :@/^/^{\beta} "b"
		"b" :@/^/^{\gamma} "a"
	}
	$ & $W(2B)$ & 
		$\begin{array}{c}
		\beta \gamma = 0 = \gamma \beta \\
		\end{array}$
		& 6
		& Fig. \ref{W2Bstring} \\\cline{1-5}
		
		$
		\xygraph{
			!{<0cm,0cm>;<1.5cm,0cm>:<0cm,1cm>::}
			!{(0,0) }*+{\bullet_{1}}="a"
			!{(1,0) }*+{\bullet_{0}}="b"
			!{(2,0) }*+{\bullet_{2}}="c"
			"b" :@/^/^{\gamma} "a"
			"a" :@/^/^{\beta} "b"
			"b" :@/^/^{\delta} "c"
			"c" :@/^/^{\eta} "b"
		}
		$ & $R(3ABD)$ & 
		$\begin{array}{c}
		\beta\gamma = 0 = \eta\delta \\
		\beta\delta\eta\gamma=0=\eta\gamma\beta\delta \\
		\gamma\beta\delta\eta=0=\delta\eta\gamma\beta
		\end{array}$
		& $32$
		& Fig. \ref{R3ABDstring}\\\cline{1-5}
	$
		\xygraph{
			!{<0cm,0cm>;<1.5cm,0cm>:<0cm,1cm>::}
			!{(0,0) }*+{\bullet_{1}}="a"
			!{(1,0) }*+{\bullet_{0}}="b"
			!{(2,0) }*+{\bullet_{2}}="c"
			"b" :@/^/^{\gamma} "a"
			"a" :@/^/^{\beta} "b"
			"b" :@/^/^{\delta} "c"
			"c" :@/^/^{\eta} "b"
		}
		$ & $R(3C)$ & 
		$\begin{array}{c}
		\beta\gamma = 0 = \gamma \beta \\
		\eta \delta = 0	= \delta \eta \\
		\end{array}$
		& $24
		$
		& Fig. \ref{R3Cstring}	\\\cline{1-5}
		$
		\xygraph{
			!{<0cm,0cm>;<1.5cm,0cm>:<0cm,1cm>::}
			!{(0,0) }*+{\bullet_{0}}="a"
			!{(2,0) }*+{\bullet_{1}}="b"
			!{(1,-2) }*+{\bullet_{2}}="c"
			"b" :@/^/^{\gamma} "a"
			"a" :@/^/^{\beta} "b"
			"b" :@/^/^{\delta} "c"
			"c" :@/^/^{\eta} "b"
			"c" :@/^/^{\lambda} "a"
		}
		$ & $R(3H)$
		& $
		\begin{array}{c}
		\delta \lambda = 0 = \lambda \beta \\
		\eta \delta = 0 = \delta \eta \\
		\beta \gamma = 0 = \gamma \beta \\
		\eta \gamma=0
		\end{array}
		$
		& $28
		$
		& Fig. \ref{R3Hstring} \\\cline{1-5}
		$
		\xygraph{
			!{<0cm,0cm>;<1.5cm,0cm>:<0cm,1cm>::}
			!{(0,0) }*+{\bullet_{0}}="a"
			!{(2,0) }*+{\bullet_{1}}="b"
			!{(1,-2) }*+{\bullet_{2}}="c"
			"b" :@/^/^{\gamma} "a"
			"a" :@/^/^{\beta} "b"
			"b" :@/^/^{\delta} "c"
			"c" :@/^/^{\eta} "b"
			"c" :@/^/^{\lambda} "a"
			"a" :@/^/^{\kappa} "c"
		}
		$ & $R(3K)$ 
		& $
		\begin{array}{c}
		\beta \delta = \delta \lambda = \lambda \beta = 0 \\
		\gamma \kappa = \kappa \eta = \eta \gamma = 0 \\
		\beta \gamma = 0 = \gamma \beta\\
		\lambda \kappa = 0 = \kappa \lambda \\
		\delta \eta = 0 = \eta \delta\\
		\end{array}
		$
		& $32
		$
		& Fig. \ref{R3Kstring} \\\cline{1-5}	
		
		$
		\xygraph{
			!{<0cm,0cm>;<1.5cm,0cm>:<0cm,1cm>::}
			!{(0,0) }*+{\bullet_{1}}="a"
			!{(1,0) }*+{\bullet_{0}}="b"
			!{(2,0) }*+{\bullet_{2}}="c"
			"b" :@/^/^{\gamma} "a"
			"a" :@/^/^{\beta} "b"
			"b" :@/^/^{\delta} "c"
			"c" :@/^/^{\eta} "b"
		}
		$ & $W(3ABCD)$ & 
		$\begin{array}{c}
		\beta\gamma = 0 = \gamma \beta \\
		\delta \eta = 0 = \eta \delta \\
		\beta \delta = 0 = \eta \gamma\\
		\end{array}$
		& 20
		& Fig. \ref{W3ABCDstring} \\\cline{1-5}
		
		$
		\xygraph{
			!{<0cm,0cm>;<1.5cm,0cm>:<0cm,1cm>::}
			!{(0,0) }*+{\bullet_{1}}="a"
			!{(1,0) }*+{\bullet_{0}}="b"
			!{(2,0) }*+{\bullet_{2}}="c"
			"b" :@/^/^{\gamma} "a"
			"a" :@/^/^{\beta} "b"
			"b" :@/^/^{\delta} "c"
			"c" :@/^/^{\eta} "b"
		}
		$ & $W(Q(3A)_1)$ & 
		$\begin{array}{c}
		\beta\gamma = 0 = \gamma \beta \\
		\delta \eta = 0 = \eta \delta \\
		\beta \delta \eta = 0 = \delta \eta \gamma \\
		\gamma \beta \delta =0 = \eta \gamma \beta
		\end{array}$
		& 24
		& Fig. \ref{WQ3A1string} \\\cline{1-5}
		
		$
		\xygraph{
			!{<0cm,0cm>;<1.5cm,0cm>:<0cm,1cm>::}
			!{(0,0) }*+{\bullet_{0}}="a"
			!{(2,0) }*+{\bullet_{1}}="b"
			!{(1,-2) }*+{\bullet_{2}}="c"
			"a" :@/^/^{\beta} "b"
			"b" :@/^/^{\delta} "c"
			"c" :@/^/^{\eta} "b"
			"c" :@/^/^{\lambda} "a"
		}
		$ & $W(3F)$ & 
		$\begin{array}{c}
		\delta \eta = 0 = \eta \delta \\
		\lambda \beta = 0 \\
		\beta \delta \lambda = 0 = \delta \lambda \beta = \lambda \beta \delta \\
		\end{array}$
		& 24
		& Fig. \ref{W3Fstring} \\\cline{1-5}
		
			$
		\xygraph{
			!{<0cm,0cm>;<1.5cm,0cm>:<0cm,1cm>::}
			!{(0,0) }*+{\bullet_{0}}="a"
			!{(2,0) }*+{\bullet_{1}}="b"
			!{(1,-2) }*+{\bullet_{2}}="c"
			"a" :@/^/^{\beta} "b"
			"b" :@/^/^{\delta} "c"
			"c" :@/^/^{\lambda} "a"
		}
		$ & $W(3QLR)$ & 
		$\begin{array}{c}
		\beta \delta \lambda = 0 = \delta \lambda \beta = \lambda \beta \delta \\
		\end{array}$
		& 20
		& Fig. \ref{W3QLRstring} \\\cline{1-5}
\end{longtable}

\end{landscape}
\input{string_tables}

\begin{figure}[ht]
{\RtwoAB}%checked

{\normalsize Hasse Quiver:}
\begin{equation}
	\resizebox{.25\hsize}{!}{
	\xygraph{
		!{<0cm,0cm>;<1.5cm,0cm>:<0cm,1.5cm>::}
		!{(0,0) }*+{P_0 \oplus P_1 }="1"
		!{(-1,-1) }*+{X^{\vee} \oplus P_1}="2"
		!{(1,-1) }*+{P_0 \oplus Y}="3"
		!{(-1,-2) }*+{X^{\vee} \oplus Y^{\vee}}="4"
		!{(1,-2) }*+{X \oplus Y}="5"
		!{(-1,-3) }*+{Y^{\vee}}="6"
		!{(1,-3) }*+{X}="7"
		!{(0,-4) }*+{0}="8"
		"1" - "2"
		"1" - "3"
		"2" - "4"
		"3" - "5"
		"4" - "6"
		"5" - "7"
		"6" - "8"
		"7" - "8"
	}
	}
\end{equation}
\caption{$R(2AB)$}
\label{R2ABstring}
\end{figure}

\begin{figure}[ht]
{\WtwoB}%checked

{\normalsize Hasse Quiver:}
\begin{equation}
	\resizebox{.3\hsize}{!}{
	\xygraph{
		!{<0cm,0cm>;<1.5cm,0cm>:<0cm,1.5cm>::}
		!{(0,0) }*+{P_0 \oplus P_1 }="1"
		!{(-1,-1) }*+{Y^{\vee} \oplus P_1}="2"
		!{(1,-1) }*+{P_0 \oplus Y}="3"
		!{(-1,-2) }*+{Y^{\vee}}="4"
		!{(1,-2) }*+{Y}="5"
		!{(0,-3) }*+{0}="6"
		"1" - "2"
		"1" - "3"
		"2" - "4"
		"3" - "5"
		"4" - "6"
		"5" - "6"
	}
	}
\end{equation}
\caption{$W(2B)$}
\label{W2Bstring}
\end{figure}

\newpage 
\begin{figure}[ht]%checked
{\RthreeABD}

{\normalsize Hasse Quiver:}
{\large \begin{equation}
	\resizebox{.80\hsize}{!}{
	\xygraph{
		!{<0cm,0cm>;<1.5cm,0cm>:<0cm,1.5cm>::}
		!{(0,0) }*+{P_1\oplus P_0 \oplus P_2}="1"
		!{(-2,-2) }*+{M_2\oplus P_0 \oplus P_2}="2"
		!{(0,-2) }*+{P_1\oplus X^\vee \oplus P_2}="3"
		!{(2,-2) }*+{P_1\oplus P_0 \oplus M_1}="4"
		!{(-4,-4) }*+{M_2\oplus {M_0}^\vee \oplus P_2}="5"
		!{(-2,-4) }*+{T_1^\vee\oplus X^\vee \oplus P_2}="6"
		!{(0,-4) }*+{M_2\oplus P_0 \oplus M_1}="7"
		!{(2,-4) }*+{P_1\oplus X^\vee \oplus T_2^\vee}="8"
		!{(4,-4) }*+{P_1\oplus M_0 \oplus M_1}="9"
		!{(-6,-6) }*+{M_2\oplus {M_0}^\vee \oplus T_2}="10"
		!{(-3,-5) }*+{T_1^\vee \oplus {M_0}^\vee \oplus P_2}="11"
		!{(-2,-6) }*+{T_1^\vee \oplus X^\vee \oplus M_1^\vee}="12"
		!{(0,-6) }*+{M_2\oplus X \oplus M_1}="13"
		!{(2,-6) }*+{M_2^\vee \oplus X^\vee \oplus T_2^\vee}="14"
		!{(3,-5) }*+{P_1\oplus M_0 \oplus T_2^\vee}="15"
		!{(6,-6) }*+{T_1\oplus M_0 \oplus M_1}="16"		
		!{(0,-14) }*+{0}="1d"
		!{(2,-12) }*+{M_2^\vee}="2d"
		!{(0,-12) }*+{X}="3d"
		!{(-2,-12) }*+{M_1^\vee}="4d"
		!{(4,-10) }*+{M_2^\vee\oplus M_0}="5d"
		!{(2,-10) }*+{T_1\oplus X}="6d"
		!{(0,-10) }*+{M_2^\vee\oplus {M_1}^\vee}="7d"
		!{(-2,-10) }*+{X \oplus T_2}="8d"
		!{(-4,-10) }*+{{M_0}^\vee \oplus {M_1}^\vee}="9d"
		!{(6,-8) }*+{{M_2}^\vee\oplus M_0 \oplus T_2^\vee}="10d"
		!{(3,-9) }*+{T_1 \oplus M_0}="11d"
		!{(2,-8) }*+{T_1 \oplus X \oplus M_1}="12d"
		!{(0,-8) }*+{{M_2}^\vee\oplus X^\vee \oplus {M_1}^\vee}="13d"
		!{(-2,-8) }*+{M_2 \oplus X \oplus T_2}="14d"
		!{(-3,-9) }*+{{M_0}^\vee \oplus T_2}="15d"
		!{(-6,-8) }*+{T_1^\vee\oplus {M_0}^\vee \oplus {M_1}^\vee}="16d"
		"1" - "2"
		"1" - "3"
		"1" - "4"
		"2" - "5"
		"2" - "7"
		"3" - "6"
		"4" - "7"
		"3" - "8"
		"4" - "9"
		"5" - "10"
		"5" - "11"
		"6" - "11"
		"6" - "12"
		"7" - "13"
		"8" - "14"
		"8" - "15"
		"9" - "15"
		"9" - "16"
		"1d" - "2d"
		"1d" - "3d"
		"1d" - "4d"
		"2d" - "5d"
		"2d" - "7d"
		"3d" - "6d"
		"4d" - "7d"
		"3d" - "8d"
		"4d" - "9d"
		"5d" - "10d"
		"5d" - "11d"
		"6d" - "11d"
		"6d" - "12d"
		"7d" - "13d"
		"8d" - "14d"
		"8d" - "15d"
		"9d" - "15d"
		"9d" - "16d"
		"10" - "15d"
		"15" - "10d"
		"11" - "16d"
		"16" - "11d"
		"10" - "14d"
		"14" - "10d"
		"16" - "12d"
		"12" - "16d"
		"13" - "12d"
		"12" - "13d"
		"13" - "14d"
		"14" - "13d"  
	}
	}
\end{equation}}
\caption{$R(3ABD)$}
\label{R3ABDstring}
\end{figure}

\newpage

\begin{figure}[ht]%checked
{\RthreeC}

{\normalsize Hasse Quiver:}
\begin{equation}
	\resizebox{1.0\hsize}{!}{
	\xygraph{
		!{<0cm,0cm>;<1.5cm,0cm>:<0cm,1.5cm>::}
		!{(0,0) }*+{P_0 \oplus P_1 \oplus P_2 }="1"
		!{(-2,-2) }*+{M_2 \oplus P_0 \oplus P_2}="2"
		!{(0,-2) }*+{P_1 \oplus X^{\vee} \oplus P_2}="3"
		!{(2,-2) }*+{P_1 \oplus P_0 \oplus M_1}="4"
		!{(-4,-4) }*+{M_2 \oplus M_0^{\vee} \oplus P_2}="5"
		!{(-2,-4)}*+{M_1^{\vee} \oplus X^{\vee} \oplus P_2}="6"
		!{(0,-4) }*+{M_2 \oplus P_0 \oplus M_1}="7"
		!{(2,-4)}*+{P_1 \oplus X^{\vee} \oplus M_2^{\vee}}="8"
		!{(4,-4)}*+{P_1 \oplus M_0 \oplus M_1}="9"
		!{(0,-5)}*+{M_2 \oplus X \oplus M_1}="10"
		!{(-3,-6)}*+{M_1^{\vee} \oplus M_0^{\vee} \oplus P_2}="11"
		!{(-6,-6)}*+{M_2 \oplus M_0^{\vee}}="12"
		!{(6,-6)}*+{M_0 \oplus M_1}="13"
		!{(3,-6)}*+{P_1 \oplus M_0 \oplus M_2^{\vee}}="14"
		!{(0,-7)}*+{M_1^{\vee} \oplus X^{\vee} \oplus M_2^{\vee}}="15"
		!{(-4,-8)}*+{M_1^{\vee} \oplus M_0^{\vee}}="16"
		!{(-2,-8)}*+{M_2 \oplus X}="17"
		!{(0,-8)}*+{M_1^{\vee} \oplus M_2^{\vee}}="18"
		!{(2,-8)}*+{X \oplus M_1}="19"
		!{(4,-8)}*+{M_0 \oplus M_2^{\vee}}="20"
		!{(-2,-10)}*+{M_1^{\vee}}="21"
		!{(0,-10)}*+{X}="22"
		!{(2,-10)}*+{M_2^{\vee}}="23"
		!{(0,-12)}*+{0}="24"
		"1" - "2"
		"1" - "3"
		"1" - "4"
		"2" - "5"
		"2" - "7"
		"3" - "6"
		"3" - "8"
		"4" - "7"
		"4" - "9"
		"5" - "11"
		"5" - "12"
		"6" - "11"
		"6" - "15"
		"7" - "10"
		"8" - "15"
		"8" - "14"
		"9" - "13"
		"9" - "14"
		"10" - "17"
		"10" - "19"
		"11" - "16"
		"12" - "16"
		"12" - "17"
		"13" - "19"
		"13" - "20"
		"14" - "20"
		"15" - "18"
		"16" - "21"
		"17" - "22"
		"18" - "21"
		"18" - "23"
		"19" - "22"
		"20" - "23"
		"21" - "24"
		"22" - "24"
		"23" - "24"
	}
	}
\end{equation}
\caption{$R(3C)$}
\label{R3Cstring}
\end{figure}

\newpage

\begin{figure}[ht]%checked
{\RthreeH}

{\normalsize Hasse Quiver:}
{\large \begin{equation}
	\resizebox{.85\hsize}{!}{
	\xygraph{
		!{<0cm,0cm>;<1.5cm,0cm>:<0cm,1.5cm>::}
		!{(0,0) }*+{P_0 \oplus P_1 \oplus P_2 }="1"
		!{(7,-1) }*+{X^{\vee} \oplus P_1 \oplus P_2}="2"
		!{(-7,-2) }*+{P_0\oplus Y \oplus P_2}="3"
		!{(0,-2) }*+{P_0 \oplus P_1 \oplus M_0}="4"
		!{(7,-2) }*+{X^{\vee} \oplus P_1 \oplus M_2^{\vee}}="6"
		!{(-7,-4)}*+{P_0 \oplus Y \oplus M_2}="8"
		!{(-3,-4) }*+{P_0 \oplus M_1 \oplus M_0}="10"
		!{(3,-4)}*+{M_2^{\vee} \oplus P_1 \oplus M_0}="9"
		!{(-5,-5)}*+{P_0 \oplus M_1 \oplus M_2}="14"
		!{(-2,-6)}*+{M_0^{\vee} \oplus Y \oplus P_2}="7"
		!{(2,-6)}*+{X^{\vee} \oplus M_1^{\vee} \oplus P_2}="5"
		!{(-7,-7)}*+{M_0^{\vee} \oplus Y \oplus M_2}="13"
		!{(0,-7)}*+{M_0^{\vee} \oplus M_1^{\vee} \oplus P_2}="11"
		!{(7,-7)}*+{X^{\vee} \oplus M_1^{\vee} \oplus M_2^{\vee}}="12"
		!{(-7,-9)}*+{M_2^{\vee} \oplus Y^{\vee} \oplus M_0}="15"
		!{(0,-9)}*+{M_1^{\vee} \oplus M_2^{\vee}}="18"
		!{(7,-9)}*+{M_0^{\vee} \oplus M_1^{\vee} \oplus Z^{\vee}}="17"
		!{(-2,-10)}*+{M_2^{\vee} \oplus Y^{\vee}}="22"
		!{(2,-10)}*+{M_1^{\vee} \oplus Z^{\vee}}="23"
		!{(-5,-11)}*+{M_1 \oplus M_0}="16"
		!{(-7,-12)}*+{Y^{\vee} \oplus M_0}="21"
		!{(-3,-12)}*+{M_1 \oplus M_2}="20"
		!{(3,-12)}*+{M_0^{\vee} \oplus M_2}="19"
		!{(-7,-14)}*+{Y^{\vee}}="26"
		!{(0,-14)}*+{M_2}="25"
		!{(7,-14)}*+{M_0^{\vee} \oplus Z^{\vee}}="24"
		!{(7,-15)}*+{Z^{\vee}}="27"
		!{(0,-16)}*+{0}="28"
		"1" - "2"
		"1" - "3"
		"1" - "4"
		"2" - "5"
		"2" - "6"
		"3" - "7"
		"3" - "8"
		"4" - "10"
		"4" - "9"
		"5" - "11"
		"5" - "12"
		"6" - "12"
		"6" - "9"
		"7" - "11"
		"7" - "13"
		"8" - "13"
		"8" - "14"
		"9" - "15"
		"10" - "16"
		"10" - "14"
		"11" - "17"
		"12" - "18"
		"13" - "19"
		"14" - "20"
		"15" - "21"
		"15" - "22"
		"16" - "21"
		"16" - "20"
		"17" - "23"
		"17" - "24"
		"18" - "22"
		"18" - "23"
		"19" - "25"
		"19" - "24"
		"20" - "25"
		"21" - "26"
		"22" - "26"
		"23" - "27"
		"24" - "27"
		"25" - "28"
		"26" - "28"
		"27" - "28"
	}
	}
\end{equation}}
\caption{$R(3H)$}
\label{R3Hstring}
\end{figure}

\newpage

\begin{figure}[ht]%checked
{\renewcommand\tabcolsep{3.5pt}\RthreeK}
{\normalsize Hasse Quiver:}
{\Large \begin{equation}
	\resizebox{1.00\hsize}{!}{
	\xygraph{
		!{<0cm,0cm>;<1.5cm,0cm>:<0cm,1.5cm>::}
		!{(0,0) }*+{P_0 \oplus P_1 \oplus P_2 }="1"
		!{(-3,-2) }*+{X^{\vee} \oplus P_1 \oplus P_2}="2"
		!{(0,-2) }*+{P_0 \oplus Y \oplus P_2}="3"
		!{(3,-2) }*+{P_0 \oplus P_1 \oplus Z}="4"
		!{(-6,-4) }*+{X^{\vee} \oplus M_1^{\vee} \oplus P_2}="5"
		!{(-3,-4)}*+{X^{\vee}\oplus P_1 \oplus M_2^{\vee}}="6"
		!{(-1,-4) }*+{M_0^{\vee} \oplus Y \oplus P_2}="7"
		!{(1,-4)}*+{P_0 \oplus Y \oplus M_2}="8"
		!{(3,-4)}*+{M_0 \oplus P_1 \oplus Z}="9"
		!{(6,-4)}*+{P_0 \oplus M_1 \oplus Z}="10"
		!{(0,-5)}*+{M_0^{\vee} \oplus Y \oplus M_2}="14"
		!{(-9,-6)}*+{M_0^{\vee} \oplus M_1^{\vee} \oplus P_2}="11"
		!{(-6,-6)}*+{X^{\vee} \oplus M_1^{\vee} \oplus M_2^{\vee}}="12"
		!{(0,-6)}*+{M_0 \oplus P_1 \oplus M_2^{\vee}}="13"
		!{(6,-6)}*+{M_0 \oplus M_1 \oplus Z}="16"
		!{(9,-6)}*+{P_0 \oplus M_1 \oplus M_2}="15"
		!{(0,-8)}*+{M_0^{\vee} \oplus M_2}="20"
		!{(-9,-8)}*+{M_1^{\vee} \oplus M_2^{\vee}}="18"
		!{(-6,-8)}*+{M_0^{\vee} \oplus M_1^{\vee}\oplus Z^{\vee}}="17"
		!{(0,-9)}*+{M_0 \oplus Y^{\vee} \oplus M_2^{\vee}}="19"
		!{(6,-8)}*+{X \oplus M_1 \oplus M_2}="21"
		!{(9,-8)}*+{M_0 \oplus M_1}="22"
		!{(-6,-10)}*+{M_1^{\vee} \oplus Z^{\vee}}="23"
		!{(-3,-10)}*+{M_0^{\vee} \oplus Z^{\vee}}="24"
		!{(-1,-10)}*+{Y^{\vee} \oplus M_2^{\vee}}="25"
		!{(1,-10)}*+{M_0 \oplus Y^{\vee}}="26"
		!{(3,-10)}*+{X \oplus M_2}="27"
		!{(6,-10)}*+{X \oplus M_1}="28"
		!{(-3,-12)}*+{Z^{\vee}}="29"
		!{(0,-12)}*+{Y^{\vee}}="30"
		!{(3,-12)}*+{X}="31"
		!{(0,-14)}*+{0}="32"
		"1" - "2"
		"1" - "3"
		"1" - "4"
		"2" - "5"
		"2" - "6"
		"3" - "7"
		"3" - "8"
		"4" - "10"
		"4" - "9"
		"5" - "11"
		"5" - "12"
		"6" - "12"
		"6" - "13"
		"7" - "11"
		"7" - "14"
		"8" - "15"
		"8" - "14"
		"9" - "13"
		"9" - "16"
		"10" - "16"
		"10" - "15"
		"11" - "17"
		"12" - "18"
		"19" -@/^5pc/ "13"
		"14" -@/^5pc/ "20"
		"15" - "21"
		"16" - "22"
		"17" - "23"
		"17" - "24"
		"18" - "25"
		"18" - "23"
		"19" - "25"
		"19" - "26"
		"20" - "24"
		"20" - "27"
		"21" - "27"
		"21" - "28"
		"22" - "26"
		"22" - "28"
		"23" - "29"
		"24" - "29"
		"25" - "30"
		"26" - "30"
		"27" - "31"
		"28" - "31"
		"29" - "32"
		"30" - "32"
		"31" - "32"
	}
	}
\end{equation}}
\caption{$R(3K)$}
\label{R3Kstring}
\end{figure}

\newpage

\begin{figure}[ht]%checked
{\WthreeABCD}

{\normalsize Hasse Quiver:}
{\large \begin{equation}
	\resizebox{.80\hsize}{!}{
	\xygraph{
		!{<0cm,0cm>;<1.5cm,0cm>:<0cm,1.5cm>::}
		!{(0,0) }*+{P_0 \oplus P_1 \oplus P_2 }="1"	
		!{(-6,-2)}*+{P_0 \oplus M_2 \oplus P_2}="2"
		!{(0,-2)}*+{X^{\vee} \oplus P_1 \oplus P_2}="3"
		!{(6,-2) }*+{P_0 \oplus P_1 \oplus M_1}="4"
		!{(0,-4) }*+{P_0 \oplus M_2 \oplus M_1}="6"
		!{(-3,-6)}*+{X^{\vee}\oplus M_1^{\vee} \oplus P_2}="7"
		!{(3,-6) }*+{X^{\vee} \oplus P_1 \oplus M_2^{\vee}}="8"
		!{(-6,-7)}*+{M_2 \oplus P_2}="5"
		!{(0,-6)}*+{X \oplus M_2 \oplus M_1}="12"
		!{(6,-7)}*+{P_1 \oplus M_1}="9"
		!{(-6,-9)}*+{M_1^{\vee} \oplus P_2}="10"
		!{(6,-9)}*+{P_1 \oplus M_2^{\vee}}="14"
		!{(0,-10)}*+{X^{\vee} \oplus M_1^{\vee} \oplus M_2^{\vee}}="13"
		!{(-3,-10)}*+{M_2 \oplus X}="11"
		!{(3,-10)}*+{X \oplus M_1}="15"
		!{(0,-12)}*+{M_1^{\vee} \oplus M_2^{\vee}}="18"
		!{(-6,-14)}*+{M_1^{\vee}}="16"
		!{(0,-14)}*+{X}="17"
		!{(6,-14)}*+{M_2^{\vee}}="19"
		!{(0,-16)}*+{0}="20"
		"1" - "2"
		"1" - "3"
		"1" - "4"
		"2" - "5"
		"2" - "6"
		"3" - "7"
		"3" - "8"
		"4" - "6"
		"4" - "9"
		"5" - "10"
		"5" - "11"
		"6" - "12"
		"7" - "10"
		"7" - "13"
		"8" - "13"
		"8" - "14"
		"9" - "14"
		"9" - "15"
		"10" - "16"
		"11" - "17"
		"12" - "11"
		"12" - "15"
		"13" - "18"
		"14" - "19"
		"15" - "17"
		"16" - "20"
		"17" - "20"
		"18" - "16"
		"18" - "19"
		"19" - "20"
	}
	}
\end{equation}	}
\caption{$W(3ABCD)$}
\label{W3ABCDstring}
\end{figure}

\newpage

\begin{figure}[ht]%checked
{\WthreeQA}	

{\normalsize Hasse Quiver:}
{\large \begin{equation}
	\resizebox{1.0\hsize}{!}{
	\xygraph{
		!{<0cm,0cm>;<1.5cm,0cm>:<0cm,1.5cm>::}
		!{(0,0) }*+{P_0 \oplus P_1 \oplus P_2 }="1"
		!{(-4,-2) }*+{P_0 \oplus M_2 \oplus P_2}="2"
		!{(0,-2) }*+{X^{\vee} \oplus P_1 \oplus P_2}="3"
		!{(4,-2) }*+{P_0 \oplus P_1 \oplus M_1}="4"
		!{(-8,-4) }*+{M_0^{\vee} \oplus M_2 \oplus P_2}="5"
		!{(-4,-4)}*+{X^{\vee} \oplus M_1^{\vee} \oplus P_2}="7"
		!{(4,-4) }*+{X^{\vee} \oplus P_1 \oplus M_2^{\vee}}="8"
		!{(8,-4)}*+{M_0 \oplus P_1 \oplus M_1}="9"
		!{(0,-5)}*+{X^{\vee} \oplus M_1^{\vee} \oplus M_2^{\vee}}="13"
		!{(0,-7)}*+{P_0 \oplus M_2 \oplus M_1}="6"
		!{(-8,-8)}*+{M_0^{\vee} \oplus M_1^{\vee} \oplus P_2}="10"
		!{(-4,-8)}*+{M_0^{\vee} \oplus M_2}="11"
		!{(4,-8)}*+{M_0 \oplus M_1}="15"
		!{(8,-8)}*+{M_0 \oplus P_1 \oplus M_2^{\vee}}="14"
		!{(0,-9)}*+{M_2^{\vee} \oplus M_1^{\vee}}="19"
		!{(0,-11)}*+{X \oplus M_2 \oplus M_1}="12"
		!{(-8,-12)}*+{M_0^{\vee} \oplus M_1^{\vee}}="16"
		!{(-4,-12)}*+{X \oplus M_2}="17"
		!{(4,-12)}*+{M_1 \oplus X}="18"
		!{(8,-12)}*+{M_0 \oplus M_2^{\vee}}="20"
		!{(-4,-14)}*+{M_1^{\vee}}="21"
		!{(0,-14)}*+{X}="22"
		!{(4,-14)}*+{M_2^{\vee}}="23"
		!{(0,-16)}*+{0}="24"
		"1" - "2"
		"1" - "3"
		"1" - "4"
		"2" - "5"
		"2" - "6"
		"3" - "7"
		"3" - "8"
		"4" - "6"
		"4" - "9"
		"5" - "10"
		"5" - "11"
		"6" -@/^5pc/ "12"
		"7" - "10"
		"7" - "13"
		"8" - "13"
		"8" - "14"
		"9" - "15"
		"9" - "14"
		"10" - "16"
		"11" - "16"
		"11" - "17"
		"12" - "17"
		"12" - "18"
		"19" -@/^5pc/ "13"
		"14" - "20"
		"15" - "18"
		"15" - "20"
		"16" - "21"
		"17" - "22"
		"18" - "22"
		"19" - "21"
		"19" - "23"
		"20" - "23"
		"21" - "24"
		"22" - "24"
		"23" - "24"
	}
	}
\end{equation}}
\caption{$W(3QA_1)$}
\label{WQ3A1string}
\end{figure}

\newpage

\begin{figure}[ht]%checked
{\WthreeF}
		
Hasse Quiver:
{\large \begin{equation}
	\resizebox{0.9\hsize}{!}{
	\xygraph{
		!{<0cm,0cm>;<1.5cm,0cm>:<0cm,1.5cm>::}
		!{(0,0) }*+{P_0 \oplus P_1 \oplus P_2 }="1"
		!{(-3,-2)}*+{P_0 \oplus P_1 \oplus M_0}="4"
		!{(0,-2)}*+{M_1^{\vee} \oplus P_1 \oplus P_2}="2"
		!{(3,-2)}*+{P_0 \oplus Y \oplus P_2}="3"
		!{(6,-4)}*+{M_0^{\vee} \oplus Y \oplus P_2}="7"
		!{(3,-5)}*+{P_0 \oplus Y \oplus M_2}="8"
		!{(-3,-4)}*+{P_0 \oplus M_1 \oplus M_0}="10"
		!{(-3,-6)}*+{M_1^{\vee} \oplus P_1 \oplus M_2^{\vee}}="6"
		!{(0,-6)}*+{P_0 \oplus M_1 \oplus M_2}="14"
		!{(3,-7)}*+{M_1^{\vee} \oplus M_0^{\vee} \oplus P_2}="5"
		!{(6,-6)}*+{M_0^{\vee} \oplus Y \oplus M_2}="13"
		!{(-8,-7)}*+{M_2^{\vee} \oplus P_1 \oplus M_0}="9"
		!{(-8,-9)}*+{M_1 \oplus M_0}="16"
		!{(-3,-10)}*+{M_1 \oplus M_2}="20"
		!{(0,-10)}*+{M_1^{\vee} \oplus M_2^{\vee}}="12"
		!{(3,-9)}*+{M_0^{\vee} \oplus M_2}="19"
		!{(6,-10)}*+{M_1^{\vee} \oplus M_0^{\vee} \oplus Z^{\vee}}="11"
		!{(-3,-12)}*+{M_2^{\vee} \oplus M_0}="15"
		!{(3,-11)}*+{M_1^{\vee} \oplus Z^{\vee}}="18"
		!{(6,-12)}*+{M_0^{\vee} \oplus Z^{\vee}}="17"
		!{(-3,-14)}*+{M_0}="21"
		!{(0,-14)}*+{M_2}="23"
		!{(3,-14)}*+{Z^{\vee}}="22"
		!{(0,-16)}*+{0}="24"
		"1" - "2"
		"1" - "3"
		"1" - "4"
		"2" - "5"
		"2" - "6"
		"3" - "7"
		"3" - "8"
		"4" - "10"
		"4" - "9"
		"5" - "11"
		"6" - "9"
		"6" - "12"
		"7" - "5"
		"7" - "13"
		"8" - "13"
		"8" - "14"
		"9" - "15"
		"10" - "16"
		"10" - "14"
		"11" - "18"
		"11" - "17"
		"12" - "15"
		"12" - "18"
		"13" - "19"
		"14" - "20"
		"15" - "21"
		"16" - "20"
		"16" - "21"
		"17" - "22"
		"18" - "22"
		"19" - "17"
		"19" - "23"
		"20" - "23"
		"21" - "24"
		"22" - "24"
		"23" - "24"
	}
	}
\end{equation}}
\caption{$W(3F)$}
\label{W3Fstring}
\end{figure}	

\newpage

\begin{figure}[ht]%checked
{\WthreeQLR}

{\normalsize Hasse Quiver:}
{\Large \begin{equation}
	\resizebox{.9\hsize}{!}{
	\xygraph{
		!{<0cm,0cm>;<1.5cm,0cm>:<0cm,1.5cm>::}
		!{(0,0) }*+{P_0 \oplus P_1 \oplus P_2 }="1"
		!{(-3,-2)}*+{M_1^{\vee} \oplus P_1 \oplus P_2}="2"
		!{(0,-2)}*+{P_0 \oplus P_1 \oplus M_0}="4"
		!{(3,-2)}*+{P_0 \oplus M_2 \oplus P_2}="3"
		!{(-6,-4)}*+{M_1^{\vee} \oplus P_1 \oplus M_2^{\vee}}="6"
		!{(0,-5)}*+{M_0^{\vee} \oplus M_2 \oplus P_2}="7"
		!{(3,-4)}*+{P_0 \oplus M_1 \oplus M_0}="10"
		!{(6,-6)}*+{P_0 \oplus M_2 \oplus M_1}="8"
		!{(-10,-7)}*+{M_2^{\vee} \oplus P_1 \oplus M_0}="9"
		!{(-3,-7)}*+{M_1^{\vee} \oplus M_0^{\vee} \oplus P_2}="5"
		!{(-10,-9)}*+{M_1^{\vee} \oplus M_2^{\vee}}="12"
		!{(-3,-9)}*+{M_1 \oplus M_0}="16"
		!{(6,-10)}*+{M_0^{\vee} \oplus M_2}="13"
		!{(0,-10)}*+{M_2 \oplus M_1}="14"
		!{(3,-11)}*+{M_1^{\vee} \oplus M_0^{\vee}}="11"
		!{(-6,-11)}*+{M_2^{\vee} \oplus M_0}="15"
		!{(-3,-13)}*+{M_0}="19"
		!{(0,-13)}*+{M_1^{\vee}}="17"
		!{(3,-13)}*+{M_2}="18"
		!{(0,-15)}*+{0}="20"
		"1" - "2"
		"1" - "3"
		"1" - "4"
		"2" - "5"
		"2" - "6"
		"3" - "7"
		"3" - "8"
		"4" - "10"
		"4" - "9"
		"5" - "11"
		"6" - "9"
		"6" - "12"
		"7" - "5"
		"7" - "13"
		"8" - "14"
		"9" - "15"
		"10" - "16"
		"10" - "8"
		"11" - "17"
		"12" - "15"
		"12" - "17"
		"13" - "18"
		"13" - "11"
		"14" - "18"
		"15" - "19"
		"16" - "14"
		"16" - "19"
		"17" - "20"
		"18" - "20"
		"19" - "20"
	}
	}
\end{equation}	}
\caption{$W(3QLR)$}
\label{W3QLRstring}
\end{figure}
%
%
%	$
%	\xygraph{
%		!{<0cm,0cm>;<1cm,0cm>:<0cm,1cm>::}
%		!{(0,0) }*+{\bullet_{0}}="a"
%		!{(2,0) }*+{\bullet_{1}}="b"
%		!{(1,-2) }*+{\bullet_{1}}="c"
%		"b" :@/_/_{\gamma} "a"
%		"a" :@/_/_{\beta} "b"
%		"a" :@(ld,lu)^{\alpha} "a"
%		"b" :@(rd,ru)_{\eta} "b"
%		"b" :@/_/_{\delta} "c"
%		"c" :@/_/_{\tau} "b"
%	}
%	$
%\begin{landscape}
%\input{string_tables}
%\RtwoAB
%\WtwoB
%\RthreeABD
%\RthreeC
%\RthreeH
%\RthreeK
%\WthreeABCD
%\WthreeQA
%\WthreeF
%\WthreeQLR
%\end{landscape}

\newpage
\bibliography{mybibs}{}
\bibliographystyle{amsplain}
\def\cprime{$'$} \def\cprime{$'$} \def\cprime{$'$}
\providecommand{\bysame}{\leavevmode\hbox to3em{\hrulefill}\thinspace}
\providecommand{\MR}{\relax\ifhmode\unskip\space\fi MR }
\providecommand{\MRhref}[2]{%
  \href{http://www.ams.org/mathscinet-getitem?mr=#1}{#2}
}
\providecommand{\href}[2]{#2}

\end{document}

%% file: string_tables.tex
\newcommand{\RtwoAB}{\renewcommand\arraystretch{1.5}
\tiny\begin{longtable}{ccccc}
\Cline{1-5}{1.7pt}
$\mathbf{R(2AB)}$ & \textbf{Rigid string $C$} & $_PC_P$ & \textbf{$g$-vector} & \textbf{Other $C$-rigid strings} \\\Cline{1-5}{1.3pt} \endhead
\Cline{1-5}{1.7pt} \endfoot
$\mathbf{P_0}$ & $1\stackrel{\beta}{\longleftarrow}0\stackrel{\alpha\beta}{\longrightarrow}1$ & $0$ & $[ 1, 0 ]$ & $\mathbf{P_1}$, ${X}$, ${X^\vee}$, $\mathbf{Y}$, ${Y^\vee}$\\
$\mathbf{P_1}$ & $1\stackrel{\gamma\alpha}{\longrightarrow}0$ & $1$ & $[ 0, 1 ]$ & $\mathbf{P_0}$, ${X}$, $\mathbf{X^\vee}$, ${Y}$, ${Y^\vee}$\\
$\mathbf{X}$ & $0\stackrel{\alpha}{\longrightarrow}0$ & $1\stackrel{\beta}{\longleftarrow}0\stackrel{\alpha\beta}{\longrightarrow}1$ & $[ 1, -2 ]$ & $\mathbf{P_1^\vee}$, $\mathbf{Y}$\\
$\mathbf{X^\vee}$ & $1\stackrel{\gamma}{\longrightarrow}0\stackrel{\gamma\alpha}{\longleftarrow}1$ & $1\stackrel{\gamma}{\longrightarrow}0\stackrel{\gamma\alpha}{\longleftarrow}1$ & $[ -1, 2 ]$ & $\mathbf{P_1}$, $\mathbf{Y^\vee}$\\
$\mathbf{Y}$ & $0\stackrel{\alpha\beta}{\longrightarrow}1$ & $0\stackrel{\beta}{\longrightarrow}1$ & $[ 1, -1 ]$ & $\mathbf{P_0}$, $\mathbf{X}$\\
$\mathbf{Y^\vee}$ & $1$ & $1\stackrel{\gamma}{\longrightarrow}0$ & $[ -1, 1 ]$ & $\mathbf{P_0^\vee}$, $\mathbf{X^\vee}$\\
\end{longtable}
}
\newcommand{\WtwoB}{\renewcommand\arraystretch{1.5}
\tiny\begin{longtable}{ccccc}
\Cline{1-5}{1.7pt}
$\mathbf{W(2B)}$ & \textbf{Rigid string $C$} & $_PC_P$ & \textbf{$g$-vector} & \textbf{Other $C$-rigid strings} \\\Cline{1-5}{1.3pt} \endhead
\Cline{1-5}{1.7pt} \endfoot
$\mathbf{P_0}$ & $0\stackrel{\beta}{\longrightarrow}1$ & $0$ & $[ 1, 0 ]$ & $\mathbf{P_1}$, $\mathbf{Y}$, ${Y^\vee}$\\
$\mathbf{P_1}$ & $1\stackrel{\gamma}{\longrightarrow}0$ & $1$ & $[ 0, 1 ]$ & $\mathbf{P_0}$, ${Y}$, $\mathbf{Y^\vee}$\\
$\mathbf{Y}$ & $0$ & $0\stackrel{\beta}{\longrightarrow}1$ & $[ 1, -1 ]$ & $\mathbf{P_1^\vee}$, $\mathbf{P_0}$\\
$\mathbf{Y^\vee}$ & $1$ & $1\stackrel{\gamma}{\longrightarrow}0$ & $[ -1, 1 ]$ & $\mathbf{P_0^\vee}$, $\mathbf{P_1}$\\
\end{longtable}
}
\newcommand{\RthreeABD}{\renewcommand\arraystretch{1.5}
\tiny\begin{longtable}{ccccc}
\Cline{1-5}{1.7pt}
$\mathbf{R(3ABD)}$ & \textbf{Rigid string $C$} & $_PC_P$ & \textbf{$g$-vector} & \textbf{Other $C$-rigid strings} \\\Cline{1-5}{1.3pt} \endhead
\Cline{1-5}{1.7pt} \endfoot
$\mathbf{P_0}$ & $1\stackrel{\delta\eta\gamma}{\longleftarrow}0\stackrel{\gamma\beta\delta}{\longrightarrow}2$ & $0$ & $[ 1, 0, 0 ]$ & $\mathbf{P_1}$, $\mathbf{P_2}$, ${X}$, ${X^\vee}$, $\mathbf{M_2}$, ${M_2^\vee}$, $\mathbf{M_1}$, ${M_1^\vee}$, ${M_0}$, ${M_0^\vee}$, ${T_1}$, ${T_1^\vee}$, ${T_2}$, ${T_2^\vee}$\\
$\mathbf{P_1}$ & $1\stackrel{\beta\delta\eta}{\longrightarrow}0$ & $1$ & $[ 0, 1, 0 ]$ & $\mathbf{P_0}$, $\mathbf{P_2}$, ${X}$, $\mathbf{X^\vee}$, ${M_2}$, ${M_2^\vee}$, $\mathbf{M_1}$, ${M_1^\vee}$, $\mathbf{M_0}$, ${M_0^\vee}$, ${T_1}$, ${T_1^\vee}$, ${T_2}$, $\mathbf{T_2^\vee}$\\
$\mathbf{P_2}$ & $2\stackrel{\eta\gamma\beta}{\longrightarrow}0$ & $2$ & $[ 0, 0, 1 ]$ & $\mathbf{P_0}$, $\mathbf{P_1}$, ${X}$, $\mathbf{X^\vee}$, $\mathbf{M_2}$, ${M_2^\vee}$, ${M_1}$, ${M_1^\vee}$, ${M_0}$, $\mathbf{M_0^\vee}$, ${T_1}$, $\mathbf{T_1^\vee}$, ${T_2}$, ${T_2^\vee}$\\
$\mathbf{X}$ & $0$ & $2\stackrel{\delta}{\longleftarrow}0\stackrel{\gamma}{\longrightarrow}1$ & $[ 1, -1, -1 ]$ & $\mathbf{P_1^\vee}$, $\mathbf{P_2^\vee}$, $\mathbf{M_2}$, $\mathbf{M_1}$, $\mathbf{T_1}$, $\mathbf{T_2}$\\
$\mathbf{X^\vee}$ & $1\stackrel{\beta}{\longrightarrow}0\stackrel{\eta}{\longleftarrow}2$ & $1\stackrel{\beta}{\longrightarrow}0\stackrel{\eta}{\longleftarrow}2$ & $[ -1, 1, 1 ]$ & $\mathbf{P_1}$, $\mathbf{P_2}$, $\mathbf{M_2^\vee}$, $\mathbf{M_1^\vee}$, ${M_0}$, ${M_0^\vee}$, $\mathbf{T_1^\vee}$, $\mathbf{T_2^\vee}$\\
$\mathbf{M_2}$ & $0\stackrel{\delta\eta\gamma}{\longrightarrow}1$ & $0\stackrel{\gamma}{\longrightarrow}1$ & $[ 1, -1, 0 ]$ & $\mathbf{P_0}$, $\mathbf{P_2}$, $\mathbf{X}$, $\mathbf{M_1}$, ${M_1^\vee}$, $\mathbf{M_0^\vee}$, ${T_1}$, ${T_1^\vee}$, $\mathbf{T_2}$\\
$\mathbf{M_2^\vee}$ & $1$ & $1\stackrel{\beta}{\longrightarrow}0$ & $[ -1, 1, 0 ]$ & $\mathbf{P_0^\vee}$, $\mathbf{P_2^\vee}$, $\mathbf{X^\vee}$, $\mathbf{M_1^\vee}$, $\mathbf{M_0}$, $\mathbf{T_2^\vee}$\\
$\mathbf{M_1}$ & $0\stackrel{\gamma\beta\delta}{\longrightarrow}2$ & $0\stackrel{\delta}{\longrightarrow}2$ & $[ 1, 0, -1 ]$ & $\mathbf{P_0}$, $\mathbf{P_1}$, $\mathbf{X}$, $\mathbf{M_2}$, ${M_2^\vee}$, $\mathbf{M_0}$, $\mathbf{T_1}$, ${T_2}$, ${T_2^\vee}$\\
$\mathbf{M_1^\vee}$ & $2$ & $2\stackrel{\eta}{\longrightarrow}0$ & $[ -1, 0, 1 ]$ & $\mathbf{P_0^\vee}$, $\mathbf{P_1^\vee}$, $\mathbf{X^\vee}$, $\mathbf{M_2^\vee}$, $\mathbf{M_0^\vee}$, $\mathbf{T_1^\vee}$\\
$\mathbf{M_0}$ & $1\stackrel{\beta}{\longrightarrow}0$ & $1\stackrel{\beta\delta}{\longrightarrow}2$ & $[ 0, 1, -1 ]$ & $\mathbf{P_2^\vee}$, $\mathbf{P_1}$, ${X}$, $\mathbf{M_2^\vee}$, $\mathbf{M_1}$, $\mathbf{T_1}$, $\mathbf{T_2^\vee}$\\
$\mathbf{M_0^\vee}$ & $2\stackrel{\eta}{\longrightarrow}0$ & $2\stackrel{\eta\gamma}{\longrightarrow}1$ & $[ 0, -1, 1 ]$ & $\mathbf{P_1^\vee}$, $\mathbf{P_2}$, ${X}$, $\mathbf{M_2}$, $\mathbf{M_1^\vee}$, $\mathbf{T_1^\vee}$, $\mathbf{T_2}$\\
$\mathbf{T_1}$ & $0\stackrel{\gamma\beta}{\longrightarrow}0$ & $2\stackrel{\delta}{\longleftarrow}0\stackrel{\gamma\beta\delta}{\longrightarrow}2$ & $[ 1, 0, -2 ]$ & $\mathbf{P_2^\vee}$, $\mathbf{X}$, ${M_2^\vee}$, $\mathbf{M_1}$, $\mathbf{M_0}$\\
$\mathbf{T_1^\vee}$ & $2\stackrel{\eta}{\longrightarrow}0\stackrel{\eta\gamma\beta}{\longleftarrow}2$ & $2\stackrel{\eta}{\longrightarrow}0\stackrel{\eta\gamma\beta}{\longleftarrow}2$ & $[ -1, 0, 2 ]$ & $\mathbf{P_2}$, $\mathbf{X^\vee}$, ${M_2^\vee}$, $\mathbf{M_1^\vee}$, $\mathbf{M_0^\vee}$\\
$\mathbf{T_2}$ & $0\stackrel{\delta\eta}{\longrightarrow}0$ & $1\stackrel{\gamma}{\longleftarrow}0\stackrel{\delta\eta\gamma}{\longrightarrow}1$ & $[ 1, -2, 0 ]$ & $\mathbf{P_1^\vee}$, $\mathbf{X}$, $\mathbf{M_2}$, ${M_1^\vee}$, $\mathbf{M_0^\vee}$\\
$\mathbf{T_2^\vee}$ & $1\stackrel{\beta}{\longrightarrow}0\stackrel{\beta\delta\eta}{\longleftarrow}1$ & $1\stackrel{\beta}{\longrightarrow}0\stackrel{\beta\delta\eta}{\longleftarrow}1$ & $[ -1, 2, 0 ]$ & $\mathbf{P_1}$, $\mathbf{X^\vee}$, $\mathbf{M_2^\vee}$, ${M_1^\vee}$, $\mathbf{M_0}$\\
\end{longtable}
}
\newcommand{\RthreeC}{\renewcommand\arraystretch{1.5}
\tiny\begin{longtable}{ccccc}
\Cline{1-5}{1.7pt}
$\mathbf{R(3C)}$ & \textbf{Rigid string $C$} & $_PC_P$ & \textbf{$g$-vector} & \textbf{Other $C$-rigid strings} \\\Cline{1-5}{1.3pt} \endhead
\Cline{1-5}{1.7pt} \endfoot
$\mathbf{P_0}$ & $1\stackrel{\gamma}{\longleftarrow}0\stackrel{\delta}{\longrightarrow}2$ & $0$ & $[ 1, 0, 0 ]$ & $\mathbf{P_1}$, $\mathbf{P_2}$, ${X}$, ${X^\vee}$, $\mathbf{M_2}$, ${M_2^\vee}$, $\mathbf{M_1}$, ${M_1^\vee}$, ${M_0}$, ${M_0^\vee}$\\
$\mathbf{P_1}$ & $1\stackrel{\beta\delta}{\longrightarrow}2$ & $1$ & $[ 0, 1, 0 ]$ & $\mathbf{P_0}$, $\mathbf{P_2}$, ${X}$, $\mathbf{X^\vee}$, ${M_2}$, $\mathbf{M_2^\vee}$, $\mathbf{M_1}$, ${M_1^\vee}$, $\mathbf{M_0}$, ${M_0^\vee}$\\
$\mathbf{P_2}$ & $2\stackrel{\eta\gamma}{\longrightarrow}1$ & $2$ & $[ 0, 0, 1 ]$ & $\mathbf{P_0}$, $\mathbf{P_1}$, ${X}$, $\mathbf{X^\vee}$, $\mathbf{M_2}$, ${M_2^\vee}$, ${M_1}$, $\mathbf{M_1^\vee}$, ${M_0}$, $\mathbf{M_0^\vee}$\\
$\mathbf{X}$ & $0$ & $2\stackrel{\delta}{\longleftarrow}0\stackrel{\gamma}{\longrightarrow}1$ & $[ 1, -1, -1 ]$ & $\mathbf{P_1^\vee}$, $\mathbf{P_2^\vee}$, $\mathbf{M_2}$, $\mathbf{M_1}$\\
$\mathbf{X^\vee}$ & $1\stackrel{\beta}{\longrightarrow}0\stackrel{\eta}{\longleftarrow}2$ & $1\stackrel{\beta}{\longrightarrow}0\stackrel{\eta}{\longleftarrow}2$ & $[ -1, 1, 1 ]$ & $\mathbf{P_1}$, $\mathbf{P_2}$, $\mathbf{M_2^\vee}$, $\mathbf{M_1^\vee}$, ${M_0}$, ${M_0^\vee}$\\
$\mathbf{M_2}$ & $0\stackrel{\delta}{\longrightarrow}2$ & $0\stackrel{\gamma}{\longrightarrow}1$ & $[ 1, -1, 0 ]$ & $\mathbf{P_1^\vee}$, $\mathbf{P_0}$, $\mathbf{P_2}$, $\mathbf{X}$, $\mathbf{M_1}$, ${M_1^\vee}$, $\mathbf{M_0^\vee}$\\
$\mathbf{M_2^\vee}$ & $1$ & $1\stackrel{\beta}{\longrightarrow}0$ & $[ -1, 1, 0 ]$ & $\mathbf{P_0^\vee}$, $\mathbf{P_2^\vee}$, $\mathbf{P_1}$, $\mathbf{X^\vee}$, $\mathbf{M_1^\vee}$, $\mathbf{M_0}$\\
$\mathbf{M_1}$ & $0\stackrel{\gamma}{\longrightarrow}1$ & $0\stackrel{\delta}{\longrightarrow}2$ & $[ 1, 0, -1 ]$ & $\mathbf{P_2^\vee}$, $\mathbf{P_0}$, $\mathbf{P_1}$, $\mathbf{X}$, $\mathbf{M_2}$, ${M_2^\vee}$, $\mathbf{M_0}$\\
$\mathbf{M_1^\vee}$ & $2$ & $2\stackrel{\eta}{\longrightarrow}0$ & $[ -1, 0, 1 ]$ & $\mathbf{P_0^\vee}$, $\mathbf{P_1^\vee}$, $\mathbf{P_2}$, $\mathbf{X^\vee}$, $\mathbf{M_2^\vee}$, $\mathbf{M_0^\vee}$\\
$\mathbf{M_0}$ & $1\stackrel{\beta}{\longrightarrow}0$ & $1\stackrel{\beta\delta}{\longrightarrow}2$ & $[ 0, 1, -1 ]$ & $\mathbf{P_2^\vee}$, $\mathbf{P_1}$, ${X}$, $\mathbf{M_2^\vee}$, $\mathbf{M_1}$\\
$\mathbf{M_0^\vee}$ & $2\stackrel{\eta}{\longrightarrow}0$ & $2\stackrel{\eta\gamma}{\longrightarrow}1$ & $[ 0, -1, 1 ]$ & $\mathbf{P_1^\vee}$, $\mathbf{P_2}$, ${X}$, $\mathbf{M_2}$, $\mathbf{M_1^\vee}$\\
\end{longtable}
}
\newcommand{\RthreeH}{\renewcommand\arraystretch{1.5}
\tiny\begin{longtable}{ccccc}
\Cline{1-5}{1.7pt}
$\mathbf{R(3H)}$ & \textbf{Rigid string $C$} & $_PC_P$ & \textbf{$g$-vector} & \textbf{Other $C$-rigid strings} \\\Cline{1-5}{1.3pt} \endhead
\Cline{1-5}{1.7pt} \endfoot
$\mathbf{P_0}$ & $0\stackrel{\beta\delta}{\longrightarrow}2$ & $0$ & $[ 1, 0, 0 ]$ & $\mathbf{P_1}$, $\mathbf{P_2}$, ${X^\vee}$, $\mathbf{M_2}$, ${M_2^\vee}$, $\mathbf{M_1}$, ${M_1^\vee}$, $\mathbf{M_0}$, ${M_0^\vee}$, $\mathbf{Y}$, ${Y^\vee}$, ${Z^\vee}$\\
$\mathbf{P_1}$ & $0\stackrel{\gamma}{\longleftarrow}1\stackrel{\delta}{\longrightarrow}2$ & $1$ & $[ 0, 1, 0 ]$ & $\mathbf{P_0}$, $\mathbf{P_2}$, $\mathbf{X^\vee}$, ${M_2}$, $\mathbf{M_2^\vee}$, ${M_1}$, ${M_1^\vee}$, $\mathbf{M_0}$, ${M_0^\vee}$, ${Y}$, ${Y^\vee}$, ${Z^\vee}$\\
$\mathbf{P_2}$ & $0\stackrel{\lambda}{\longleftarrow}2\stackrel{\eta}{\longrightarrow}1$ & $2$ & $[ 0, 0, 1 ]$ & $\mathbf{P_0}$, $\mathbf{P_1}$, $\mathbf{X^\vee}$, ${M_2}$, ${M_2^\vee}$, ${M_1}$, $\mathbf{M_1^\vee}$, ${M_0}$, $\mathbf{M_0^\vee}$, $\mathbf{Y}$, ${Y^\vee}$, ${Z^\vee}$\\
$\mathbf{X^\vee}$ & $1\stackrel{\eta}{\longleftarrow}2\stackrel{\lambda}{\longrightarrow}0\stackrel{\gamma}{\longleftarrow}1\stackrel{\delta}{\longrightarrow}2$ & $2\stackrel{\lambda}{\longrightarrow}0\stackrel{\gamma}{\longleftarrow}1$ & $[ -1, 1, 1 ]$ & $\mathbf{P_1}$, $\mathbf{P_2}$, $\mathbf{M_2^\vee}$, $\mathbf{M_1^\vee}$, ${M_0}$, ${M_0^\vee}$, ${Y^\vee}$, ${Z^\vee}$\\
$\mathbf{M_2}$ & $0$ & $0\stackrel{\beta}{\longrightarrow}1$ & $[ 1, -1, 0 ]$ & $\mathbf{P_1^\vee}$, $\mathbf{P_2^\vee}$, $\mathbf{P_0}$, $\mathbf{M_1}$, $\mathbf{M_0^\vee}$, $\mathbf{Y}$, ${Z^\vee}$\\
$\mathbf{M_2^\vee}$ & $1\stackrel{\delta}{\longrightarrow}2$ & $1\stackrel{\gamma}{\longrightarrow}0$ & $[ -1, 1, 0 ]$ & $\mathbf{P_0^\vee}$, $\mathbf{P_1}$, $\mathbf{X^\vee}$, $\mathbf{M_1^\vee}$, $\mathbf{M_0}$, $\mathbf{Y^\vee}$, ${Z^\vee}$\\
$\mathbf{M_1}$ & $0\stackrel{\beta}{\longrightarrow}1$ & $0\stackrel{\beta\delta}{\longrightarrow}2$ & $[ 1, 0, -1 ]$ & $\mathbf{P_2^\vee}$, $\mathbf{P_0}$, $\mathbf{M_2}$, $\mathbf{M_0}$, ${Y^\vee}$\\
$\mathbf{M_1^\vee}$ & $2\stackrel{\eta}{\longrightarrow}1$ & $2\stackrel{\lambda}{\longrightarrow}0$ & $[ -1, 0, 1 ]$ & $\mathbf{P_0^\vee}$, $\mathbf{P_2}$, $\mathbf{X^\vee}$, $\mathbf{M_2^\vee}$, $\mathbf{M_0^\vee}$, ${Y^\vee}$, $\mathbf{Z^\vee}$\\
$\mathbf{M_0}$ & $1\stackrel{\gamma}{\longrightarrow}0$ & $1\stackrel{\delta}{\longrightarrow}2$ & $[ 0, 1, -1 ]$ & $\mathbf{P_2^\vee}$, $\mathbf{P_0}$, $\mathbf{P_1}$, ${M_2}$, $\mathbf{M_2^\vee}$, $\mathbf{M_1}$, $\mathbf{Y^\vee}$\\
$\mathbf{M_0^\vee}$ & $2\stackrel{\lambda}{\longrightarrow}0$ & $2\stackrel{\eta}{\longrightarrow}1$ & $[ 0, -1, 1 ]$ & $\mathbf{P_1^\vee}$, $\mathbf{P_2}$, $\mathbf{M_2}$, $\mathbf{M_1^\vee}$, $\mathbf{Y}$, $\mathbf{Z^\vee}$\\
$\mathbf{Y}$ & $0\stackrel{\beta}{\longrightarrow}1\stackrel{\eta}{\longleftarrow}2\stackrel{\lambda}{\longrightarrow}0$ & $0\stackrel{\beta}{\longrightarrow}1\stackrel{\eta}{\longleftarrow}2$ & $[ 1, -1, 1 ]$ & $\mathbf{P_0}$, $\mathbf{P_2}$, $\mathbf{M_2}$, ${M_1}$, ${M_1^\vee}$, $\mathbf{M_0^\vee}$, ${Z^\vee}$\\
$\mathbf{Y^\vee}$ & $1$ & $2\stackrel{\delta}{\longleftarrow}1\stackrel{\gamma}{\longrightarrow}0$ & $[ -1, 1, -1 ]$ & $\mathbf{P_0^\vee}$, $\mathbf{P_2^\vee}$, $\mathbf{M_2^\vee}$, $\mathbf{M_0}$\\
$\mathbf{Z^\vee}$ & $2$ & $1\stackrel{\eta}{\longleftarrow}2\stackrel{\lambda}{\longrightarrow}0$ & $[ -1, -1, 1 ]$ & $\mathbf{P_0^\vee}$, $\mathbf{P_1^\vee}$, $\mathbf{M_1^\vee}$, $\mathbf{M_0^\vee}$\\
\end{longtable}
}
\newcommand{\RthreeK}{\renewcommand\arraystretch{1.5}
\tiny\begin{longtable}{ccccc}
\Cline{1-5}{1.7pt}
$\mathbf{R(3K)}$ & \textbf{Rigid string $C$} & $_PC_P$ & \textbf{$g$-vector} & \textbf{Other $C$-rigid strings} \\\Cline{1-5}{1.3pt} \endhead
\Cline{1-5}{1.7pt} \endfoot
$\mathbf{P_0}$ & $1\stackrel{\beta}{\longleftarrow}0\stackrel{\kappa}{\longrightarrow}2$ & $0$ & $[ 1, 0, 0 ]$ & $\mathbf{P_1}$, $\mathbf{P_2}$, ${X}$, ${X^\vee}$, $\mathbf{M_2}$, ${M_2^\vee}$, $\mathbf{M_1}$, ${M_1^\vee}$, ${M_0}$, ${M_0^\vee}$, $\mathbf{Y}$, ${Y^\vee}$, $\mathbf{Z}$, ${Z^\vee}$\\
$\mathbf{P_1}$ & $0\stackrel{\gamma}{\longleftarrow}1\stackrel{\delta}{\longrightarrow}2$ & $1$ & $[ 0, 1, 0 ]$ & $\mathbf{P_0}$, $\mathbf{P_2}$, ${X}$, $\mathbf{X^\vee}$, ${M_2}$, $\mathbf{M_2^\vee}$, ${M_1}$, ${M_1^\vee}$, $\mathbf{M_0}$, ${M_0^\vee}$, ${Y}$, ${Y^\vee}$, $\mathbf{Z}$, ${Z^\vee}$\\
$\mathbf{P_2}$ & $0\stackrel{\lambda}{\longleftarrow}2\stackrel{\eta}{\longrightarrow}1$ & $2$ & $[ 0, 0, 1 ]$ & $\mathbf{P_0}$, $\mathbf{P_1}$, ${X}$, $\mathbf{X^\vee}$, ${M_2}$, ${M_2^\vee}$, ${M_1}$, $\mathbf{M_1^\vee}$, ${M_0}$, $\mathbf{M_0^\vee}$, $\mathbf{Y}$, ${Y^\vee}$, ${Z}$, ${Z^\vee}$\\
$\mathbf{X}$ & $0$ & $1\stackrel{\beta}{\longleftarrow}0\stackrel{\kappa}{\longrightarrow}2$ & $[ 1, -1, -1 ]$ & $\mathbf{P_1^\vee}$, $\mathbf{P_2^\vee}$, $\mathbf{M_2}$, $\mathbf{M_1}$\\
$\mathbf{X^\vee}$ & $1\stackrel{\eta}{\longleftarrow}2\stackrel{\lambda}{\longrightarrow}0\stackrel{\gamma}{\longleftarrow}1\stackrel{\delta}{\longrightarrow}2$ & $2\stackrel{\lambda}{\longrightarrow}0\stackrel{\gamma}{\longleftarrow}1$ & $[ -1, 1, 1 ]$ & $\mathbf{P_1}$, $\mathbf{P_2}$, $\mathbf{M_2^\vee}$, $\mathbf{M_1^\vee}$, ${M_0}$, ${M_0^\vee}$, ${Y^\vee}$, ${Z^\vee}$\\
$\mathbf{M_2}$ & $0\stackrel{\kappa}{\longrightarrow}2$ & $0\stackrel{\beta}{\longrightarrow}1$ & $[ 1, -1, 0 ]$ & $\mathbf{P_1^\vee}$, $\mathbf{P_0}$, $\mathbf{X}$, $\mathbf{M_1}$, $\mathbf{M_0^\vee}$, $\mathbf{Y}$, ${Z^\vee}$\\
$\mathbf{M_2^\vee}$ & $1\stackrel{\delta}{\longrightarrow}2$ & $1\stackrel{\gamma}{\longrightarrow}0$ & $[ -1, 1, 0 ]$ & $\mathbf{P_0^\vee}$, $\mathbf{P_1}$, $\mathbf{X^\vee}$, $\mathbf{M_1^\vee}$, $\mathbf{M_0}$, $\mathbf{Y^\vee}$, ${Z^\vee}$\\
$\mathbf{M_1}$ & $0\stackrel{\beta}{\longrightarrow}1$ & $0\stackrel{\kappa}{\longrightarrow}2$ & $[ 1, 0, -1 ]$ & $\mathbf{P_2^\vee}$, $\mathbf{P_0}$, $\mathbf{X}$, $\mathbf{M_2}$, $\mathbf{M_0}$, ${Y^\vee}$, $\mathbf{Z}$\\
$\mathbf{M_1^\vee}$ & $2\stackrel{\eta}{\longrightarrow}1$ & $2\stackrel{\lambda}{\longrightarrow}0$ & $[ -1, 0, 1 ]$ & $\mathbf{P_0^\vee}$, $\mathbf{P_2}$, $\mathbf{X^\vee}$, $\mathbf{M_2^\vee}$, $\mathbf{M_0^\vee}$, ${Y^\vee}$, $\mathbf{Z^\vee}$\\
$\mathbf{M_0}$ & $1\stackrel{\gamma}{\longrightarrow}0$ & $1\stackrel{\delta}{\longrightarrow}2$ & $[ 0, 1, -1 ]$ & $\mathbf{P_2^\vee}$, $\mathbf{P_1}$, ${X}$, $\mathbf{M_2^\vee}$, $\mathbf{M_1}$, $\mathbf{Y^\vee}$, $\mathbf{Z}$\\
$\mathbf{M_0^\vee}$ & $2\stackrel{\lambda}{\longrightarrow}0$ & $2\stackrel{\eta}{\longrightarrow}1$ & $[ 0, -1, 1 ]$ & $\mathbf{P_1^\vee}$, $\mathbf{P_2}$, ${X}$, $\mathbf{M_2}$, $\mathbf{M_1^\vee}$, $\mathbf{Y}$, $\mathbf{Z^\vee}$\\
$\mathbf{Y}$ & $0\stackrel{\lambda}{\longleftarrow}2\stackrel{\eta}{\longrightarrow}1\stackrel{\beta}{\longleftarrow}0\stackrel{\kappa}{\longrightarrow}2$ & $2\stackrel{\eta}{\longrightarrow}1\stackrel{\beta}{\longleftarrow}0$ & $[ 1, -1, 1 ]$ & $\mathbf{P_0}$, $\mathbf{P_2}$, ${X}$, $\mathbf{M_2}$, ${M_1}$, ${M_1^\vee}$, $\mathbf{M_0^\vee}$, ${Z^\vee}$\\
$\mathbf{Y^\vee}$ & $1$ & $2\stackrel{\delta}{\longleftarrow}1\stackrel{\gamma}{\longrightarrow}0$ & $[ -1, 1, -1 ]$ & $\mathbf{P_0^\vee}$, $\mathbf{P_2^\vee}$, $\mathbf{M_2^\vee}$, $\mathbf{M_0}$\\
$\mathbf{Z}$ & $0\stackrel{\gamma}{\longleftarrow}1\stackrel{\delta}{\longrightarrow}2\stackrel{\kappa}{\longleftarrow}0\stackrel{\beta}{\longrightarrow}1$ & $1\stackrel{\delta}{\longrightarrow}2\stackrel{\kappa}{\longleftarrow}0$ & $[ 1, 1, -1 ]$ & $\mathbf{P_0}$, $\mathbf{P_1}$, ${X}$, ${M_2}$, ${M_2^\vee}$, $\mathbf{M_1}$, $\mathbf{M_0}$, ${Y^\vee}$\\
$\mathbf{Z^\vee}$ & $2$ & $1\stackrel{\eta}{\longleftarrow}2\stackrel{\lambda}{\longrightarrow}0$ & $[ -1, -1, 1 ]$ & $\mathbf{P_0^\vee}$, $\mathbf{P_1^\vee}$, $\mathbf{M_1^\vee}$, $\mathbf{M_0^\vee}$\\
\end{longtable}
}
\newcommand{\WthreeABCD}{\renewcommand\arraystretch{1.5}
\tiny\begin{longtable}{ccccc}
\Cline{1-5}{1.7pt}
$\mathbf{W(3ABCD)}$ & \textbf{Rigid string $C$} & $_PC_P$ & \textbf{$g$-vector} & \textbf{Other $C$-rigid strings} \\\Cline{1-5}{1.3pt} \endhead
\Cline{1-5}{1.7pt} \endfoot
$\mathbf{P_0}$ & $1\stackrel{\gamma}{\longleftarrow}0\stackrel{\delta}{\longrightarrow}2$ & $0$ & $[ 1, 0, 0 ]$ & $\mathbf{P_1}$, $\mathbf{P_2}$, ${X}$, ${X^\vee}$, $\mathbf{M_2}$, ${M_2^\vee}$, $\mathbf{M_1}$, ${M_1^\vee}$\\
$\mathbf{P_1}$ & $1\stackrel{\beta}{\longrightarrow}0$ & $1$ & $[ 0, 1, 0 ]$ & $\mathbf{P_2^\vee}$, $\mathbf{P_0}$, $\mathbf{P_2}$, ${X}$, $\mathbf{X^\vee}$, ${M_2}$, $\mathbf{M_2^\vee}$, $\mathbf{M_1}$, ${M_1^\vee}$\\
$\mathbf{P_2}$ & $2\stackrel{\eta}{\longrightarrow}0$ & $2$ & $[ 0, 0, 1 ]$ & $\mathbf{P_1^\vee}$, $\mathbf{P_0}$, $\mathbf{P_1}$, ${X}$, $\mathbf{X^\vee}$, $\mathbf{M_2}$, ${M_2^\vee}$, ${M_1}$, $\mathbf{M_1^\vee}$\\
$\mathbf{X}$ & $0$ & $2\stackrel{\delta}{\longleftarrow}0\stackrel{\gamma}{\longrightarrow}1$ & $[ 1, -1, -1 ]$ & $\mathbf{P_1^\vee}$, $\mathbf{P_2^\vee}$, $\mathbf{M_2}$, $\mathbf{M_1}$\\
$\mathbf{X^\vee}$ & $1\stackrel{\beta}{\longrightarrow}0\stackrel{\eta}{\longleftarrow}2$ & $1\stackrel{\beta}{\longrightarrow}0\stackrel{\eta}{\longleftarrow}2$ & $[ -1, 1, 1 ]$ & $\mathbf{P_1}$, $\mathbf{P_2}$, $\mathbf{M_2^\vee}$, $\mathbf{M_1^\vee}$\\
$\mathbf{M_2}$ & $0\stackrel{\delta}{\longrightarrow}2$ & $0\stackrel{\gamma}{\longrightarrow}1$ & $[ 1, -1, 0 ]$ & $\mathbf{P_1^\vee}$, $\mathbf{P_0}$, $\mathbf{P_2}$, $\mathbf{X}$, $\mathbf{M_1}$, ${M_1^\vee}$\\
$\mathbf{M_2^\vee}$ & $1$ & $1\stackrel{\beta}{\longrightarrow}0$ & $[ -1, 1, 0 ]$ & $\mathbf{P_0^\vee}$, $\mathbf{P_2^\vee}$, $\mathbf{P_1}$, $\mathbf{X^\vee}$, $\mathbf{M_1^\vee}$\\
$\mathbf{M_1}$ & $0\stackrel{\gamma}{\longrightarrow}1$ & $0\stackrel{\delta}{\longrightarrow}2$ & $[ 1, 0, -1 ]$ & $\mathbf{P_2^\vee}$, $\mathbf{P_0}$, $\mathbf{P_1}$, $\mathbf{X}$, $\mathbf{M_2}$, ${M_2^\vee}$\\
$\mathbf{M_1^\vee}$ & $2$ & $2\stackrel{\eta}{\longrightarrow}0$ & $[ -1, 0, 1 ]$ & $\mathbf{P_0^\vee}$, $\mathbf{P_1^\vee}$, $\mathbf{P_2}$, $\mathbf{X^\vee}$, $\mathbf{M_2^\vee}$\\
\end{longtable}
}
\newcommand{\WthreeQA}{\renewcommand\arraystretch{1.5}
\tiny\begin{longtable}{ccccc}
\Cline{1-5}{1.7pt}
$\mathbf{W(Q(3A)_1)}$ & \textbf{Rigid string $C$} & $_PC_P$ & \textbf{$g$-vector} & \textbf{Other $C$-rigid strings} \\\Cline{1-5}{1.3pt} \endhead
\Cline{1-5}{1.7pt} \endfoot
$\mathbf{P_0}$ & $1\stackrel{\gamma}{\longleftarrow}0\stackrel{\delta}{\longrightarrow}2$ & $0$ & $[ 1, 0, 0 ]$ & $\mathbf{P_1}$, $\mathbf{P_2}$, ${X}$, ${X^\vee}$, $\mathbf{M_2}$, ${M_2^\vee}$, $\mathbf{M_1}$, ${M_1^\vee}$, ${M_0}$, ${M_0^\vee}$\\
$\mathbf{P_1}$ & $1\stackrel{\beta\delta}{\longrightarrow}2$ & $1$ & $[ 0, 1, 0 ]$ & $\mathbf{P_0}$, $\mathbf{P_2}$, ${X}$, $\mathbf{X^\vee}$, ${M_2}$, $\mathbf{M_2^\vee}$, $\mathbf{M_1}$, ${M_1^\vee}$, $\mathbf{M_0}$, ${M_0^\vee}$\\
$\mathbf{P_2}$ & $2\stackrel{\eta\gamma}{\longrightarrow}1$ & $2$ & $[ 0, 0, 1 ]$ & $\mathbf{P_0}$, $\mathbf{P_1}$, ${X}$, $\mathbf{X^\vee}$, $\mathbf{M_2}$, ${M_2^\vee}$, ${M_1}$, $\mathbf{M_1^\vee}$, ${M_0}$, $\mathbf{M_0^\vee}$\\
$\mathbf{X}$ & $0$ & $2\stackrel{\delta}{\longleftarrow}0\stackrel{\gamma}{\longrightarrow}1$ & $[ 1, -1, -1 ]$ & $\mathbf{P_1^\vee}$, $\mathbf{P_2^\vee}$, $\mathbf{M_2}$, $\mathbf{M_1}$\\
$\mathbf{X^\vee}$ & $1\stackrel{\beta}{\longrightarrow}0\stackrel{\eta}{\longleftarrow}2$ & $1\stackrel{\beta}{\longrightarrow}0\stackrel{\eta}{\longleftarrow}2$ & $[ -1, 1, 1 ]$ & $\mathbf{P_1}$, $\mathbf{P_2}$, $\mathbf{M_2^\vee}$, $\mathbf{M_1^\vee}$, ${M_0}$, ${M_0^\vee}$\\
$\mathbf{M_2}$ & $0\stackrel{\delta}{\longrightarrow}2$ & $0\stackrel{\gamma}{\longrightarrow}1$ & $[ 1, -1, 0 ]$ & $\mathbf{P_1^\vee}$, $\mathbf{P_0}$, $\mathbf{P_2}$, $\mathbf{X}$, $\mathbf{M_1}$, ${M_1^\vee}$, $\mathbf{M_0^\vee}$\\
$\mathbf{M_2^\vee}$ & $1$ & $1\stackrel{\beta}{\longrightarrow}0$ & $[ -1, 1, 0 ]$ & $\mathbf{P_0^\vee}$, $\mathbf{P_2^\vee}$, $\mathbf{P_1}$, $\mathbf{X^\vee}$, $\mathbf{M_1^\vee}$, $\mathbf{M_0}$\\
$\mathbf{M_1}$ & $0\stackrel{\gamma}{\longrightarrow}1$ & $0\stackrel{\delta}{\longrightarrow}2$ & $[ 1, 0, -1 ]$ & $\mathbf{P_2^\vee}$, $\mathbf{P_0}$, $\mathbf{P_1}$, $\mathbf{X}$, $\mathbf{M_2}$, ${M_2^\vee}$, $\mathbf{M_0}$\\
$\mathbf{M_1^\vee}$ & $2$ & $2\stackrel{\eta}{\longrightarrow}0$ & $[ -1, 0, 1 ]$ & $\mathbf{P_0^\vee}$, $\mathbf{P_1^\vee}$, $\mathbf{P_2}$, $\mathbf{X^\vee}$, $\mathbf{M_2^\vee}$, $\mathbf{M_0^\vee}$\\
$\mathbf{M_0}$ & $1\stackrel{\beta}{\longrightarrow}0$ & $1\stackrel{\beta\delta}{\longrightarrow}2$ & $[ 0, 1, -1 ]$ & $\mathbf{P_2^\vee}$, $\mathbf{P_1}$, ${X}$, $\mathbf{M_2^\vee}$, $\mathbf{M_1}$\\
$\mathbf{M_0^\vee}$ & $2\stackrel{\eta}{\longrightarrow}0$ & $2\stackrel{\eta\gamma}{\longrightarrow}1$ & $[ 0, -1, 1 ]$ & $\mathbf{P_1^\vee}$, $\mathbf{P_2}$, ${X}$, $\mathbf{M_2}$, $\mathbf{M_1^\vee}$\\
\end{longtable}
}
\newcommand{\WthreeF}{\renewcommand\arraystretch{1.5}
\tiny\begin{longtable}{ccccc}
\Cline{1-5}{1.7pt}
$\mathbf{W(3F)}$ & \textbf{Rigid string $C$} & $_PC_P$ & \textbf{$g$-vector} & \textbf{Other $C$-rigid strings} \\\Cline{1-5}{1.3pt} \endhead
\Cline{1-5}{1.7pt} \endfoot
$\mathbf{P_0}$ & $0\stackrel{\beta\delta}{\longrightarrow}2$ & $0$ & $[ 1, 0, 0 ]$ & $\mathbf{P_1}$, $\mathbf{P_2}$, $\mathbf{M_2}$, ${M_2^\vee}$, $\mathbf{M_1}$, ${M_1^\vee}$, $\mathbf{M_0}$, ${M_0^\vee}$, $\mathbf{Y}$, ${Z^\vee}$\\
$\mathbf{P_1}$ & $1\stackrel{\delta\lambda}{\longrightarrow}0$ & $1$ & $[ 0, 1, 0 ]$ & $\mathbf{P_0}$, $\mathbf{P_2}$, ${M_2}$, $\mathbf{M_2^\vee}$, ${M_1}$, $\mathbf{M_1^\vee}$, $\mathbf{M_0}$, ${M_0^\vee}$, ${Y}$, ${Z^\vee}$\\
$\mathbf{P_2}$ & $0\stackrel{\lambda}{\longleftarrow}2\stackrel{\eta}{\longrightarrow}1$ & $2$ & $[ 0, 0, 1 ]$ & $\mathbf{P_0}$, $\mathbf{P_1}$, ${M_2}$, ${M_2^\vee}$, ${M_1}$, $\mathbf{M_1^\vee}$, ${M_0}$, $\mathbf{M_0^\vee}$, $\mathbf{Y}$, ${Z^\vee}$\\
$\mathbf{M_2}$ & $0$ & $0\stackrel{\beta}{\longrightarrow}1$ & $[ 1, -1, 0 ]$ & $\mathbf{P_1^\vee}$, $\mathbf{P_2^\vee}$, $\mathbf{P_0}$, $\mathbf{M_1}$, $\mathbf{M_0^\vee}$, $\mathbf{Y}$, ${Z^\vee}$\\
$\mathbf{M_2^\vee}$ & $1\stackrel{\delta}{\longrightarrow}2$ & $1\stackrel{\delta\lambda}{\longrightarrow}0$ & $[ -1, 1, 0 ]$ & $\mathbf{P_0^\vee}$, $\mathbf{P_1}$, $\mathbf{M_1^\vee}$, $\mathbf{M_0}$, ${Z^\vee}$\\
$\mathbf{M_1}$ & $0\stackrel{\beta}{\longrightarrow}1$ & $0\stackrel{\beta\delta}{\longrightarrow}2$ & $[ 1, 0, -1 ]$ & $\mathbf{P_2^\vee}$, $\mathbf{P_0}$, $\mathbf{M_2}$, $\mathbf{M_0}$\\
$\mathbf{M_1^\vee}$ & $2\stackrel{\eta}{\longrightarrow}1$ & $2\stackrel{\lambda}{\longrightarrow}0$ & $[ -1, 0, 1 ]$ & $\mathbf{P_0^\vee}$, $\mathbf{P_1}$, $\mathbf{P_2}$, $\mathbf{M_2^\vee}$, ${M_0}$, $\mathbf{M_0^\vee}$, $\mathbf{Z^\vee}$\\
$\mathbf{M_0}$ & $1$ & $1\stackrel{\delta}{\longrightarrow}2$ & $[ 0, 1, -1 ]$ & $\mathbf{P_0^\vee}$, $\mathbf{P_2^\vee}$, $\mathbf{P_0}$, $\mathbf{P_1}$, ${M_2}$, $\mathbf{M_2^\vee}$, $\mathbf{M_1}$\\
$\mathbf{M_0^\vee}$ & $2\stackrel{\lambda}{\longrightarrow}0$ & $2\stackrel{\eta}{\longrightarrow}1$ & $[ 0, -1, 1 ]$ & $\mathbf{P_1^\vee}$, $\mathbf{P_2}$, $\mathbf{M_2}$, $\mathbf{M_1^\vee}$, $\mathbf{Y}$, $\mathbf{Z^\vee}$\\
$\mathbf{Y}$ & $0\stackrel{\beta}{\longrightarrow}1\stackrel{\eta}{\longleftarrow}2\stackrel{\lambda}{\longrightarrow}0$ & $0\stackrel{\beta}{\longrightarrow}1\stackrel{\eta}{\longleftarrow}2$ & $[ 1, -1, 1 ]$ & $\mathbf{P_0}$, $\mathbf{P_2}$, $\mathbf{M_2}$, ${M_1}$, ${M_1^\vee}$, $\mathbf{M_0^\vee}$, ${Z^\vee}$\\
$\mathbf{Z^\vee}$ & $2$ & $1\stackrel{\eta}{\longleftarrow}2\stackrel{\lambda}{\longrightarrow}0$ & $[ -1, -1, 1 ]$ & $\mathbf{P_0^\vee}$, $\mathbf{P_1^\vee}$, $\mathbf{M_1^\vee}$, $\mathbf{M_0^\vee}$\\
\end{longtable}
}
\newcommand{\WthreeQLR}{\renewcommand\arraystretch{1.5}
\tiny\begin{longtable}{ccccc}
\Cline{1-5}{1.7pt}
$\mathbf{W(3QLR)}$ & \textbf{Rigid string $C$} & $_PC_P$ & \textbf{$g$-vector} & \textbf{Other $C$-rigid strings} \\\Cline{1-5}{1.3pt} \endhead
\Cline{1-5}{1.7pt} \endfoot
$\mathbf{P_0}$ & $0\stackrel{\beta\delta}{\longrightarrow}2$ & $0$ & $[ 1, 0, 0 ]$ & $\mathbf{P_1}$, $\mathbf{P_2}$, $\mathbf{M_2}$, ${M_2^\vee}$, $\mathbf{M_1}$, ${M_1^\vee}$, $\mathbf{M_0}$, ${M_0^\vee}$\\
$\mathbf{P_1}$ & $1\stackrel{\delta\lambda}{\longrightarrow}0$ & $1$ & $[ 0, 1, 0 ]$ & $\mathbf{P_0}$, $\mathbf{P_2}$, ${M_2}$, $\mathbf{M_2^\vee}$, ${M_1}$, $\mathbf{M_1^\vee}$, $\mathbf{M_0}$, ${M_0^\vee}$\\
$\mathbf{P_2}$ & $2\stackrel{\lambda\beta}{\longrightarrow}1$ & $2$ & $[ 0, 0, 1 ]$ & $\mathbf{P_0}$, $\mathbf{P_1}$, $\mathbf{M_2}$, ${M_2^\vee}$, ${M_1}$, $\mathbf{M_1^\vee}$, ${M_0}$, $\mathbf{M_0^\vee}$\\
$\mathbf{M_2}$ & $0$ & $0\stackrel{\beta}{\longrightarrow}1$ & $[ 1, -1, 0 ]$ & $\mathbf{P_1^\vee}$, $\mathbf{P_2^\vee}$, $\mathbf{P_0}$, $\mathbf{P_2}$, $\mathbf{M_1}$, ${M_1^\vee}$, $\mathbf{M_0^\vee}$\\
$\mathbf{M_2^\vee}$ & $1\stackrel{\delta}{\longrightarrow}2$ & $1\stackrel{\delta\lambda}{\longrightarrow}0$ & $[ -1, 1, 0 ]$ & $\mathbf{P_0^\vee}$, $\mathbf{P_1}$, $\mathbf{M_1^\vee}$, $\mathbf{M_0}$\\
$\mathbf{M_1}$ & $0\stackrel{\beta}{\longrightarrow}1$ & $0\stackrel{\beta\delta}{\longrightarrow}2$ & $[ 1, 0, -1 ]$ & $\mathbf{P_2^\vee}$, $\mathbf{P_0}$, $\mathbf{M_2}$, $\mathbf{M_0}$\\
$\mathbf{M_1^\vee}$ & $2$ & $2\stackrel{\lambda}{\longrightarrow}0$ & $[ -1, 0, 1 ]$ & $\mathbf{P_0^\vee}$, $\mathbf{P_1^\vee}$, $\mathbf{P_1}$, $\mathbf{P_2}$, $\mathbf{M_2^\vee}$, ${M_0}$, $\mathbf{M_0^\vee}$\\
$\mathbf{M_0}$ & $1$ & $1\stackrel{\delta}{\longrightarrow}2$ & $[ 0, 1, -1 ]$ & $\mathbf{P_0^\vee}$, $\mathbf{P_2^\vee}$, $\mathbf{P_0}$, $\mathbf{P_1}$, ${M_2}$, $\mathbf{M_2^\vee}$, $\mathbf{M_1}$\\
$\mathbf{M_0^\vee}$ & $2\stackrel{\lambda}{\longrightarrow}0$ & $2\stackrel{\lambda\beta}{\longrightarrow}1$ & $[ 0, -1, 1 ]$ & $\mathbf{P_1^\vee}$, $\mathbf{P_2}$, $\mathbf{M_2}$, $\mathbf{M_1^\vee}$\\
\end{longtable}
}